\documentclass[11pt]{amsart}

\usepackage[utf8]{inputenc}
\usepackage[T1]{fontenc}
\DeclareSymbolFont{rsfscript}{OMS}{rsfs}{m}{b}
\DeclareSymbolFontAlphabet{\mathrsfs}{rsfscript}
\usepackage{amsfonts,enumerate,amsmath,amssymb,amsthm}
\usepackage{pgf,pstricks}
\usepackage{colortbl}
\usepackage[absolute]{textpos}
\usepackage{stmaryrd}
\usepackage{subcaption}
\usepackage{dsfont}
\usepackage{float}
\usepackage{pstricks,tikz}
\usepackage{graphicx}
\usepackage{pst-grad}
\usepackage{multido}
\usepackage{tikz,tikzrput}
\usetikzlibrary{shapes,snakes}
\usetikzlibrary{positioning}
\usepackage{float}
\usepackage{ytableau}

\usepackage{geometry}
\newtheorem{theorem}{Theorem}[section]
\newtheorem{corollary}[theorem]{Corollary}

\newtheorem{lemma}[theorem]{Lemma}
\newtheorem{proposition}[theorem]{Proposition}

\theoremstyle{definition}
\newtheorem{definition}[theorem]{Definition}
\newtheorem{example}[theorem]{Example}
\newtheorem{remark}[theorem]{Remark}

\newcommand{\add}{\raisebox{.5mm}{\hspace{.7mm}\scalebox{.5}{$\bullet$}\hspace{.7mm}}}

\newcommand{\ti}{\tilde}

\newcommand{\al}{\alpha}
\newcommand{\om}{\omega}
\newcommand{\ga}{\gamma}
\newcommand{\la}{\lambda}
\newcommand{\eps}{\varepsilon}
\newcommand{\si}{\sigma}

\newcommand{\La}{\Lambda}

\newcommand{\Ga}{\Gamma}

\newcommand{\ov}[1]{\overline{#1}}
\newcommand{\os}[2]{\overset{#1}{#2}}
\newcommand{\scb}[2]{\scalebox{#1}{#2}}

\newcommand{\crossp}{\overset{+}{\leadsto}}

\newcommand{\fh}{\mathfrak{h}}
\newcommand{\fg}{\mathfrak{g}}
\newcommand{\cC}{\mathcal{C}}
\newcommand{\cF}{\mathcal{F}}
\newcommand{\cU}{\mathcal{U}}

\newcommand{\cA}{\mathcal{A}}
\renewcommand{\sb}{\mathsf{b}}
\newcommand{\sq}{\mathsf{q}}

\newcommand{\scal}[2]{\langle #1,#2\rangle}
\newcommand{\scald}[2]{(#1,#2)}
\newcommand{\qu}[1]{\quad\text{#1}\quad}

\newcommand{\tbb}{{\tiny $\bullet$}}

\newcommand{\wt}{\mathsf{wt}}

\newcommand{\w}{\mathsf{w}}
\renewcommand{\sf}{\mathsf{f}}

\newcommand{\Z}{\mathbb{Z}}
\newcommand{\N}{\mathbb{N}}

\newcommand{\R}{\mathbb{R}}

\newcommand{\Q}{\mathbb{Q}}

\newcommand{\sz}{\mathsf{z}}

\newcommand{\sB}{\mathsf{B}}
\newcommand{\sZ}{\mathsf{Z}}

\newcommand{\wB}{\widehat{B}}

\newcommand{\sLa}{\mathsf{\La}}
\newcommand{\cAJ}{\mathcal{A}_J}

\newcommand{\func}[1]{\operatorname{#1}}

\newcommand{\lra}{\longrightarrow}

 \newcommand{\cI}{\mathcal{I}}

\definecolor{purple}{rgb}{0.8,0.12,0.8}
\definecolor{orange}{rgb}{1.0,0.7,0.0}
\definecolor{pink}{rgb}{1,0.5,0.8}
\definecolor{blackg}{rgb}{0.1,0.25,0.1}
\definecolor{ForestGreen}{cmyk}{0.91,0,0.88,0.42}
\definecolor{Turquoise}{cmyk}{0.85,0,0.20,0}

\definecolor{GreenYellow}{cmyk}{0.15,0,0.69,0} 
\definecolor{Yellow}{cmyk}{0,0,1.,0} 
\definecolor{Goldenrod}{cmyk}{0,0.10,0.84,0} 
\definecolor{Dandelion}{cmyk}{0,0.29,0.84,0} 
\definecolor{Apricot}{cmyk}{0,0.32,0.52,0} 
\definecolor{Peach}{cmyk}{0,0.50,0.70,0} 
\definecolor{Melon}{cmyk}{0,0.46,0.50,0} 
\definecolor{YellowOrange}{cmyk}{0,0.42,1.,0} 
\definecolor{Orange}{cmyk}{0,0.61,0.87,0} 
\definecolor{BurntOrange}{cmyk}{0,0.51,1.,0} 
\definecolor{Bittersweet}{cmyk}{0,0.75,1.,0.24} 
\definecolor{RedOrange}{cmyk}{0,0.77,0.87,0} 
\definecolor{Mahogany}{cmyk}{0,0.85,0.87,0.35} 
\definecolor{Maroon}{cmyk}{0,0.87,0.68,0.32} 
\definecolor{BrickRed}{cmyk}{0,0.89,0.94,0.28} 
\definecolor{Red}{cmyk}{0,1.,1.,0} 
\definecolor{OrangeRed}{cmyk}{0,1.,0.50,0} 
\definecolor{RubineRed}{cmyk}{0,1.,0.13,0} 
\definecolor{WildStrawberry}{cmyk}{0,0.96,0.39,0} 
\definecolor{Salmon}{cmyk}{0,0.53,0.38,0} 
\definecolor{CarnationPink}{cmyk}{0,0.63,0,0} 
\definecolor{Magenta}{cmyk}{0,1.,0,0} 
\definecolor{VioletRed}{cmyk}{0,0.81,0,0} 
\definecolor{Rhodamine}{cmyk}{0,0.82,0,0} 
\definecolor{Mulberry}{cmyk}{0.34,0.90,0,0.02} 
\definecolor{RedViolet}{cmyk}{0.07,0.90,0,0.34} 
\definecolor{Fuchsia}{cmyk}{0.47,0.91,0,0.08} 
\definecolor{Lavender}{cmyk}{0,0.48,0,0} 
\definecolor{Thistle}{cmyk}{0.12,0.59,0,0} 
\definecolor{Orchid}{cmyk}{0.32,0.64,0,0} 
\definecolor{DarkOrchid}{cmyk}{0.40,0.80,0.20,0} 
\definecolor{Purple}{cmyk}{0.45,0.86,0,0} 
\definecolor{Plum}{cmyk}{0.50,1.,0,0} 
\definecolor{Violet}{cmyk}{0.79,0.88,0,0} 
\definecolor{RoyalPurple}{cmyk}{0.75,0.90,0,0} 
\definecolor{BlueViolet}{cmyk}{0.86,0.91,0,0.04} 
\definecolor{Periwinkle}{cmyk}{0.57,0.55,0,0} 
\definecolor{CadetBlue}{cmyk}{0.62,0.57,0.23,0} 
\definecolor{CornflowerBlue}{cmyk}{0.65,0.13,0,0} 
\definecolor{MidnightBlue}{cmyk}{0.98,0.13,0,0.43} 
\definecolor{NavyBlue}{cmyk}{0.94,0.54,0,0} 
\definecolor{RoyalBlue}{cmyk}{1.,0.50,0,0} 
\definecolor{Blue}{cmyk}{1.,1.,0,0} 
\definecolor{Cerulean}{cmyk}{0.94,0.11,0,0} 
\definecolor{Cyan}{cmyk}{1.,0,0,0} 
\definecolor{ProcessBlue}{cmyk}{0.96,0,0,0} 
\definecolor{SkyBlue}{cmyk}{0.62,0,0.12,0} 
\definecolor{Turquoise}{cmyk}{0.85,0,0.20,0} 
\definecolor{TealBlue}{cmyk}{0.86,0,0.34,0.02} 
\definecolor{Aquamarine}{cmyk}{0.82,0,0.30,0} 
\definecolor{BlueGreen}{cmyk}{0.85,0,0.33,0} 
\definecolor{Emerald}{cmyk}{1.,0,0.50,0} 
\definecolor{JungleGreen}{cmyk}{0.99,0,0.52,0} 
\definecolor{SeaGreen}{cmyk}{0.69,0,0.50,0} 
\definecolor{Green}{cmyk}{1.,0,1.,0} 
\definecolor{ForestGreen}{cmyk}{0.91,0,0.88,0.12} 
\definecolor{PineGreen}{cmyk}{0.92,0,0.59,0.25} 
\definecolor{LimeGreen}{cmyk}{0.50,0,1.,0} 
\definecolor{YellowGreen}{cmyk}{0.44,0,0.74,0} 
\definecolor{SpringGreen}{cmyk}{0.26,0,0.76,0} 
\definecolor{OliveGreen}{cmyk}{0.64,0,0.95,0.40} 
\definecolor{RawSienna }{cmyk}{0,0.72,1.,0.45} 
\definecolor{Sepia}{cmyk}{0,0.83,1.,0.70} 
\definecolor{Brown}{cmyk}{0,0.81,1.,0.60} 
\definecolor{Tan}{cmyk}{0.14,0.42,0.56,0} 
\definecolor{Gray}{cmyk}{0,0,0,0.50} 
\definecolor{Black}{cmyk}{0,0,0,1.} 
\definecolor{White}{cmyk}{0,0,0,0} 

\newcommand\blfootnote[1]{%
  \begingroup
  \renewcommand\thefootnote{}\footnote{#1}%
  \addtocounter{footnote}{-1}%
  \endgroup
}

\newcommand{\syc}{\mathsf{s_C}}

\begin{document}

\title{Homology rings of affine grassmannians and positively multiplicative graphs}
\author{J\'{e}r\'{e}mie Guilhot, C\'{e}dric Lecouvey and Pierre Tarrago}

\begin{abstract} 
Let $\fg$ be an untwisted affine Lie algebra with associated Weyl group $W_a$. To any level 0 weight~$\ga$ we associate a weighted graph $\Ga_\ga$ that encodes the orbit of $\ga$ under the action $W_a$. We show that the graph~$\Ga_\ga$ encodes the periodic orientation of certain subsets of alcoves in $W_a$ and therefore can be interpreted as an automaton determining the reduced expressions in these subsets. Then, by using some relevant quotients of the homology ring of affine Grassmannians, we show that $\Ga_\ga$ is positively multiplicative. This allows us in particular to compute the structure constants of the homology rings using elementary linear algebra on multiplicative graphs. In another direction, the positivity of $\Ga_\ga$ yields the key ingredients to study a large class
of central random walks on alcoves.\end{abstract}

\blfootnote{{\it 2020 Mathematics Subject Classification. 05,15,16,20\smallskip}}

\maketitle

\blfootnote{The three authors are supported by the Agence Nationale de la Recherche funding ANR CORTIPOM 21-CE40-001. 
The first author is supported by the Agence Nationale de la Recherche funding ANR JCJC Project ANR-18-CE40-0001 and by the Australian Research Council discovery project DP200100712.\smallskip}

\section{Introduction}

The notion of positively multiplicative graphs naturally comes from the
study of harmonic functions on infinite graphs and random walks on simple
geometric configurations. As defined in \cite{Ker} for infinite graded
graphs and in \cite{GLT1} in general, a graph $\Gamma $ is positively
multiplicative if there exist a commutative algebra~$A$ together with an
element $x\in A$ and a basis $\mathsf{B}$ such that the structure constants
of~$A$ with respect to~$\mathsf{B}$ are nonnegative and the matrix of the
multiplication by $x$ expressed in the basis $\mathsf{B}$ is the adjacency
matrix of $\Gamma $. This is  the case for example for the well-known Young
lattice of partitions with $A$ the algebra of symmetric functions, $\mathsf{B%
}$ its basis of Schur functions and $x=s_{(1)}$ the Schur function
associated to the partition $(1)$. As explained in \cite{GLT1}, multiplicative graphs give a
convenient tool to study various problems connected to representation theory
of finite groups (character theory), representation theory of Lie algebras
and their associated random walks, harmonic functions and Markov chains on
infinite graded graphs, automatons recognizing reduced expressions in
Coxeter groups, orbits in affine crystals etc. In \cite{LT2} it is shown
how the study of a natural class of random walks on the alcove tesselation
of type $A$ can be reduced to the determination of the Perron-Frobenius
elements for a suitable positively multiplicative graph defined on $k$%
-irreducible partitions. One of the goal of this paper is to extend some
results of \cite{LT2} by showing that a much larger class of graphs defined
from general affine Lie algebras are positively multiplicative. 

\medskip

More precisely, let $\mathfrak{g}$ be a non twisted affine Lie
algebra with associated affine Weyl group $W_{a}$. To any level $0$ weight $\gamma $, we associate an oriented weighted graph $\Gamma _{\gamma }$ that
encodes the orbit of $\gamma $ under the action of $W_{a}$. In this
paper, we establish that the graphs $\Gamma _{\gamma }$ are positively
multplicative. This will be achieved by showing that the orientation of $\Gamma _{\gamma }$ actually describes the periodic orientation between alcoves lying
in a subset of $W_{a}$ (subset that depends on $\gamma $) and that the
adjacency matrix of $\Gamma _{\gamma }$ is the matrix of the multiplication
by the distinguished element $\xi _{s_{0}}$ in some quotient of the homology
ring of affine Grassmannians (quotient that also depends on~$\gamma $). In type $A
$ and when $\gamma =\rho $ (the sum of all classical fundamental weights~of~$\mathfrak{g}$), the graph $\Ga_\rho$ is the positively multiplicative graph
on $k$-irreducible partitions studied in \cite{LT2}. With this positivity property in hand, it becomes possible to generalize all the results of~\cite{LT2} by studying in each case walks on the relevant subset of alcoves.

\medskip

The affine Weyl group $W_a$ associated to $\mathfrak{g}$ is generated by the
simple reflections $s_0,\ldots,s_n$ with respect to the roots $\alpha_0,\ldots,\alpha_n$. The weight lattice of $\mathfrak{g}$ is $P_a=\oplus \mathbb{Z}\Lambda_i \oplus\mathbb{Z}\delta$ where $\Lambda_0,\ldots,\Lambda_n$ are the fundamental weights and $\delta$ is the
imaginary root. Let $v^\vee = (a_0^\vee,\ldots,a_n^\vee)\in \mathbb{Z}^n$
and $v=(a_0,\ldots,a_n)\in\mathbb{Z}^n$ be the unique vectors with
relatively prime coordinates such that $v^\vee\cdot A = A\cdot {^t}v = 0$
where $A$ is the Cartan matrix of~$\mathfrak{g}$. Let $\omega_i=\Lambda_i -
a_i^\vee \Lambda_0$ for $1\leq i \leq n$. Then the $\omega_i$'s generate the
weight lattice $P$ of the finite type Cartan submatrix associated to $A$
which can be obtained by erasing the row and the column corresponding to the
root~$\alpha_0$. We denote by $W$ the finite Weyl group associated to this
submatrix, it is generated by the reflections $s_1,\ldots,s_n$. Any weight
in~$P_a$ can be written under the form $\ell \Lambda_0+\lambda+k\delta$
where $\lambda\in P$ and $k,\ell\in \mathbb{Z}$. The integer $\ell$ is the
level of the weight and the level $0$ weight lattice is thus $P_0 = P\oplus\mathbb{Z}\delta$. It is stable under the action of $W_a$. In this paper, we
are interested in the orbit of the level 0 weights under the action of~$W_a$. First, since $\delta$ is a fixed point for this action, it is enough to
look at the orbit of the elements of $P$. Next, for all $\gamma\in P$, one
can show that~$s_i(\gamma)\in P$ if $i\neq 0$ and that $s_0(\gamma) =
s_\theta(\gamma) + (\gamma,\beta)\delta$ for some $\beta\in P$. Here~$s_\theta$ is the reflection with respect to the root $\theta = \delta -
\alpha_0$. We can now define the oriented weighted graph $\Gamma_\gamma$ for
all $\gamma\in P$. The weights of $\Gamma_\gamma$ lie in the group algebra $\mathbb{Z}[P]=\langle \mathsf{z}^\beta\mid \beta\in P\rangle_\mathbb{Z}$ and
the orientation of $\Gamma_\gamma$ is determined by the usual partial order
on $P$, that is $\lambda^{\prime }\prec \lambda$ if $\lambda-\lambda^{\prime
}$ is a sum of positive roots. We have

\begin{enumerate}
	\item[{\protect\tiny $\bullet$}] the vertices of $\Gamma_\gamma$ are the
	elements of the orbit of $\gamma$ under the action of $W$;

	\item[{\protect\tiny $\bullet$}] there is an arrow of type $i\neq 0$ and
	weight $a_i\in \mathbb{Z}$ from $\lambda$ to $\lambda^{\prime }$ when $%
	s_i(\lambda) = \lambda^{\prime }$ and $\lambda^{\prime }\prec \lambda$;

	\item[{\protect\tiny $\bullet$}] there is an arrow of type $0$ and weight $a_0\mathsf{z}^{\beta}$ from $\lambda$ to $\lambda^{\prime }$ when $s_0(\lambda) = s_\theta(\lambda) +(\lambda,\beta)\delta$, $s_\theta(\lambda)=\lambda^{\prime }$ and $\lambda\prec\lambda^{\prime }$.
\end{enumerate}
The affine Weyl group $W_a$ is a semi-direct product $W\ltimes Q^\vee$ where 
$Q^\vee$ is the $\mathbb{Z}$-lattice generated by the $W$-orbit of $\theta$.
Any element in $W_a$ can be uniquely written as $ut_\beta$ where $(u,\beta)\in W\times Q^\vee$ and the graph~$\Gamma_\gamma$ can be defined
only in terms of $W_a$. For instance, in the case where $\gamma=\rho$, the
sum of all classical fundamental weights $\omega_1+\cdots +\omega_n$, the
stabiliser of $\rho$ in $W$ is reduced to the identity and the orbit of $\gamma$ under the action of $W$ is in bijection with $W$. There is an arrow
of type $i\neq 0$ and weight $a_i$ in $\Gamma_\rho$ from~$w$ to $w^{\prime }$
whenever $w^{\prime }=s_iw$ and~$s_iw>w$ in the (weak) Bruhat order and
there is an arrow of type $0$ and weight~$a_0\mathsf{z}^\beta$ between $w$
and $w^{\prime }$ whenever~$s_0w = w^{\prime }t_{\beta}$, $w^{\prime }\in W$
and $w^{\prime }<w$.

\medskip

The graph $\Gamma_\gamma$ is defined using the level $0$ action of $W_a$ on $P$. It does however contains a lot of information on the geometry of alcoves. It
is well-known that $W_a$ can be realised as a group of affine motions in an
affine Euclidean space $V$ generated by affine reflections. More precisely,
there exists a finite root system~$\Phi$ with dual root system $\Phi^\vee$
such that $W_a$ is generated by all reflections $\sigma_{\alpha,k}$ with
respect to the affine hyperplane $H_{\alpha,k} = \{x\in V\mid (
x,\alpha) = k\}$ where $\alpha\in \Phi$ and $k\in \mathbb{Z}$. Again,
we get that $W_a=W\ltimes Q^\vee$ where~$Q^\vee$ is the coroot lattice
generated by $\Phi^\vee$. The set of alcoves is the set of connected
components of the set $V\setminus \cup_{\alpha\in \Phi,k\in \mathbb{Z}}\{H_{\alpha,k}\}$. The group $W_a$ acts simply transitively on this set,
hence, fixing the fundamental alcove $A_0$, there is a one to one
correspondance between alcoves and elements of $W_a$. In the case where~$\gamma=\rho$, the graph $\Gamma_\rho$ encodes the periodic orientation of adjacent
alcoves in the following way. Let $A$ and $A^{\prime }$ be two adjacent
alcoves corresponding respectively to $ut_\beta$ and $u^{\prime
}t_\beta^{\prime }$ in $W_a$ and let $H_{\alpha,k}$ with~$\alpha\in \Phi^+$
(the set of positive roots) and $k\in \mathbb{Z}$ be the unique hyperplane
separating $A$ and $A^{\prime }$. Then there is an arrow from $u$ to $u^{\prime }$ in the graph $\Gamma_\rho$ if and only if $A$ lies on the
positive side of $H_{\alpha,k}$ i.e the half-plane caracterised by $(x,\alpha)>k$ and $A^{\prime }$ on the negative side. In the case where 
$\gamma$ has a non trivial stabiliser $W_J$ in $W$ (here $W_J$ is a standard
parabolic subgroup of $W$ generated by $\{s_j\mid j\in J\}$), the graph~$\Gamma_\gamma$ encodes the orientation of adjacent alcoves in the so-called fundamental $J$-alcove $\cAJ$ which is defined as the set of points that lie between
hyperplanes $H_{\beta,0}$ and $H_{\beta,1}$ for all positive roots $\beta$
in $\Phi_J$, the parabolic root system of~$W_J$.

\medskip

It is interesting to note that since the graph $\Gamma _{\rho }$ naturally
encodes the orientation of adjacents alcoves, it can be interpreted as an
automaton for reduced expressions of affine Grassmanians elements. Indeed, the set $W_{a}^{\Lambda _{0}}$ of affine Grassmannians elements is the set of elements of
minimal length in the cosets of $W_{a}$ with respect to the finite Weyl
group $W$. This set can also be caracterised as the set of alcoves that lies
in the fundamental chamber $\mathcal{C}_{0}=\{x\in V\mid \langle x,\alpha
\rangle \geq 0,\ \forall \alpha \in \Phi ^{+}\}$. It follows that if $w\in
W_{a}^{\Lambda _{0}}$ and $s\in S$ are such that $sw,w\in W_{a}^{\Lambda_{0}}$ then we will have we have $sw>w$ if and only if $sw$ lies on the
positive side of the unique hyperplane separating $w$ and $sw$. As a
consequence, reduced expression of elements in $W_{a}^{\Lambda _{0}}$ will
be in bjection with certain paths in the graph $\Gamma _{\rho }$.
In the case where $\gamma $ has a non trivial stabiliser $W_{J}$, we obtain
an automaton for affine Grassmanians lying in the fundamental $J$-alcove $\cAJ$. 

\medskip

As mentionned before, we show that the graph $\Gamma_\gamma$ is positively
multiplicative for any level 0 weight $\gamma$. The main tool to prove this
result is the homology ring of affine Grassmannians. We refer to
Section \ref{Sec_Homolo_Ring} for precise definitions and some references.
The homology ring of affine Grassmanians $\sLa$ is a commutative $\mathbb{Q}$-algebra with a distinguished basis $\{\xi_w\mid w\in W_a^{\Lambda_0}\}$
such that the structure constants with respect to this basis are positive
integers. In type $A$, this ring can be identified with a subalgebra of the
symmetric functions and the previous distinguished basis then coincides with
the basis of $k$-Schur functions (see \cite{LLMSSZ}). This positive
structure is the key to our result. The connection between $\sLa$ and our
graph is made through the Pieri rules that describe in particular the
multiplication by $\xi_{s_0}$ in $\sLa$: 
\begin{equation*}
\xi_{s_0}\xi_{w} = \sum a_i \xi_{s_iw}\quad\text{for all $w\in
	W_a^{\Lambda_0}$}\quad
\end{equation*}
where the sum is over all $i\in \{0,\ldots,n\}$ such that $s_iw>w$ and $s_iw\in W_{a}^{\Lambda_{0}}$. It turns out that the algebra~$\sLa$ has very nice factorisation properties: for all $w\in W_a^{\Lambda_0}$ and all $\omega_i$, we have $\xi_{wt_{\omega_i}} = \xi_{w}\xi_{t_{\omega_i}}$. This
property can be used to show that $\sLa$ is finite-dimensional over the
ring of polyomials $\mathbb{Q}[\xi_{t_{\omega_1}},\ldots,\xi_{t_{\omega_n}}]$. In turn, this allows us to extend the scalars in the definition of $\sLa$
to the ring of Laurent polynomials and to define a basis in bijection with
the full set of alcoves, not only the affine Grassmanians elements. We then show that
the adjacency matrix of $\Gamma_\rho$ is essentially equal to the matrix of the
multiplication by $\xi_{s_0}$ in this extended algebra expressed in a well chosen
basis, hence showing the positivity. In the more general case where $\gamma$ has a non trivial stabiliser in $W$, we introduce a
quotient of some ideal of our extended algebra that has a basis in bijection with the set
of alcoves in $\cAJ$ and show that there are extra factorisation
properties by "pseudo-translations" in this quotient. We then interpret the
adjacency matrix of $\Gamma_\gamma$ as the matrix of the multiplication by $\tilde{\xi}_{s_0}$, the image of $\xi_{s_0}$ in the quotient, hence showing
the positivity of~$\Gamma_\gamma$.  It is worth noticing here that one can apply to $\Gamma _{\rho}$
the general procedure for computing positively multiplicative bases obtained
in \cite{GLT1}. This gives a simple way to compute the structure constants of
the homology ring of affine Grassmannians, at least in small ranks. We have done this in type $G_2$ \cite{G2}. 

\medskip

As previously mentionned, the fact that $\Gamma _{\gamma }$ is positively
multiplicative has important consequences in the theory of random walks on
alcoves that never cross twice the same hyperplane. We can
expand the graph $\Gamma _{\gamma }$ to obtain an infinite graph $\Gamma
_{\gamma }^{e}$ which is the graph of the weak Bruhat order on the set of affine
Grassmanians that lie in the $J$-alcove $\cAJ$ and study Markov chains
on this graph. The case where $\gamma =\rho $ and each new alcove (or each
new vertex) is chosen uniformly among all the possible alcoves (vertices)
was considered in \cite{Lam2}. More in the spirit of the works by Kerov and
Vershik \cite{Ker}, one can also study central random walks on alcoves or
central Markov chains on $\Gamma _{\gamma }^{e}$. This is another class of random
walks that never cross twice the same hyperplane for which two trajectories
with the same end points have the same probability. These are controlled by the positive harmonic functions on $\Gamma
_{\gamma }^{e}$ (see \cite{LT2} for the case $\gamma =\rho $ in type $A$).
Then, thanks to the positivity of $\Gamma _{\gamma }$, the determination of
these positive harmonic functions is completely explicit. This is a
particular case of a more general phenomenon for positively multiplicative
graphs detailed in \cite{GLT1}. It then becomes possible to obtain
a rational expression for the drift of the previous central random walks
(i.e. for all types and for all level $0$ weight $\gamma $) from the methods
and results obtained in \cite{LT2} for the affine Grassmanian in type $A$.
We will not pursue in this direction in the present paper where we rather
focus on the algebraic properties of the graphs $\Gamma _{\gamma }$. 

\medskip

In classical types, the case $\Gamma _{\omega _{i}}$ where $\omega _{i}$ is a fundamental
classical weight is also of special interest since it decribes the
transitions of $k$ particles moving with small steps on a
circle (in type $A$) or on a segment (in type $B,C,D$) with $n$ sites  without intersecting each other. This problem in type $A$ is studied
in details in \cite{GLT3} using the fact the eigenvalues and the
eigenvectors of the adjacency matrix of $\Gamma _{\omega _{i}}$ are known. A precise asymptotic behavior of the random walk on the graph is obtained and asymptotic formulas are deduced for the coefficients of the quantum cohomology of Grassmanniann.
It would be interesting to study the spectral properties of the
adjacency matrices of the graphs $\Gamma _{\gamma }$ in general.

\medskip

Finally, the graph $\Gamma _{\gamma }$ also carries information on finite
affine crystals. To any fundamental classical weight~$\omega _{i}\in P$,
there exists a finite-dimensional representation of the quantum group $U_{q}^{\prime }(\mathfrak{g})$ called a Kirillov-Reshetikhin module. In many
cases (in particular for the classical affine root systems), this module is
known to admit a combinatorial skeleton $\widehat{B}(\omega _{i})$, its
so-called associated crystal (see \cite{BumpSchilling2017} for a complete
review on crystals). There is an action of the affine Weyl group $W_{a}$ on $\widehat{B}(\omega _{i})$ defined using its $i$-chains structure: for $v\in 
\widehat{B}(\omega _{i})$, the image of $v$ under the action of $s_{i}$ is
the symmetric of~$v$ in the $i$-chain containing $v$. The crystal $\widehat{B}(\omega _{i})$ contains a unique vertex $v_{0}$ of weight $\omega _{i}$ and
the orbit of~$v_{0}$ under the action of $W_{a}$ on this crystal is the same
as the orbit of $\omega _{i}$ under the action of $W_{a}$ in the following
sense: the weight of the vertex $w\cdot v_{0}$ is $w(\omega _{i})$. Thus the
graph $\Gamma _{\omega _{i}}$ also encodes the action of $W_{a}$ on $v_{0}$.
In type $A$, since all the $i$-chains in $\widehat{B}(\omega _{i})$ have
length at most 1, we actually so recover the full crystal~$\widehat{B}(\omega _{i})$. More generally, one can associate an affine crystal $\widehat{B}(\gamma )$ to any dominant classical weight $\gamma =\sum \gamma
_{i}\omega _{i}$ using tensor products of crystals by setting 
\begin{equation*}
\widehat{B}(\gamma )=\widehat{B}(\omega _{1})^{\otimes \gamma _{1}}\otimes
\cdots \otimes \widehat{B}(\omega _{n})^{\otimes \gamma _{n}}.
\end{equation*}%
As before, there is an action of $W_{a}$ on $\widehat{B}(\gamma )$ and the
orbit of the unique vertex of weight $\gamma $ in $\widehat{B}(\gamma )$
under this action is the same as the orbit of $\gamma $ under the action of $W_{a}$. Hence the graph $\Gamma _{\gamma }$ also encodes the action of the
affine Weyl group on $\widehat{B}(\gamma )$. Here again, it is a natural
question to ask for the affine crystals~$\widehat{B}(\gamma )$ that could be
positively multiplicative when they do not reduce to one orbit. In connection with the alcove model for
KR-crystals developed in \cite{Len1,Len2} it would yield interesting
generalisations of the previous alcove random walks.

\medskip

We now review the content of the paper. In Section 2, we give some results
and definitions of \cite{GLT1} about positively multiplicative graphs. In
Section 3, we introduce the notion of affine Weyl groups together with their
action on the weight lattice $P_a$ and their action on the set of alcoves. In
Section 4, we define the graph~$\Gamma _{\gamma }$ associated to a level $0$
weight $\gamma $ and provide a lot of examples. In Section 5, we introduce
the homology ring of affine Grassmanians and present some of its properties.
In Section 6, we prove that $\Gamma _{\gamma }$ is positively
multiplicative for all level $0$ weight $\gamma$. Finally, in Section 7, 
we present an interpretation of the graph $\Ga_\ga$ in terms of particle moving on either a circle or a segment.

\medskip
\noindent\textbf{Acknowledgments:} The authors would like to thank James Parkinson and Eloïse Little for many helpful discussions.

\section{Multiplicative graphs}
\label{mutl_graph}

In this section we review some results of \cite{GLT1}. We start by defining the notion of positively mutliplicative algebras and positively multiplicative graphs.

\subsection{Multiplicative finite graphs}
\label{Sec_PM_Graphs}

Let $\sZ = \{\sz_1,\ldots,\sz_N\}$ be a set (possibly empty) of formal indeterminates.  Let $\Q[\sZ^{\pm 1}]$ be the ring of Laurent polynomials in the variables $\sZ$ and let $\Q(\sZ^{\pm 1})$ be its fraction field. Write $\Q_+[\sZ^{\pm 1}]$ for the subset of~$\Q[\sZ^{\pm 1}]$ containing the
polynomials with nonnegative coefficients. 

\medskip

Let $\sLa$ be a unital algebra over $\Q[\sZ^{\pm 1}]$ and let $\sB$ be a basis of $\sLa$ containing $1$.  We say that $\sLa$ is positively multiplicative with respect to $\sB$, if the structure constants of the basis $\sB$ are all in $\Q_+[\sZ^{\pm 1}]$. We say that~$\sLa$ is positively multiplicative if it is  positively multiplicative with respect to some basis containing~$1$.

\medskip

Let $\Gamma=(V,E)$ be a finite oriented weighted graph with set of
vertices $V=\{v_{1},\ldots,v_{n}\}$, set of arrows~$E\subset V\times V$ and edge weight function with values in $\Q_+[\sZ^{\pm1}]$. We denote by $A_\Ga=(a_{i,j})_{1\leq i,j\leq n}$ the adjacency matrix of $\Gamma$,  that is the coefficient $a_{i,j}$ is equal to $0$ if there is no arrow from $v_j$ to $v_i$ and to the sum of the weights of the arrows from $v_j$ to $v_i$ otherwise. The adjacency algebra of $\Ga$ is the algebra $\Q[\sZ^{\pm1}][A_\Ga]$ of polynomials in $A_\Ga$ with coefficients in~$\Q[\sZ^{\pm1}]$. It will often be convenient to split any arrow in $\Gamma$ weighted by a polynomial $P$ in $\Q_+[\sZ^{\pm1}]$ into a multiple set of arrows, each of them being weighted by a monomial in $\Q_+[\sZ^{\pm1}]$.   

\medskip

 In this section, we will assume that all graphs are strongly connected. This means that for each pair of vertices
$(u,v)$, there is an oriented path from $u$ to~$v$. 

\begin{definition}
\label{Def_MGraphs}
A graph  $\Gamma=(V,E)$ is said to be 
\emph{multiplicative} at $v_{i_{0}}\in V$ if its adjacency algebra is contained in
a commutative algebra $\sLa$ of $n\times n$ matrices with coefficients in $\Q(\sZ^{\pm 1})$ with a basis $\sB_{i_{0}}=\{\sb_{i},i=1,\ldots,n\}$ such that 
\begin{equation*}
\sb_{i_{0}}=1\text{ and }A_\Ga \sb_{j}=\sum_{i=1}^{n}a_{i,j}\sb_{i}\text{ for any
}j=1,\ldots,n. \label{harmonic}%
\end{equation*}
In other words, the matrix $\func{Mat}_{\sB_{i_0}}(m_{A_\Ga})$ of the multiplication by $A_\Ga$ in the basis $\sB_{i_{0}}$ is $A_\Ga$ itself. In addition, if the algebra $\sLa$ is positively multiplicative with respect to $\sB_{i_0}$ we say that $\Ga$ is positively multiplicative at~$v_{i_0}$.
\end{definition}

Let $\sLa$ be a commutative algebra and $x\in \sLa$. As above, we write $m_x$ for the linear map $\sLa\to\sLa$ defined by $m_x(a) = xa$. We start by stating a useful and easy result.  
\begin{proposition}
\label{m_x}
Let $\Ga$ be a finite oriented weighted graph with adjacency matrix $A_\Ga = (a_{i,j})$. The graph $\Ga$ is multiplicative at $v_{i_0}$ if and only if there exists a commutative $n$-dimensional  algebra $\sLa$ over $\Q(\sZ^{\pm 1})$, an element $x\in \sLa$ and a basis $\sB=(\sb_1,\ldots,\sb_{n})$ of $\sLa$ such that $\sb_{i_0} =1$ and $\func{Mat}_{\sB}(m_x) = A_\Ga$.
\end{proposition}
\begin{proof}
Assume that there exists such an algebra $\sLa$. The map
$$\begin{array}{ccccccc}
\func{Mat}_{\sB}&:& \sLa&\to&\func{Mat}_n(\Q(\sZ^{\pm 1}))\\
&& y &\mapsto & \func{Mat}_{\sB}(m_y)
\end{array}$$
is an injective morphism of algebras. Then $\func{Mat}_{\sB}(\sLa)$ is a commutative subalgebra of $\func{Mat}_n(\Q(\sZ^{\pm 1}))$. We claim that $\Ga$ is multiplicative at $v_{i_0}$ with associated algebra $\func{Mat}_{\sB}(\sLa)$ and basis 
$$\{\func{Mat}_{\sB}(m_{\sb_i})\mid i=1,\ldots,n\}.$$
Indeed we have $x\sb_j=m_x(\sb_j) = \sum_{i=1}^{n} a_{i,j} \sb_i$ so that $m_{x\sb_j} = \sum a_{i,j} m_{\sb_i}$. Then 
$$A_\Ga \func{Mat}_{\sB}(m_{\sb_j}) = \func{Mat}_{\sB}(m_x) \func{Mat}_{\sB}(m_{\sb_j}) =\func{Mat}_{\sB}(m_xm_{\sb_j}) =\func{Mat}_{\sB}(m_{x\sb_j}) = \sum_{i=1}^{n} a_{i,j} \func{Mat}_{\sB}(m_{\sb_i}).$$
Since $\sb_{i_0}=1$, we have $\func{Mat}_{\sB}(m_{x\sb_j}) =1$ hence the result. The converse is obvious by definition of a positively multiplicative graph.
\end{proof}

\subsection{Graph of maximal dimension}

We say that a graph $\Ga$ is of maximal dimension when its adjacency algebra is of dimension $n$, the number of vertices in $\Ga$ and the size of $A_\Ga$. This is the case if and only if the minimal polynomial $\mu_{A_\Ga}$ of~$A_\Ga$ is of degree $n$.

\begin{theorem}{\cite[Proposition 3.6]{GLT1}}
\label{unicity_mult_basis}
Let $\Ga$ be a graph of maximal dimension. Then there exists a basis $\sB=(\sb_1,\ldots,\sb_n)$ of the adjacency algebra of $A_\Ga$ such that $\func{Mat}_{\sB}(m_{A_\Ga})=A_\Ga$. Further, if $\sb_{i_0}$ is invertible, there exists a unique basis $\sB'=(\sb'_1,\ldots,\sb'_n)$ such that $b'_{i_0}=1$ and $\func{Mat}_{\sB'}(m_{A_\Ga})=A_\Ga$.
\end{theorem}

\noindent
Note that this is not always possible to find an index $i_0$ such that $\sb_{i_0}$ is invertible, see \cite[\S 3.2]{GLT1} for details.  

\medskip

A path of length $N\in \N$ in a graph $\Ga$ is a sequence $v_0,v_1,\ldots,v_N$ of vertices together with a family of arrows $(e_1,\ldots,e_N)$ where $e_i$ is an arrow from  $v_{i-1}$ to $v_i$.  The weight of a path is the product of the weights of the arrows, it is therefore an element of $\Q_+[\sZ^{\pm1}]$. 

\medskip

Let $k\in \{1,\ldots,n\}$. We denote by $M_k$ the $n\times n$ matrix whose coefficient $m_{i,j}$ is the sum of the weights of all paths of length $j-1$ starting at $v_k$ and ending at $v_i$.

\begin{proposition}{\cite[Proposition 3.6]{GLT1}}
\label{ThGG1}
\begin{enumerate}
\item The graph $\Gamma$  is multiplicative at $v_{i_0}$ and is of maximal dimension if
and only if the matrix $M_{i_{0}}$ is invertible. In this case, the
columns of $M_{i_{0}}^{-1}$ define a basis
$\sB_{i_{0}}=\{\sb_{1},\ldots,\sb_{i_{0}}=1,\ldots,\sb_{n}\}$  of the adjacency algebra (expressed
in the basis $\{1,A_\Ga,\ldots,A_\Ga^{n-1}\}$) that satisfies $\func{Mat}_{\sB_{i_0}}(m_{A_\Ga}) = A_\Ga$.

\item When $M_{i_0}$ is invertible, the entries of the matrices in the basis $\sB_{i_0}$ belong to
$\frac{1}{\det M_{i_{0}}}\Q[\sZ^{\pm1}]$.

\item When $\mu_{A}$ is irreducible, the matrix $M_{i_{0}}$ is invertible for any vertex $v_{i_0}$.
\end{enumerate}
\end{proposition}

\begin{example}
	\label{ex_A3}
Consider the graph $\Gamma$ with adjacency matrix $A_\Ga$ given as follows:

\medskip

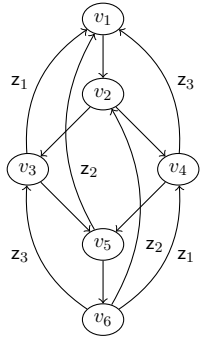
\begin{minipage}{8cm}
$$
\begin{tikzpicture}[scale =1]
\tikzstyle{vertex}=[inner sep=2pt,minimum size=10pt,ellipse,draw]

\node[vertex] (a0) at (0,0) {\scalebox{.7}{$v_1$}};
\node[vertex] (a1) at (0,-1) {\scalebox{.7}{$v_2$}};
\node[vertex] (a2) at (-1,-2) {\scalebox{.7}{$v_3$}};
\node[vertex] (a3) at (1,-2) {\scalebox{.7}{$v_4$}};
\node[vertex] (a4) at (0,-3) {\scalebox{.7}{$v_5$}};
\node[vertex] (a5) at (0,-4) {\scalebox{.7}{$v_6$}};

\draw[->] (a0) edge node  {} (a1);
\draw[->] (a1) edge node{} (a2);
\draw[->] (a1) edge node{} (a3);
\draw[->] (a2) edge node{} (a4);
\draw[->] (a3) edge node{} (a4);
\draw[->] (a4) edge node{} (a5);

\draw[->]  (a2) edge[bend left = 30] node[left]{\scalebox{.7}{$\sz_1$}} (a0);
\draw[->]  (a3) edge[bend left = -30] node[right]{\scalebox{.7}{$\sz_3$}} (a0);

\draw[->]  (a4) edge[bend left = 30] node[right,pos=.3]{\scalebox{.7}{$\sz_2$}} (a0);
\draw[->]  (a5) edge[bend right = 30] node[right,pos=.3]{\scalebox{.7}{$\sz_2$}} (a1);

\draw[->]  (a5) edge[bend left = 30] node[left]{\scalebox{.7}{$\sz_3$}} (a2);
\draw[->]  (a5) edge[bend left = -30] node[right]{\scalebox{.7}{$\sz_1$}} (a3);

\end{tikzpicture}$$
\end{minipage}
\begin{minipage}{4cm}
$$
A_\Ga=\left(
\begin{array}
[c]{cccccc}%
0 & 0 & \sz_{1} & \sz_{3} & \sz_{2} & 0\\
1 & 0 & 0 & 0 & 0 & \sz_{2}\\
0 & 1 & 0 & 0 & 0 & \sz_{3}\\
0 & 1 & 0 & 0 & 0 & \sz_{1}\\
0 & 0 & 1 & 1 & 0 & 0\\
0 & 0 & 0 & 0 & 1 & 0
\end{array}
\right)
.$$
\end{minipage}

\medskip

\noindent
The minimal polynomial of $A_\Ga$ is
$$\mu_{A_\Ga}(T)=T^{6}-2\left(  \sz_{1}+\sz_{3}\right)  T^{3}-4\sz_{2}T^{2}+\left(  \sz_{1}-\sz_{3}\right)  ^{2}$$
 which is irreducible. The matrices
$M_{1}$ and $M_{1}^{-1}$ are respectively%
\[
\begin{pmatrix}
1 & 0 & 0 & \sz_{1}+\sz_{3} & 2\sz_{2} & 0\\
0 & 1 & 0 & 0 & \sz_{1}+\sz_{3} & 4\sz_{2}\\
0 & 0 & 1 & 0 & 0 & \sz_{1}+3\sz_{3}\\
0 & 0 & 1 & 0 & 0 & 3\sz_{1}+\sz_{3}\\
0 & 0 & 0 & 2 & 0 & 0\\
0 & 0 & 0 & 0 & 2 & 0
\end{pmatrix}
\qu{and}
\begin{pmatrix}
1 & 0 & 0 & 0 & -\frac{\sz_{1}+\sz_{3}}{2} & -\sz_{2}\\
0 & 1 & \frac{2\sz_{2}}{\sz_{1}-\sz_{3}} & \frac{2\sz_{2}}{\sz_{3}-\sz_{1}} & 0 &
-\frac{\sz_{1}+\sz_{3}}{2}\\
0 & 0 & \frac{3\sz_{1}+\sz_{3}}{2\sz_{1}-2\sz_{3}} & \frac{\sz_{1}+3\sz_{3}}{2\sz_{3}%
-2\sz_{1}} & 0 & 0\\
0 & 0 & 0 & 0 & \frac{1}{2} & 0\\
0 & 0 & 0 & 0 & 0 & \frac{1}{2}\\
0 & 0 & \frac{-1}{2\sz_{1}-2\sz_{3}} & \frac{-1}{2\sz_{3}-2\sz_{1}} & 0 & 0
\end{pmatrix}.
\]
The graph $\Gamma$ is multiplicative at $v_1$. Observe that in this particular example, that all the entries in $\{\sb_{1}=1,\ldots,\sb_{6}\}$ belong to $\mathbb{Q_{+}}%
[\sz_{1},\sz_{2},\sz_{3}]$. For example, we have
\[
b_{3}=\frac{1}{\sz_{1}-\sz_{3}}\left(  2\sz_{2}A+\frac{1}{2}(3\sz_{1}+\sz_{3}%
)A^{2}-\frac{1}{2}A^{5}\right)  =\left(
\begin{array}
[c]{cccccc}%
0 & \sz_{1} & \sz_{2} & 0 & 0 & \sz_{1}\sz_{3}\\
0 & 0 & \sz_{1} & 0 & \sz_{2} & 0\\
1 & 0 & 0 & 0 & 0 & 0\\
0 & 0 & 0 & 0 & \sz_{1} & \sz_{2}\\
0 & 1 & 0 & 0 & 0 & \sz_{1}\\
0 & 0 & 0 & 1 & 0 & 0
\end{array}
\right).  
\]

\end{example}

\subsection{Expansion of a graph}
Let $\Gamma$ be a strongly connected weighted graph with set of vertices $\{v_1,\ldots,v_n\}$. We assume that all the weights in $\Ga$ are positive  monomials in $\Q_{+}[\sZ^{\pm1}]$. 
\begin{definition}
The \emph{expansion} of the graph $\Gamma$ at $v_1$ is the graded graph with set of vertices $\Gamma_{e}=(\Gamma^{\ell})_{\ell\in\mathbb{N}^\ast}$ constructed by induction as follows:
\begin{enumerate}
\item[{\tiny $\bullet$}] $\Gamma_1=\{(v_1,1,1)\}$
\item[{\tiny $\bullet$}] if $(v_{j},\sz^{\beta},\ell)\in \Gamma^\ell$ and there is an arrow $v_{j}\overset{m_{i,j}\sz^{\beta_{i,j}}}{\rightarrow
}v_{i}$ in $\Gamma$ then $(v_{i},\sz^{\beta+\beta_{i,j}},\ell+1)\in \Gamma^{\ell+1}$ and there is an arrow $
(v_{j},\sz^{\beta},\ell)\overset{m_{i,j}}{\rightarrow}(v_{i},\sz^{\beta+\beta_{i,j}},\ell+1)
$ in $\Gamma_e$.
\end{enumerate}
\end{definition}
Assume that $\Ga$ is a finite positively multiplicative graph. As in Definition \ref{Def_MGraphs}, let $\sLa$ be a commutative $n$-dimensional algebra over $\Q(\sZ^{\pm 1})$, $x\in \sLa$ and $\sB=(\sb_1,\ldots,\sb_{n})$ a basis of $\La$ such that $\sb_{1} =1$ and $\func{Mat}_{\sB}(m_x) = A_\Ga$. 
Since $\sLa$  is an algebra over~$\Q[\sZ^{\pm1}]$, it can also be seen as a $\Q$-algebra of infinite dimension with basis 
$$\{\sz^\beta \sb_i\mid i\in \{1,\ldots,n\}, \beta\in \Z^m\}$$ where for all $\beta=(\beta_1,\ldots,\beta_m)\in \Z^m$ we have set $\sz^{\beta} = \sz_1^{\beta_1}\ldots \sz_n^{\beta_n}$. We will denote by $\sLa_\Q$ this algebra.
We set $\sLa'_{e}=\sLa_{\Q}\otimes_{\Q}\Q[\sq]$ where $\sq$ is a new
indeterminate. It is easy to check that the set 
$$\sB_{e}^{\prime}=\{\sq^{\ell
}\sz^{\beta}\sb_{i}\mid  i\in \{1,\ldots,n\},\beta\in\mathbb{Z}^{m},\ell\in\mathbb{N}%
\}$$
is a basis of $\sLa'_{e}$.  We write $\sLa_e$ for the subspace of $\sLa'_e$ with basis $\sB_{e}=\{\sq^{\ell}\sz^{\beta}\sb_{i}\mid
(v_{i},\sz^{\beta},\ell)\in\Gamma_{e}\}.$

\begin{proposition}
\label{Prop_Gammae_PM}With the previous notation, the following assertions
hold in $\sLa_{e}$.

\begin{enumerate}
\item For any vertex $(v_{j},\sz^{\beta},\ell)\in\Gamma_{e}$
\[
\sq x\times \sq^{\ell}\sz^{\beta}\sb_{j}=\sum_{(v_{j},\sz^{\beta},\ell)\overset{m_{i,j}%
}{\rightarrow}(v_{i},\sz^{\beta'},\ell+1)}m_{i,j}\sq^{\ell+1}%
\sz^{\beta'}\sb_{i}.
\]

\item The element $1=\sq^{0}\sz^{0}\sb_{1}$ belongs to $\sB_{e}$ and the
product of two elements in the basis $\sB_{e}$ expands on~$\sB_{e}$ with nonnegative coefficients. In particular
$\sLa_e$ is a positively multplicative subalgebra of $\sLa'_e$ with associated basis
$\sB_{e}$.
\end{enumerate}
\end{proposition}
\begin{example}\label{ex_expan}
The expansion of \begin{tikzpicture}[scale =1,baseline]
\tikzstyle{vertex}=[inner sep=1pt,minimum size=10pt,draw,circle]
\node[vertex] (a0) at (0,.5) {\scalebox{.9}{$v_1$}};
\node[vertex] (a1) at (0,-.5) {\scalebox{.9}{$v_2$}};

\draw[->] (a0) -- (a1);
\draw[->,bend left=55] (a1) to node[left]{$\sz_1$} (a0);
\draw[->,bend right=55] (a1) to node[right]{$\sz_2$} (a0);

\end{tikzpicture}
 is the graph 
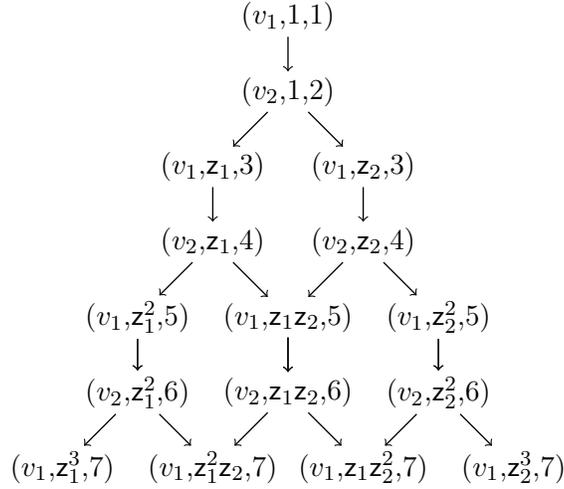
\begin{center}
\begin{tikzpicture}[scale =1]
\def\rx{0.5}
\def\ry{0.866}
\tikzstyle{vertex}=[inner sep=2pt,minimum size=10pt]

\node[vertex] (a0) at (0,0) {(\scalebox{1}{$v_1$,1,1)}};
\node[vertex] (a1) at (0,-1) {(\scalebox{1}{$v_2$,1,2)}};
\node[vertex] (a21) at (-1,-2) {(\scalebox{1}{$v_1$,$\sz_1$,3)}};
\node[vertex] (a22) at (1,-2) {(\scalebox{1}{$v_1$,$\sz_2$,3)}};

\node[vertex] (a31) at (-1,-3){(\scalebox{1}{$v_2$,$\sz_1$,4)}};
\node[vertex] (a32) at (1,-3) {(\scalebox{1}{$v_2$,$\sz_2$,4)}};

\node[vertex] (a41) at (-2,-4) {(\scalebox{1}{$v_1$,$\sz_1^2$,5)}};
\node[vertex] (a42) at (0,-4) {(\scalebox{1}{$v_1$,$\sz_1\sz_2$,5)}};
\node[vertex] (a43) at (2,-4) {(\scalebox{1}{$v_1$,$\sz_2^2$,5)}};

\node[vertex] (a51) at (-2,-5) {(\scalebox{1}{$v_2$,$\sz_1^2$,6)}};
\node[vertex] (a52) at (0,-5) {(\scalebox{1}{$v_2$,$\sz_1\sz_2$,6)}};
\node[vertex] (a53) at (2,-5)  {(\scalebox{1}{$v_2$,$\sz_2^2$,6)}};

\node[vertex] (a61) at (-3,-6) {(\scalebox{1}{$v_1$,$\sz_1^3$,7)}};
\node[vertex] (a62) at (-1,-6) {(\scalebox{1}{$v_1$,$\sz_1^2\sz_2$,7)}};
\node[vertex] (a63) at (1,-6) {(\scalebox{1}{$v_1$,$\sz_1\sz_2^2$,7)}};
\node[vertex] (a64) at (3,-6) {(\scalebox{1}{$v_1$,$\sz_2^3$,7)}};

\draw[line width=.5pt,->] (a0) to   (a1);
\draw[line width=.5pt,->] (a1) to  (a21);
\draw[line width=.5pt,->] (a1) to  (a22);

\draw[line width=.5pt,->] (a21) to  (a31);
\draw[line width=.5pt,->] (a22) to  (a32);

\draw[line width=.5pt,->] (a31) to  (a42);
\draw[line width=.5pt,->] (a32) to  (a42);
\draw[line width=.5pt,->] (a31) to  (a41);
\draw[line width=.5pt,->] (a32) to  (a43);

\draw[line width=.5pt,->] (a41) to  (a51);
\draw[line width=.5pt,->] (a42) to  (a52);
\draw[line width=.5pt,->] (a43) to  (a53);

\draw[line width=.5pt,->] (a41) to  (a51);
\draw[line width=.5pt,->] (a42) to  (a52);
\draw[line width=.5pt,->] (a43) to  (a53);

\draw[line width=.5pt,->] (a51) to  (a61);
\draw[line width=.5pt,->] (a51) to  (a62);
\draw[line width=.5pt,->] (a52) to  (a62);
\draw[line width=.5pt,->] (a52) to  (a63);
\draw[line width=.5pt,->] (a53) to  (a63);
\draw[line width=.5pt,->] (a53) to  (a64);
\end{tikzpicture}
\end{center}
%
%
%
%

\end{example}
\section{Affine Weyl group}

\subsection{Affine root systems and affine Lie algebras}

\label{Subsec_affineRS}

Let $I=\{0,1,\ldots,n\}$, $I^\ast=\{1,2,\ldots,n\}$ and $A=(a_{i,j})_{(i,j)\in I^{2}}$ be a generalized
Cartan matrix of a non twisted affine root system.\ These affine root systems
are classified in~\cite{KacB} (see the Table 1 page 54). The
matrix $A$ has rank $n$ and there exists a unique vector $v = (a_i)_{i\in I}\in\mathbb{Z}^{n+1}$ with $(a_i)_{i\in I}$ relatively primes and a unique vector $v^{\vee} =(a_i^\vee)_{i\in I}\in \Z^{n+1}$ with $(a^\vee_i)_{i\in I}$ relatively primes such that $v^{\vee}\cdot A=A\cdot {^t}\hspace{-.1mm}v=0$. Note that, since we are in the non twisted case, we have $a_0=a_0^\vee = 1$. We refer to \cite{carter_affine,KacB} for details and proofs of the results presented in this section.  


\medskip

Let $(\fh,\Pi,\Pi^\vee)$ be a realization of $A$ that is
\begin{enumerate}
\item $\fh$ is a complex vector space of dimension $n+2$,
\item $\Pi=\{\al_0,\ldots,\al_n\}\subset \fh^\ast$ and $\Pi^\vee = \{\al_0^\vee,\ldots,\al_n^\vee\}\subset \fh$ are linearly independant subsets,
\item $a_{i,j} = \scal{\al_j}{\al^\vee_i}$ for $0\leq i,j\leq n$.
\end{enumerate}
Here, $\scal{\cdot}{\cdot}:\fh^\ast\times \fh\to \mathbb{C}$ denotes the pairing $\scal{\al}{h} = \al(h)$.
Fix an element $d\in \fh$ such that $\scal{\al_i}{d} = \delta_{0,i}$ for all $i\in I$ so that $\Pi^\vee\cup \{d\}$ is a basis of $\fh$. 
We denote by $\fg$ the affine Lie algebra associated to this datum and refer to ~\cite{KacB} for its definition.

\medskip

Let $\La_0,\ldots,\La_{n}\subset \fh^\ast$ be the such that 
$$\La_i(\al^\vee_j) = \delta_{i,j}\qu{and}\La_i(d) = 0\text{ for all $i,j\in I$.}$$
We set $\delta = \sum_{j=0}^{n} a_j \al_j$ so that 
$$\delta(\al_i^\vee) = \sum_{j=0}^n a_j \al_j(\al_i^\vee)=\sum_{j=0}^n a_j a_{i,j} = [Av]_{i} =0\qu{and} \delta(d) = a_0 = 1.$$
Then the family $(\La_0,\ldots,\La_n,\delta)$ is the dual basis of $(\al^\vee_0,\ldots,\al^\vee_n,d)\subset \fh$ and we have
$$\al_j = \sum_{i\in I} \scal{\al_j}{\al_i^\vee} \La_i + \scal{\al_j}{d}\delta = \sum_{i\in I} a_{i,j} \La_i \qu{for all $j\in I^\ast.$}$$

\smallskip

There exists an invariant nondegenerate symmetric bilinear form $(\cdot,\cdot)$ on $\fg$ such that $(\cdot,\cdot)$ is nondegenerate on $\fh$ and uniquely defined by
$$\begin{cases}
(\al_i^\vee,\al^\vee_j) = \scal{\al_i}{\al_j^\vee}a_i{a_i^\vee}^{-1}=a_j{a_j^\vee}^{-1}a_{j,i}&\mbox{for $i,j\in I$}\\
(\al_i^\vee,d) = 0 &\mbox{for $i\in I^\ast$}\\
(\al_0^\vee,d) = a_0 &\\
(d,d) = 0
\end{cases}
$$
It can be checked that $(\al_i^\vee,\al_j^\vee) = (\al_j^\vee,\al_i^\vee)$ for all $i,j\in I$. Let $\nu$ be the associated map from $\fh$ to its dual:
$$\begin{array}{ccccccc}
\nu &:& \fh & \to & \fh^\ast\\
&& h&\mapsto & (h,\cdot) 
\end{array}$$
Then we have $\nu(d) =\La_0$, $a_{i,j} = \dfrac{2(\al_i,\al_j)}{(\al_i,\al_i)}$ and $\nu(\al_i^{\vee}) = \dfrac{2\al_i}{(\al_i,\al_i)}$. The form $(\cdot,\cdot)$ on $\fh$ then induces a form on $\fh^\ast$ via $\nu$. We still denote this form $(\cdot,\cdot)$ and we have
$$\begin{cases}
(\al_i,\al_j) = a_i^\vee{a_i}^{-1}a_{i,j}&\mbox{for $i,j\in I$}\\
(\al_i,\La_0) = 0 &\mbox{for $i\in I^\ast$}\\
(\al_0,\La_0) = a_0^{-1} \\
(\La_0,\La_0) = 0\\
(\La_0,\delta)=1
\end{cases}
$$
%

There are different objects associated to the datum $(\fh,\Pi,\Pi^\vee)$, some of them live in $\fh$ other in $\fh^\ast$. We will as much as possible use the following convention: we will add the suffix ``co'' to the name of the object to indicate that it naturally lives in $\fh$ (eventhough we sometime think of it as an element of $\fh^\ast$) and we will add a superscript $^\vee$ to the notation. For instance, $\al_i^\vee$ is called a coroot as it is an element of~$\fh$. We now introduce various lattices:
\begin{enumerate}
\itemsep.4em 
\item[{\tiny $\bullet$}] the coweight lattice: $P_a^\vee := \bigoplus_{i\in I} \Z \al_i^\vee +\Z d\subset \fh$,
\item[{\tiny $\bullet$}] the weight lattice $P_a = \{\ga\in \fh^\ast\mid \ga(P_a^\vee) \subset \Z\} = \bigoplus_{i\in I} \Z\La_i  + \Z\delta$,
 \item[{\tiny $\bullet$}] the dominant weights in $\fh^\ast$ are the elements of $P^+_a =  \bigoplus_{i\in I} \Z_{\geq 0}\La_i  + \Z\delta$,
\item[{\tiny $\bullet$}] the root lattice  $Q_a=\bigoplus_{i\in I} \Z \al_i$,
\item[{\tiny $\bullet$}] the coroot lattice $Q_a^\vee=\bigoplus_{i\in I} \Z \al^\vee_i$.
\end{enumerate}
For any $i\in I$, we define the simple reflection $s_{i}$ on $\fh^\ast$ by
\begin{equation*}
s_{i}(x)=x-\scal{\alpha_{i}^{\vee}}{x}\alpha_{i}\ \text{for any
}x\in \fh^\ast\text{.} \label{defSi}%
\end{equation*}
The affine Weyl group $W_{a}$ is the subgroup of $GL(\mathfrak{h}^{\ast})$ generated
by the reflections $s_{i}$. Since $(\La_0,\ldots,\La_n,\delta)$ is the dual basis of $(\al_0^\vee,\ldots,\al_n^\vee,d)$, we have for all $i\in I$
\begin{align*}
s_i(\delta) &=\delta - \scal{\al_i^\vee}{\delta} \al_i = \delta\\
s_i(\La_j) &= \La_j \qu{if $i\neq j$}\\
s_i(\al_i) &= -\al_i
\end{align*} 
The Weyl group $W_a$ is acting on the weight
lattice $P_{a}$. 

\medskip

Let $\os{\ \circ}{A}$ be the matrix obtained from $A$ by deleting the row and the column corresponding to $0$. Then it is well-known that $\os{\ \circ}{A}$ is a Cartan matrix of finite type. Let $\os{\circ}{\fh^\ast}$ and $\os{\circ}{\fh}$ be the vector spaces spanned by the subset $\os{\circ}{\Pi} = \Pi\setminus \{\al_0\}$ and $\os{\hspace{-2.5mm}\circ}{\Pi^\vee} = \Pi\setminus \{\al^\vee_0\}$.  The root and weight lattices associated to $\os{\ \circ}{A}$ are  $Q=\bigoplus_{i\in I^\ast}\mathbb{Z}\alpha_{i}$ and $P=\bigoplus_{i\in I^\ast}\mathbb{Z}\omega_{i}$ where we have set $\omega_{i}=\Lambda_{i}-a_{i}^{\vee}\Lambda_{0}$ for all $i\in I^\ast$. We denote by $W$ the finite Weyl group associated to $\os{\ \circ}{A}$: it is generated by the orthogonal reflections $s_i$ with respect to the hyperplane orthogonal to $\al_i$ in $\os{\circ}{\fh^\ast}$. Finally we set $Q^\vee=\bigoplus_{i\in I^\ast}\mathbb{Z}\alpha^\vee_{i}\subset \os{\circ}{\fh^\ast}$. Note that the reflection $s_i\in W_a$ for $i\in I^\ast$ stablises~$\os{\circ}{\fh^\ast}$ so that the Weyl group $W$ can be seen as the subgroup of $W_a$ generated by $(s_i)_{i\in I^\ast}$.

\medskip

Let $\theta$ be the highest root of the finite root system $\os{\circ}{\Pi} = \{\al_1,\ldots,\al_n\}$. Since we assumed that we are in the untwisted type, we have
$\delta=\alpha_{0}+\theta$. Let $s_\theta$ be the orthogonal reflection with respect to $\theta$ defined by $s_\theta(\ga) = \ga - \scal{\theta^\vee}{\ga}\theta$. For all $\ga\in \fh^\ast$ we have 
$$s_0s_{\theta}(\ga)= \gamma +(\ga,\delta)\theta - ((\ga,\theta)+\dfrac{1}{2}(\theta,\theta)(\ga,\delta))\delta.$$
Consider $\beta\in \os{\circ}{\fh^\ast}$. We define the map $t_\beta$ on $\fh^\ast$ by 
$$t_{\beta}(\gamma) = \gamma +(\gamma,\delta)\beta - ((\gamma,\beta)+\dfrac{1}{2}(\beta,\beta)(\gamma,\delta))\delta$$
for all $\gamma$ in $\fh^\ast$. 
Note that if $(\ga,\delta) = 0$ (as it is the case when $\ga\in \os{\circ}{\fh^\ast}$) we get a simpler formula
$$t_{\beta}(\gamma) = \gamma - (\gamma,\beta)\delta.$$
It is important to observe that $t_{\beta}$ doesn't act as a translation on $\os{\circ}{\fh^\ast}$. 
Nevertheless it can be shown that 
\begin{enumerate}
\item[\tbb] $t_\beta\circ t_{\beta'}=t_{\beta+\beta'}$ for all $\beta,\beta'\in \fh^\ast$,
\item[\tbb] $w\circ t_{\beta} \circ w^{-1} = t_{w(\beta)}$ for all $\beta\in \fh^\ast$ and $w\in W$,
\item[\tbb] $s_{0}=t_{\theta}s_{\theta}=s_{\theta}t_{-\theta}$.
\end{enumerate}
These relations tell us that $W_a$ is the semi-direct product of the finite Weyl group $W$ with  a lattice that is generated by the $W$-orbit of $\theta$. In the untwisted case, this orbit is equal to $\nu(Q^\vee)$ so that we have 
$$W_a\simeq W\ltimes \nu(Q^\vee).$$
We will simply write $W_a\simeq W\ltimes Q^\vee$ where we identify $Q^\vee$ with the lattice $\nu(Q^\vee)$ given by: 
$$\bigoplus_{i\in I} \Z \nu(\al^\vee_i)= \bigoplus_{i\in I} \Z \dfrac{2\al_i}{(\al_i,\al_i)} \subset \fh^\ast.$$

\subsection{Affine action of $W_a$ on alcoves}
\label{SubSecAlcoves}

In the previous section, we have seen that the affine Weyl group~$W_a$ is a semi-direct product $W_a\simeq W\ltimes Q^\vee$. In this section, we construct an affine action of $W_a$ on the set of alcoves in an Euclidean affine space $V$ of dimension $n$. This action can be realised from the linear action in~$\fh^\ast$ described in the previous section but the construction is technical and we will not need it here. We refer to \cite{KacB} and \cite{carter_affine} for details of this construction. For the conventions in this section, we refer to \cite{Bourbaki}.

\begin{remark}
We will write the action of $W_a$ on the affine space $V$ and on the set of alcoves on the right. Let us briefly explain why. In this paper, we are studying the orbit of some weights $\ga\in \fh^\ast$ under the action of $W$ (and $W_a$). We have written this action on the left, so that if $\ga$ has a stabiliser of the form $W_J$ where $J\subset I^\ast$, then the orbit of $\ga$ under the action of $W$ will be in bijection with the set $W^J$ of minimal length representatives  of $W/W_J = \{xW_J\mid x\in W_a\}$. By considering a right action on the set of alcoves, the set $W^J$ becomes connected (i.e. is made up of adjacent alcoves) and this makes it much easier to draw, see Examples \ref{A_2} and \ref{strip_G2}.   
\end{remark}

We start with $\Phi$, a reduced, irreducible, finite, crystallographic root system with set of simple roots $\Delta=\{\alpha_1,\ldots,\alpha_n\}$ in an $n$-dimensional $\R$-vector space $V$ with inner product $(\cdot,\cdot)$. Let $\Phi^+$ be the positive roots associated to~$\Delta$. For all $\al\in \Phi^+$, we define $H_{\al,0}$ to be the hyperplane orthogonal to $\al$ and $s_{\al,0}$ to be the orthogonal reflection with respect to $H_{\al,0}$. Setting $\al^{\vee} =2\alpha/\scald{\al}{\al}$ for all $\al\in \Phi$ we have  
$$H_{\al,0} =\{x\in V\mid \scald{x}{\al} = 0\}\qu{and}(x) s_{\alpha,0}=x-\scald{x}{\alpha}\alpha^{\vee}.$$
The Weyl group $W$ of $\Phi$ is the subgroup of $GL(V)$ generated by the orthogonal reflections $s_\al$ with $\al\in \Phi$.  It is a Coxeter group with set of distinguished generators  $S = \{s_1,\ldots,s_{n}\}$ where we set $s_i := s_{\al_i,0}$.

\medskip

The dual root system of $\Phi$ is $\Phi^{\vee}=\{\alpha^{\vee}\mid \alpha\in \Phi\}$ and the coroot lattice of $\Phi$ is 
$
Q^\vee=\mathbb{Z}\textrm{-span of }\Phi^{\vee}
$. The \textit{fundamental weights} of $\Phi$ are the vectors $\omega_1,\ldots,\omega_n$ where $\scald{\omega_i}{\alpha_j}=\delta_{ij}$. The weight lattice is $P=\mathbb{Z}\omega_1+\cdots+\mathbb{Z}\omega_n$, and the cone of \textit{dominant weights} is $P^+=\mathbb{N}\omega_1+\cdots+\mathbb{N}\omega_n$. Note that $Q^\vee\subseteq P$. 

\medskip

The Weyl chambers are the closures of the open connected components of $V\backslash\mathcal{F}_0$ where 
$$\mathcal{F}_0 = \bigcup_{\al\in \Phi^+} H_{\al,0}.$$
The fundamental Weyl chamber is defined by 
$$
\cC_0=\{x\in V\mid  (x,\alpha)\geq 0\text{ for all $\alpha\in\Phi^+$}\}.
$$
The Weyl group $W$ acts (on the right) simply transitively on the set of Weyl chambers.

 \smallskip

The Weyl group $W$ acts on $Q^\vee$ and the affine Weyl group $W_a$ is $W_a=W\ltimes Q^\vee$ 
where we identify~$\lambda\in Q^\vee$ with the translation $(x)t_\la=x+\lambda$. For $\al\in \Phi^+$ and $k\in \Z$, we define the hyperplane 
$$H_{\alpha,k}=\{x\in V\mid(x,\alpha)=k\}$$
and write $s_{\al,k}$  for the (affine) orthogonal reflection with respect to $H_{\al,k}$. Explicitly, $(x)s_{\alpha,k}=x-(( x,\alpha)-k)\alpha^{\vee}$, so that $s_{\alpha,k}=s_{\alpha}t_{k\alpha^{\vee}}$. It is well-known that $W_a$ is generated by the orthogonal reflections $s_{\alpha,k}$ and that it is a Coxeter group with disitinguished set of generators $S_a=\{s_0,s_1,\ldots,s_n\}$, where $s_0=s_{\theta}t_{\theta^{\vee}}$, with~$\theta$ the highest root of $\Phi$. Any element of $W_a$ can be uniquely written under the form $ut_\al$ where $u\in W$ and $\al\in Q^\vee$. 
We define two applications $\sf:W_a\to W$ and $\wt:W_a\to Q^\vee$ via the equation $w=\sf(w)t_{\wt(w)}$.

\smallskip

Each hyperplane $H_{\alpha,k}$ with $\alpha\in \Phi^+$ and $k\in\mathbb{Z}$ divides $V$ into two half spaces, denoted
$$
H_{\alpha,k}^+=\{x\in V\mid (x,\alpha)> k\}\quad\text{and}\quad H_{\alpha,k}^-=\{x\in V\mid(x,\alpha)< k\}.
$$
This is the "periodic orientation'': it is invariant under translation by~$\lambda\in Q^\vee$.

\medskip

The alcoves of $W_a$ are the closure of the open connected components of $V\backslash\mathcal{F}$ where 
$$\cF = \bigcup_{\al\in \Phi^+,k\in \Z} H_{\al,k}.$$
The fundamental alcove is defined by 
$$
A_0=\{x\in V\mid 0< (x,\alpha)< 1\text{ for all $\alpha\in\Phi^+$}\}.
$$
The hyperplanes bounding $A_0$ are called the \textit{walls} of $A_0$. Explicitly these walls are $H_{\alpha_i,0}$ with $i=1,\ldots,n$ and $H_{\theta,1}$. We say that a \textit{face} of $A_0$ (that is, a codimension~$1$ facet) has type $s_i$ for $i=1,\ldots,n$ if it lies on the wall $H_{\alpha_i,0}$ and of type $s_0$ if it lies on the wall $H_{\theta,1}$. 

\medskip

The affine Weyl group $W$ acts simply transitively on the set of alcoves, and we use this action to identify the set of alcoves with $W$. We will denote this action on the right and identify the set of elements of $W_a$ with the set of alcoves via $w\leftrightarrow A_0w$.  The fundamental alcove $A_0$ corresponds to the identity element $e$. Moreover, we use the action of $W_a$ to transfer the notions of walls, faces, and types of faces to arbitrary alcoves. Alcoves $A$ and $A'$ are called \textit{$s$-adjacent}, written $A\sim_s A'$, if~$A\neq A'$ and $A$ and $A'$ share a common face of type $s$. Under the identification of alcoves with elements of~$W$, the alcoves $w$ and $sw$ are $s$-adjacent.

\begin{definition}
\label{crossing}
Let $w\in W$,  $s \in S$ and $H$ be the unique hyperplane separating $w$ and $sw$. We say that the crossing  from $w$ to $sw$ is positive and write $w\overset{+}{\leadsto} sw$ if $w\in H^{-}$ and $sw\in H^+$.  We say that the crossing is negative and write $w\overset{-}{\leadsto} sw$ otherwise. 
\end{definition}
Let $w\in W_a$ and let $(s_{i_{\ell}},\ldots,s_{i_1})\in S^\ell$ be such that $w=s_{i_\ell}\ldots s_{i_1}$. Then we have
$$
e\sim_{s_{i_{1}}}s_{i_{1}}\sim_{s_{i_{2}}} s_{i_{2}}s_{i_{1}}\sim_{s_{i_{3}}}\cdots\sim_{s_{i_\ell}}s_{i_{\ell}}\cdots s_{i_1}.
$$
In this way, the expression of $s_{i_\ell}\ldots s_{i_1}$ representing $w$ determine a sequence of adjacents alcoves starting at $A_0$ and finishing at $A_0w$. It is a well known result that the length of $w\in W_a$ is equal to the number of hyperplanes that separate $A_0$ and $A_0w$.

\medskip

\begin{example}
\label{G2}
Let $\Phi$ be a root system of type $G_2$ with simple roots $\alpha_1$ and $\alpha_2$ and fundamental weights~$\om_1$ and~$\om_2$. We have $P=Q^\vee$, and the dual root system is
$$
\Phi^\vee:=\pm \{\al_1^\vee,\al_2^\vee,\al_1^\vee+\al_2^\vee,\al_1^\vee+2\al_2^\vee,\al_1^\vee+3\al_2^\vee, 2\al_1^\vee+3\al_2^\vee\}.
$$
Recall that the affine Weyl group acts on the right on the set of alcoves. In the figure below 
\begin{itemize}
\itemsep = .1mm
\item[\tbb] the dark gray alcove is the fundamental alcove $A_0$;
\item[\tbb] the thick arrows are the coroot system $\Phi^\vee$;
\item[\tbb] the alcove $w= s_2s_0s_1s_2s_1$ is in dark green.
\item[\tbb] the set of alcoves $s_2,s_2s_0,s_2s_0s_1,s_2s_0s_1s_2$ are in orange and numbered respectively from 1 to 4. 
\item[\tbb] the set of alcoves $s_1,s_2s_1,s_1s_2s_1,s_0s_1s_2s_1,s_2s_0s_1s_2s_1$ are in green.
\end{itemize}

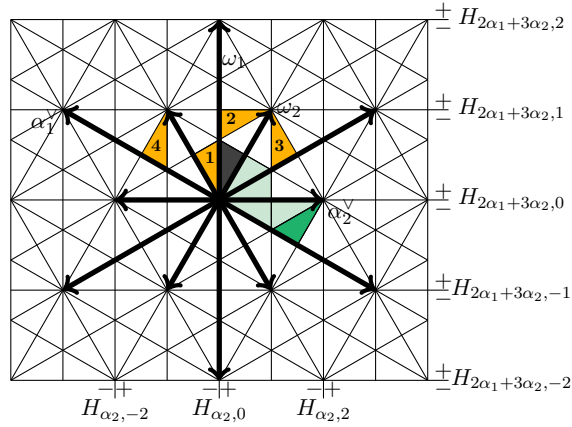
\begin{figure}[H]
\begin{center}
\begin{tikzpicture}[scale = .8]

\newcommand{\x}{.866}
\newcommand{\y}{1.5}

\draw[fill = orange] (0,0) -- (0,1) -- (-\x/2,.75);
\draw[fill = orange] (0,1) -- (0,1.5) -- (\x,1.5);
\draw[fill = orange] (-\x,1.5) -- (-\x,.5) -- (-3/2*\x,.75);
\draw[fill = orange] (\x,1.5) -- (\x,.5) -- (3/2*\x,.75);

\draw[fill = Green!20!] (0,0) -- (\x/2,.75) -- (\x,.5);
\draw[fill = Green!20!] (0,0)  -- (\x,.5) -- (\x,0);
\draw[fill = Green!20!] (0,0)  -- (\x,-.5) -- (\x,0);
\draw[fill = Green!20!] (2*\x,0)  -- (\x,-.5) -- (\x,0);
\draw[fill = Green!80!] (\x,-.5) -- (2*\x,0) -- (3/2*\x,-.75);

\draw[fill = darkgray] (0,0) -- (0,1) -- (\x/2,.75);
\draw[->,line width = .7mm] (0,0) -- (2*\x,0);
\draw[->,line width = .7mm] (0,0) -- (\x,\y);
\draw[->,line width = .7mm] (0,0) -- (-\x,\y);
\draw[->,line width = .7mm] (0,0) -- (-2*\x,0);
\draw[->,line width = .7mm] (0,0) -- (-\x,-\y);
\draw[->,line width = .7mm] (0,0) -- (\x,-\y);

\draw[->,line width = .7mm] (0,0) -- (3*\x,\y);
\draw[->,line width = .7mm] (0,0) -- (0,2*\y);
\draw[->,line width = .7mm] (0,0) -- (-3*\x,\y);
\draw[->,line width = .7mm] (0,0) -- (-3*\x,-\y);
\draw[->,line width = .7mm] (0,0) -- (0,-2*\y);
\draw[->,line width = .7mm] (0,0) -- (3*\x,-\y);

\node at (.25,2.3) {\scalebox{.8}{$\om_1$}};
\node at (\x+.3,\y) {\scalebox{.8}{$\om_2$}};
\node at (2*\x+.3,-.2) {\scalebox{.8}{$\al^\vee_2$}};
\node at (-3*\x-.3,1.3) {\scalebox{.8}{$\al^\vee_1$}};

\draw (4*\x,2*\y) -- (4.4*\x,2*\y);
\node at (4.8*\x+.7,2*\y) {\scalebox{.8}{$H_{2\al_1+3\al_2,2}$}};
\node at (4.3*\x,2*\y+.15) {\scalebox{.8}{$+$}};
\node at (4.3*\x,2*\y-.15) {\scalebox{.8}{$-$}};

\draw (4*\x,1*\y) -- (4.4*\x,1*\y);
\node at (4.8*\x+.7,1*\y) {\scalebox{.8}{$H_{2\al_1+3\al_2,1}$}};
\node at (4.3*\x,1*\y+.15) {\scalebox{.8}{$+$}};
\node at (4.3*\x,1*\y-.15) {\scalebox{.8}{$-$}};

\draw (4*\x,0*\y) -- (4.4*\x,0*\y);
\node at (4.8*\x+.7,0*\y) {\scalebox{.8}{$H_{2\al_1+3\al_2,0}$}};
\node at (4.3*\x,0*\y+.15) {\scalebox{.8}{$+$}};
\node at (4.3*\x,0*\y-.15) {\scalebox{.8}{$-$}};

\draw (4*\x,-1*\y) -- (4.4*\x,-1*\y);
\node at (4.8*\x+.7,-1*\y) {\scalebox{.8}{$H_{2\al_1+3\al_2,-1}$}};
\node at (4.3*\x,-1*\y+.15) {\scalebox{.8}{$+$}};
\node at (4.3*\x,-1*\y-.15) {\scalebox{.8}{$-$}};

\draw (4*\x,-2*\y) -- (4.4*\x,-2*\y);
\node at (4.8*\x+.7,-2*\y) {\scalebox{.8}{$H_{2\al_1+3\al_2,-2}$}};
\node at (4.3*\x,-2*\y+.15) {\scalebox{.8}{$+$}};
\node at (4.3*\x,-2*\y-.15) {\scalebox{.8}{$-$}};

\draw (2*\x,-2*\y) -- (2*\x,-2*\y-.3) ;
\node at (2*\x,-2*\y-.5) {\scalebox{.8}{$H_{\al_2,2}$}};
\node at (2*\x+.15,-2*\y-.15) {\scalebox{.8}{$+$}};
\node at (2*\x-.15,-2*\y-.15) {\scalebox{.8}{$-$}};

\draw (0*\x,-2*\y) -- (0*\x,-2*\y-.3) ;
\node at (0*\x,-2*\y-.5) {\scalebox{.8}{$H_{\al_2,0}$}};
\node at (0*\x+.15,-2*\y-.15) {\scalebox{.8}{$+$}};
\node at (0*\x-.15,-2*\y-.15) {\scalebox{.8}{$-$}};
\draw (-2*\x,-2*\y) -- (-2*\x,-2*\y-.3) ;

\node at (-2*\x,-2*\y-.5) {\scalebox{.8}{$H_{\al_2,-2}$}};
\node at (-2*\x+.15,-2*\y-.15) {\scalebox{.8}{$+$}};
\node at (-2*\x-.15,-2*\y-.15) {\scalebox{.8}{$-$}};

\node at (-.15,.7) {\scalebox{.6}{${\bf 1}$}};
\node at (.2,1.35) {\scalebox{.6}{${\bf 2}$}};
\node at (1,.9) {\scalebox{.6}{${\bf 3}$}};
\node at (-1.05,.9) {\scalebox{.6}{${\bf 4}$}};


\draw (-3.464,3) -- (3.464,3);
\draw (-3.464,1.5) -- (3.464,1.5);
\draw (-3.464,0) -- (3.464,0);
\draw (-3.464,-1.5) -- (3.464,-1.5);
\draw (-3.464,-3) -- (3.464,-3);

\foreach \n in {-4,...,4} {
\draw (\n * \x,-3) -- (\n * \x,3);}

\foreach \e in {-1,1} {
\draw (\e*2*\x,3) -- (\e*4*\x,2);
\draw (\e*0*\x,3) -- (\e*4*\x,1);
\draw (\e*-2*\x,3) -- (\e*4*\x,0);
\draw (\e*-4*\x,3) -- (\e*4*\x,-1);

\draw (-\e*4*\x,2) -- (\e*4*\x,-2);
\draw (-\e*4*\x,1) -- (\e*4*\x,-3);

\draw (-\e*4*\x,0) -- (\e*2*\x,-3);
\draw (-\e*4*\x,-1) -- (\e*0*\x,-3);
\draw (-\e*4*\x,-2) -- (-\e*2*\x,-3);}

\foreach \e in {-1,1} {
\draw (\e*2*\x,3) -- (\e*4*\x,0);
\draw (\e*0*\x,3) -- (\e*4*\x,-3);
\draw (-\e*2*\x,3) -- (\e*2*\x,-3);
\draw (-\e*4*\x,3) -- (\e*0*\x,-3);
\draw (-\e*4*\x,0) -- (-\e*2*\x,-3);}
\end{tikzpicture}
\end{center}
\caption{Alcoves and root system in type $G_2$}
\label{rootsystem}
\end{figure}
\end{example}

\subsection{Affine grassmanians} 
\label{aff_grass}
The weak Bruhat order on $W_a$ is defined as follows: we have $u\leq v$ if and only if there exists a reduced expression  of $v$ starting by a reduced expression of $u$. 
\begin{definition}
\label{Def_AG}
The affine Grassmannian elements are the minimal length
representatives of the left cosets of $W_{a}$ with respect to the finite Weyl group $W$. We will denote this set $W^{\La_0}_a$. The weak Bruhat  order induces a graph structure on $W^{\La_0}_a$ and we will denote this graph by $\Gamma_{\La_0}$. 
\end{definition}
When identifying the set of alcoves with $W_a$, the set $W^{\La_0}_a$ is simply the set of alcoves that lies in the fundamental Weyl chambers $\cC_0$. Two elements of $W^{\La_0}_a$ are connected by a directed path in the graph $\Gamma_{\La_0}$ if and only if there is a path from one to the other such that all the crossing are positive.   

\smallskip

The graph $\Gamma_{\La_0}$ also describes the orbit of the weight $\La_0$ under the action of $W_a$. Indeed the stabiliser of $\La_0$ in $W_a$ is the finite Weyl group $W$, so that the orbit of $\La_0$ is in bijection with $W_a\slash W$.

\begin{example}
\label{A_2}
Let $\Phi$ be a root system of type $A_2$ with simple roots $\alpha_1$ and $\alpha_2$. The dual root system is
$$
\Phi^\vee:=\pm \{\al_1^\vee,\al_2^\vee,\al_1^\vee+\al_2^\vee\}.
$$
In type $A$, it can be useful to color the faces according to their types. In the picture below, we have colored  the $s_1$, $s_2$ and $s_0$ faces in blue, green and red respectively. The gray alcove is the identity and the light gray alcoves represent $W^{\La_0}_a$.
We represent the graph $\Ga_{\La_0}$ on the right hand-side. Note that it coincides with the expansion obtained in Example \ref{ex_expan}.

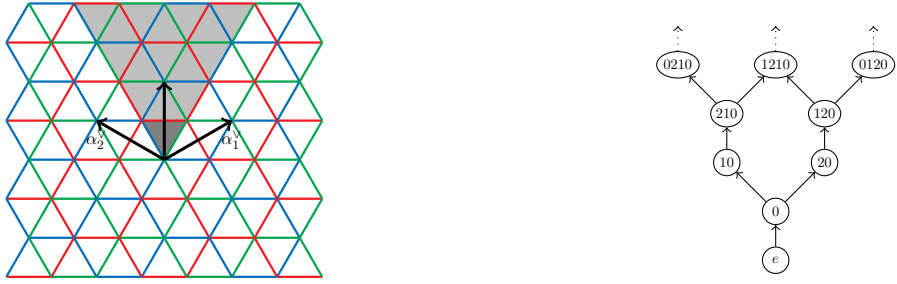
\begin{figure}[H]
\begin{minipage}{8cm}
\begin{center}
\begin{tikzpicture}[scale=.6]
\newcommand{\y}{.866}
\newcommand{\x}{.5}

\draw[fill = lightgray] (0,0) -- (-4*\x,4*\y) -- (4*\x,4*\y);
\draw[fill = gray] (0,0) -- (-\x,\y) -- (\x,\y);


\foreach \k in {0,1}{
\foreach \m in {0,...,3}{
\draw[color = NavyBlue,line width=.3mm] (-3+6*\k*\x,4*\y - \m*2*\y) -- (-2+6*\k*\x,4*\y - \m*2*\y);}}

\foreach \k in {0,1}{
\foreach \m in {0,...,3}{
\draw[color = NavyBlue,line width=.3mm] (-3+3*\x+6*\k*\x,3*\y - \m*2*\y) -- (-2+3*\x+6*\k*\x,3*\y - \m*2*\y);}}

\foreach \k in {0,1,2}{
\foreach \m in {0,...,3}{
\draw[color = Red,line width=.3mm] (-4+6*\k*\x+\x,3*\y - \m*2*\y) -- (-3+6*\k*\x+\x,3*\y - \m*2*\y);}}

\foreach \k in {0,1}{
\foreach \m in {0,...,3}{
\draw[color = Red,line width=.3mm] (-3+3*\x+6*\k*\x-\x,4*\y - \m*2*\y) -- (-2+3*\x+6*\k*\x-\x,4*\y - \m*2*\y);}}

\foreach \k in {0,1}{
\foreach \m in {0,...,3}{
\draw[color = Green,line width=.3mm] (-4+6*\k*\x+3*\x,3*\y - \m*2*\y) -- (-3+6*\k*\x+3*\x,3*\y - \m*2*\y);}}

\foreach \k in {0,1}{
\foreach \m in {0,...,3}{
\draw[color = Green,line width=.3mm] (-3+3*\x+6*\k*\x+\x,4*\y - \m*2*\y) -- (-2+3*\x+6*\k*\x+\x,4*\y - \m*2*\y);}}

\foreach \m in {0,3,6}{
\foreach \n in {0,...,3}{
\draw[color = NavyBlue,line width=.3mm] (-3+2*\m*\x,4*\y-\n*2*\y) -- (-3.5+2*\m*\x,3*\y-\n*2*\y);}

\foreach \n in {0,...,2}{
\draw[color = NavyBlue,line width=.3mm] (-3.5+2*\m*\x,3*\y-\n*2*\y) -- (-3+2*\m*\x,2*\y-\n*2*\y);}}

\foreach \m in {1,4}{
\foreach \n in {0,...,3}{
\draw[color = Red,line width=.3mm] (-3+2*\m*\x,4*\y-\n*2*\y) -- (-3.5+2*\m*\x,3*\y-\n*2*\y);}

\foreach \n in {0,...,2}{
\draw[color = Red,line width=.3mm] (-3.5+2*\m*\x,3*\y-\n*2*\y) -- (-3+2*\m*\x,2*\y-\n*2*\y);}}

\foreach \m in {2,5}{
\foreach \n in {0,...,3}{
\draw[color = Green,line width=.3mm] (-3+2*\m*\x,4*\y-\n*2*\y) -- (-3.5+2*\m*\x,3*\y-\n*2*\y);}

\foreach \n in {0,...,2}{
\draw[color = Green,line width=.3mm] (-3.5+2*\m*\x,3*\y-\n*2*\y) -- (-3+2*\m*\x,2*\y-\n*2*\y);}}

\foreach \m in {0,3,6}{
\foreach \n in {0,...,3}{
\draw[color = Green,line width=.3mm] (-3.5+2*\m*\x+\x,4*\y-\n*2*\y) -- (-3+2*\m*\x+\x,3*\y-\n*2*\y);}

\foreach \n in {0,...,2}{
\draw[color = Green,line width=.3mm](-3+2*\m*\x+\x,3*\y-\n*2*\y) -- (-3.5+2*\m*\x+\x,2*\y-\n*2*\y);}}

\foreach \m in {1,4}{
\foreach \n in {0,...,3}{
\draw[color = NavyBlue,line width=.3mm] (-3.5+2*\m*\x+\x,4*\y-\n*2*\y) -- (-3+2*\m*\x+\x,3*\y-\n*2*\y);}

\foreach \n in {0,...,2}{
\draw[color = NavyBlue,line width=.3mm](-3+2*\m*\x+\x,3*\y-\n*2*\y) -- (-3.5+2*\m*\x+\x,2*\y-\n*2*\y);}}

\foreach \m in {2,5}{
\foreach \n in {0,...,3}{
\draw[color = Red,line width=.3mm] (-3.5+2*\m*\x+\x,4*\y-\n*2*\y) -- (-3+2*\m*\x+\x,3*\y-\n*2*\y);}

\foreach \n in {0,...,2}{
\draw[color = Red,line width=.3mm](-3+2*\m*\x+\x,3*\y-\n*2*\y) -- (-3.5+2*\m*\x+\x,2*\y-\n*2*\y);}}

\draw[color = black, ->,line width = .4mm] (0,0) -- (-3*\x,\y);
\draw[color = black, ->,line width = .4mm] (0,0) -- (3*\x,\y);
\draw[color = black, ->,line width = .4mm] (0,0) -- (0,2*\y);

\node at (3*\x,.5*\y) {\scalebox{.6}{$\al^\vee_1$}};
\node at (-3*\x,.5*\y) {\scalebox{.6}{$\al^\vee_2$}};

\end{tikzpicture}
\end{center}
\end{minipage}
\begin{minipage}{8cm}
\begin{center}
\begin{tikzpicture}[scale=.65]
\tikzstyle{vertex}=[inner sep=.2pt,minimum size=10pt,ellipse,draw]

\node[vertex] (a1) at (0,0) {\scalebox{.5}{$e$}};
\node[vertex] (a2) at (0,1) {\scalebox{.5}{$0$}};
\node[vertex] (a31) at (-1,2) {\scalebox{.5}{$10$}};
\node[vertex] (a32) at (1,2) {\scalebox{.5}{$20$}};
\node[vertex] (a41) at (-1,3) {\scalebox{.5}{$210$}};
\node[vertex] (a42) at (1,3) {\scalebox{.5}{$120$}};
\node[vertex] (a51) at (-2,4) {\scalebox{.5}{$0210$}};
\node[vertex] (a52) at (0,4) {\scalebox{.5}{$1210$}};
\node[vertex] (a53) at (2,4) {\scalebox{.5}{$0120$}};

\node (a61) at (-2,5) {};
\node (a62) at (0,5) {};
\node (a63) at (2,5) {};

\draw[->] (a1) edge node  {} (a2);

\draw[->] (a2) edge node{} (a31);
\draw[->] (a2) edge node{} (a32);

\draw[->] (a31) edge node{} (a41);
\draw[->] (a32) edge node{} (a42);

\draw[->] (a41) edge node{} (a51);
\draw[->] (a41) edge node{} (a52);
\draw[->] (a42) edge node{} (a52);
\draw[->] (a42) edge node{} (a53);

\draw[->,dotted] (a51) edge node{} (a61);
\draw[->,dotted] (a52) edge node{} (a62);
\draw[->,dotted] (a53) edge node{} (a63);

\end{tikzpicture}
\end{center}
\end{minipage}
\caption{Affine grassmanians in type $\tilde{A}_2$ and its induced graph}
\label{rootsystemA2}
\end{figure}
\end{example}

\section{The graphs $\Ga_\ga$ associated to level $0$ weight}\label{SecGraphGg}

The level 0 weight lattice is the sublattice of $P_a$ defined by
\[
P_{0}=\bigoplus_{i\in I^\ast}\mathbb{Z}\omega_{i}+\mathbb{Z\delta}\text{.}%
\]
The lattice $P_{0}$ is stabilized by $W_a$.  In this section, we construct graphs encoding the orbit of the elements of $P_0$, the so-called level $0$ weights, under the action of $W_a$. We denote by $\Z[Q^\vee]$ the group algebra of~$Q^\vee$ over $\Z$ and write $\sz^\al$ for the element associated to~$\al\in Q^\vee$ so that $\sz^\al \sz^{\al'} =\sz^{\al+\al'}$. Recall that the affine Weyl group $W_a$ is the semi-direct product $W\ltimes Q^\vee$ and that we defined  $\sf:W_a\to W$ and $\wt:W_a\to Q^\vee$ via the equation~$w = \sf(w)t_{\wt(w)}$. 

\begin{remark}
We work over $\Q$ to fit with the setting of multiplicative graphs.
\end{remark}

\subsection{The graph $\Gamma_\rho$} We look at the level $0$-weight $\rho=\om_1+\cdots+\om_n$. The stabiliser of $\rho$ under the action of $W$ is reduced to the identity. 
\begin{definition}\label{defGrho}
The graph $\Gamma_{\rho}$ is the oriented weighted graph with weights in $\Z[Q^\vee]$ defined as follows
\begin{itemize}
\itemsep=1mm
\item[\tbb] the set of vertices is $W$ 
\item[\tbb] there is an arrow $w\overset{i}{\underset{a_i}{{\lra}}} s_iw$ of type $i$ and weight $a_i$ for $i\in I^{\ast}$ when $\ell(s_iw)=\ell(w)+1$
\item[\tbb] there is an arrow $w\overset{0}{\underset{\sz^\al}{{\lra}}} \sf(s_0w)$ of type $0$ and weight $\sz^{\al}$ when $\ell(\sf(s_0w))<\ell(w)$ and $\al = \wt(s_0w)$. 
\end{itemize}
\end{definition}
There are different ways to interpret the graph $\Ga_\rho$. First,  if we forget about the weights and the orientations of the arrows, then the graph $\Ga_\rho$ is the Cayley graph of $W$ with generators $s_i$ and $s_\theta$. Indeed we have $\sf(s_0w) = s_\theta w$ for all $w\in W$. As a consequence, since the stabiliser of $\rho$ is reduced to the identity, this shows that $\Ga_\rho$ represents the orbit of $\rho$ under the action of $W$. The edge weight function on $\Ga_\rho$ allows to recover the action of $W_a$ on $\rho$. Let $w\in W$. If we set $u=\sf(s_0w)$ and $\al = \wt(s_0w)$, we have
$$s_0w(\rho) =  ut_\al (\rho) = u(\rho - (\rho,\al)\delta) = u(\rho) - (\rho,\al)\delta$$
so that the weight $\al$ on the arrow from $w$ to $\sf(s_0w)$ gives the coefficients of $\delta$ in the equality above. 
Next, if we forget about the $0$-arrows, then the graph $\Ga_\rho$ is the graph of the weak Bruhat order on $W$. Finally,~$\Ga_\rho$ also encodes encodes the orientation of adjacent alcoves as we will see later in this section.  

\medskip

 \medskip
 
In the following, when drawing examples of graphs $\Gamma_\rho$, we will omit the type $i\in I$ on the arrows as it can be deduced from the two extremities of the arrows. Further, we do not indicate the weight of an arrow when it is equal to $1$.

\begin{example}
We draw the graph $\Ga_{\rho}$ in type $\tilde{A}_2$. We set $\sz_1=\sz^{\al_1}$ and $\sz_2=\sz^{\al_2}$ in $\Z[Q^\vee]$.  The arrows going downward are all of type $i$ with $i\in I^\ast$ and of weight $1$ since we have $a_0=a_1=a_2=1$. The arrows going upward are all of type~$0$. On the left hand-side, we label the vertices with $W$ and on the right hand-side with the corresponding weight $w(\rho)$. One can notice that when the arrows are going down (respectively up), so does the weights (mod $\delta$). 

\medskip
\begin{figure}[H]
\begin{minipage}{8cm}
$$
\begin{tikzpicture}[scale =1.5]
\tikzstyle{vertex}=[inner sep=2pt,minimum size=10pt,ellipse,draw]

\node[vertex] (a1) at (0,0) {\scalebox{.7}{$e$}};
\node[vertex] (a2) at (-1.1,-1) {\scalebox{.7}{$1$}};
\node[vertex] (a3) at (1.1,-1) {\scalebox{.7}{$2$}};
\node[vertex] (a4) at (-1.1,-2) {\scalebox{.7}{$21$}};
\node[vertex] (a5) at (1.1,-2) {\scalebox{.7}{$12$}};
\node[vertex] (a6) at (0,-3) {\scalebox{.7}{$121$}};

\draw[->] (a1) edge node  {} (a2);
\draw[->] (a1) edge node{} (a3);
\draw[->] (a2) edge node{} (a4);
\draw[->] (a3) edge node{} (a5);
\draw[->] (a5) edge node{} (a6);
\draw[->] (a4) edge node{} (a6);

\draw[->]  (a6) edge[bend left = 0] node[right,pos = .25]{\scalebox{.6}{$\dfrac{1}{\sz_1\sz_2}$}} (a1);
\draw[->]  (a4) edge[bend right = 0] node[pos=.8,above]{{\scalebox{.6}{$\sz_1^{-1}$}}} (a3);
\draw[->]  (a5) edge[bend left = 0] node[pos=.8,above]{{\scalebox{.6}{$\sz_2^{-1}$}}} (a2);

\end{tikzpicture}$$

\end{minipage}
\begin{minipage}{8cm}
$$
\begin{tikzpicture}[scale =1.5]
\tikzstyle{vertex}=[inner sep=2pt,minimum size=10pt,ellipse,draw]

\node[vertex] (a1) at (0,0) {\scalebox{.7}{$\rho$}};
\node[vertex] (a2) at (-1.1,-1) {\scalebox{.7}{$\rho-\al_1$}};
\node[vertex] (a3) at (1.1,-1) {\scalebox{.7}{$\rho-\al_2$}};
\node[vertex] (a4) at (-1.1,-2) {\scalebox{.7}{$\rho-2\al_2-\al_1$}};
\node[vertex] (a5) at (1.1,-2) {\scalebox{.7}{$\rho-2\al_1-\al_2$}};
\node[vertex] (a6) at (0,-3) {\scalebox{.7}{$\rho-2\al_2-2\al_1$}};

\draw[->] (a1) edge node  {} (a2);
\draw[->] (a1) edge node{} (a3);
\draw[->] (a2) edge node{} (a4);
\draw[->] (a3) edge node{} (a5);
\draw[->] (a5) edge node{} (a6);
\draw[->] (a4) edge node{} (a6);

\draw[->]  (a6) edge[bend left = 0] node[right,pos = .15]{\scalebox{.6}{$\dfrac{1}{\sz_1\sz_2}$}} (a1);
\draw[->]  (a4) edge[bend right = 0] node[pos=.75,above]{{\scalebox{.6}{$\sz_1^{-1}$}}} (a3);
\draw[->]  (a5) edge[bend left = 0] node[pos=.75,above]{{\scalebox{.6}{$\sz_2^{-1}$}}} (a2);

\end{tikzpicture}$$

\end{minipage}
\caption{The graph $\Ga_\rho$ in type $\tilde{A_2}$}
\end{figure}
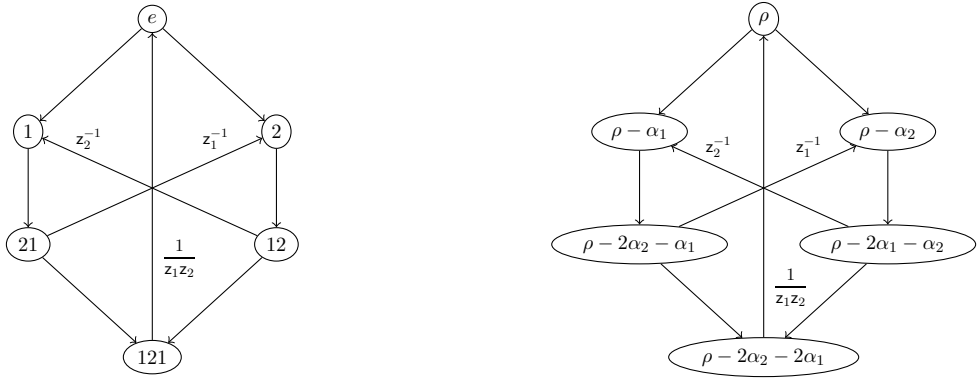
\end{example}

\begin{example}
We draw the graphs $\Ga_{\rho}$ in type $\tilde{G}_2$. Again we set $\sz_1=\sz^{\al_1}$ and $\sz_2=\sz^{\al_2}$ in $\Z[Q^\vee]$   and we keep the same convention as in the previous example.%
%
%
%
%
%
%
%

\begin{figure}[H]$$
\begin{tikzpicture}[scale =.9]
\tikzstyle{vertex}=[inner sep=2pt,minimum size=10pt,ellipse,draw]

\def\h{1.1}
\node[vertex] (a1) at (0,0) {\scalebox{.7}{$e$}};
\node[vertex] (a2) at (-1.5,-\h) {\scalebox{.7}{$1$}};
\node[vertex] (a3) at (1.5,-1*\h) {\scalebox{.7}{$2$}};
\node[vertex] (a4) at (-1.5,-2*\h) {\scalebox{.7}{$21$}};
\node[vertex] (a5) at (1.5,-2*\h) {\scalebox{.7}{$12$}};
\node[vertex] (a6) at (-1.5,-3*\h) {\scalebox{.7}{$121$}};
\node[vertex] (a7) at (1.5,-3*\h) {\scalebox{.7}{$212$}};
\node[vertex] (a8) at (-1.5,-4*\h) {\scalebox{.7}{$2121$}};
\node[vertex] (a9) at (1.5,-4*\h) {\scalebox{.7}{$1212$}};
\node[vertex] (a10) at (-1.5,-5*\h) {\scalebox{.7}{$12121$}};
\node[vertex] (a11) at (1.5,-5*\h) {\scalebox{.7}{$21212$}};
\node[vertex] (a12) at (0,-6*\h) {\scalebox{.7}{$121212$}};

\draw[->] (a1) edge node[above]  {\scb{.7}{$3$}} (a2);
\draw[->] (a2) edge node[left]  {\scb{.7}{$2$}} (a4);
\draw[->] (a4) edge node[left]  {\scb{.7}{$3$}} (a6);
\draw[->] (a6) edge node[left]  {\scb{.7}{$2$}} (a8);
\draw[->] (a8) edge node[left]  {\scb{.7}{$3$}} (a10);
\draw[->] (a10) edge node[below,pos=.3]  {\scb{.7}{$2$}} (a12);

\draw[->] (a1) edge node[above]  {\scb{.7}{$2$}} (a3);
\draw[->] (a3) edge node[right]  {\scb{.7}{$3$}} (a5);
\draw[->] (a5) edge node[right]  {\scb{.7}{$2$}} (a7);
\draw[->] (a7) edge node[right]  {\scb{.7}{$3$}} (a9);
\draw[->] (a9) edge node[right]  {\scb{.7}{$2$}} (a11);
\draw[->] (a11) edge node[below,pos=.3]  {\scb{.7}{$3$}} (a12);

\draw[->]  (a12) edge[bend left = 90] node[left,pos = .5]{\scalebox{.6}{$\dfrac{1}{\sz_1\sz^2_2}$}} (a2);
\draw[->]  (a7) edge[bend right = 0] node[above,pos = .75]{\scalebox{.6}{$\sz_2^{-1}$}} (a4);
\draw[->]  (a9) edge[bend left = 0] node[above,pos = .15]{\scalebox{.6}{$\sz_2^{-1}$}} (a6);
\draw[->] (a11) edge[bend right = 90] node[right,pos = .5]{\scalebox{.6}{$\dfrac{1}{\sz_1\sz^2_2}$}} (a1);
\draw[->] (a8) edge[bend right = 00] node[above=.2,pos = .75]{\scalebox{.6}{$\dfrac{1}{\sz_1\sz_2}$}} (a3);
\draw[->] (a10) edge[bend right = 00] node[below,pos = .25]{\scalebox{.6}{$\dfrac{1}{\sz_1\sz_2}$}} (a5);

\end{tikzpicture}$$
\caption{The graph $\Ga_\rho$ in type $\tilde{G}_2$}
\end{figure}
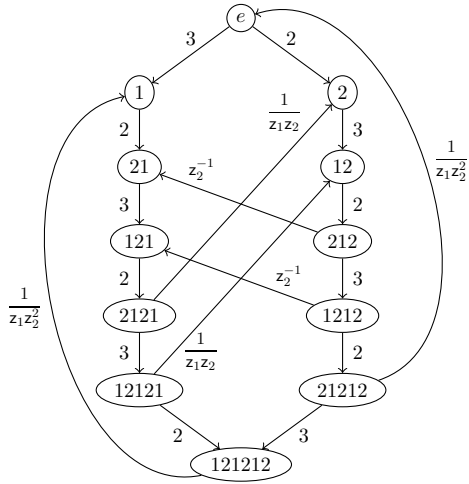
\end{example}

All the arrows in the graph $\Ga_\rho$ are of the form $a_i\sz^{\al}$ where $\al$ can be equal to $0$. Recall that $\al\neq 0$ if and only if $i=0$ and in this case, we have $a_0=1$. 
In the following proposition, we show how one can determine the sign of the crossing between two adjacents alcoves $w$ and $s_iw$ in $W_a$ by looking at the direction of the arrow between $\sf(w)$ and $\sf(s_iw)$ in the graph $\Ga_\rho$.
\begin{proposition}
\label{Gamma_rho_orientation}
Let $w\in W_a$ and $i\in I$.
\begin{enumerate}
\item If there is an arrow from $\sf(w)$ to $\sf(s_iw)$ of weight $a_i\sz^{\al}$ in $\Ga_\rho$, then the crossing from $w$ to $s_iw$ is negative and $\wt(s_iw) =\wt(w)+\al$. 
\item If there is an arrow from $\sf(s_iw)$ to $\sf(w)$ of weight $a_i\sz^\al$ in $\Ga_\rho$, then the crossing from $w$ to $s_iw$ is positive and $\wt(s_iw) =\wt(w)-\al$. 
\end{enumerate} 
\end{proposition}
\begin{proof}
We only need to prove (1). Let $\beta\in Q^\vee$ and $u\in W$ be such that  $w=ut_{\beta}$. The crossing from $w$ to $s_iw$ is of the same sign as the crossing from $u$ to $s_iu$ since the orientation is left unchanged by translation. In the case where $i\in I^\ast$, we know that the crossing from $u$ to $s_iu$ is positive if and only if $\ell(s_iu)<\ell(u)$. Further, we have $\wt(s_iw) = \wt(w)$ and $\al=0$ so that $\wt(s_iw) =\wt(w)+\al$.
 Assume now that $i=0$. The alcoves $u$ and $s_0u$ are separated by the hyperplane $(H_{\theta,1})u = H_{\theta u,1}$. We have 
\begin{align*}
u\overset{-}{\leadsto}s_0u &\iff  (\theta)u\in \Phi^-\\
&\iff \ell(s_\theta u)<\ell(u)\\
&\iff \text{there is an arrow from $u$ to $s_\theta u$ in $\Ga_\rho$}.
\end{align*}
The last equivalence comes from the fact that $\sf(u)=u$ and $\sf(s_0u) =s_\theta u$. We prove the equality about the weight. By definition of $\Ga_\rho$ we have $\wt(s_0u) = \al$ and we get $\wt(s_0w) = \wt(s_0ut_\beta) = \al+\beta$ as required. 
\end{proof}
Let $\pi$ be a (undirected) path on the graph $\Ga_\rho$, that is, $\pi$ is a sequence $u_0,u_1,\ldots,u_N$ in $W$ such that for all $k\in \{1,\ldots,N\}$ there is an arrow from $u_{k-1}$ to $u_k$ or from $u_k$ to $u_{k-1}$ ($\pi$ does not have to follow the directions of the arrows). We define the word of $\pi$ to be $(i_1,\ldots,i_N)$ where $i_k$ is the type of the arrow between $u_{k-1}$ and $u_k$ and the group element of $\pi$ to be  $s_{i_1}\cdots s_{i_N}\in W_a$. Obviously, two different paths can have the same group element. Conversely, to any $w\in W$ and any expression (not necessarily reduced) $w=s_{i_1}\ldots s_{i_n}$, we can associate a path $\pi_{\vec{w}}$  in $\Ga_\rho$  that starts at $e$ and follows the arrow of type $i_1,\ldots,i_N$. 

\smallskip

With this in mind, the graph $\Ga_\rho$ can be interpreted as an automaton for the set of affine Grassmanians. This is a consequence of the following well known fact on affine Grassmanians. Let $w\in W_a$ such that $w=s_{i_1}\ldots s_{i_N}$.  Then $w\in W_a^{\La_0}$ and the expression $s_{i_1}\ldots s_{i_N}$ is reduced if and only if the crossing from $s_{i_{k}}\ldots s_{i_N}$ to $s_{i_{k+1}}s_{i_k}\ldots s_{i_N}$ is positive for all $1\leq k\leq N-1$. Thus, $w\in W_a^{\La_0}$ if and only if the path~$\pi_{\vec{w}}$ (starting at $e$) corresponding to the expression $s_{i_1}\ldots s_{i_N}$ of $w$ only follows arrows in the opposite direction. 

\begin{example}
\label{automate_G2}
We continue the previous example in type $\tilde{G}_2$. In order to picture the automaton more easily, we start at the identity (the black vertex) and we reverse the orientation of the arrows. We do not indicate the weight of the arrows (it doesn't play any role here) but we put the generator $s_i$ when an arrow is of type $i$. One can obtain all reduced expressions of the affine Grassmanian elements simply by considering all the oriented paths starting at $e$ and by concatenating the type of the arrows from right to left. For instance, one can check that the translations by the fundamental weights $t_{\om_1},t_{\om_2}$ are in $W_a^{\La_0}$ and that the expressions $t_{\om_1} =s_1s_2s_1s_2s_0s_1s_2s_1s_2s_0$ and $t_{\om_2}=s_2s_1s_2s_1s_2s_0$ are reduced. 
\begin{figure}[H]$$
\begin{tikzpicture}[scale=.65]
\tikzstyle{vertex}=[inner sep=2pt,minimum size=5pt,circle,draw]
\def\rs{0.7}
\node[vertex,fill=black] (a1) at (0,0) {};
\node[vertex] (a2) at (0,-1)   {\scalebox{.7}{}};
\node[vertex] (a3) at (0,-2)   {\scalebox{.7}{}};
\node[vertex] (a4) at (0,-3) {\scalebox{.7}{}};
\node[vertex] (a5) at (0,-4)   {\scalebox{.7}{}};
\node[vertex] (a6) at (-1,-5) {\scalebox{.7}{}};
\node[vertex] (a7) at (1,-5)  {\scalebox{.7}{}};
\node[vertex] (a8) at (0,-6)  {\scalebox{.7}{}};
\node[vertex] (a9) at (0,-7)  {\scalebox{.7}{}};
\node[vertex] (a10) at (0,-8) {\scalebox{.7}{}};
\node[vertex] (a11) at (0,-9)  {\scalebox{.7}{}};
\node[vertex] (a12) at (0,-10)  {\scalebox{.7}{}};

\draw[->] (a1) to node[left] {\scalebox{\rs}{$s_0$}} (a2);
\draw[->] (a2) to node[left] {\scalebox{\rs}{$s_2$}}  (a3);
\draw[->] (a3) to node[left] {\scalebox{\rs}{$s_1$}}  (a4);
\draw[->] (a4) to node[left] {\scalebox{\rs}{$s_2$}}  (a5);

\draw[->] (a5) to node[left] {\scalebox{\rs}{$s_0$}}  (a6);
\draw[->] (a5) to node[right] {\scalebox{\rs}{$s_1$}}  (a7);

\draw[->] (a6) to node[left] {\scalebox{\rs}{$s_1$}}  (a8);
\draw[->] (a7) to node[right] {\scalebox{\rs}{$s_0$}}  (a8);

\draw[->] (a8) to node[left] {\scalebox{\rs}{$s_2$}}  (a9);
\draw[->] (a9) to node[left] {\scalebox{\rs}{$s_1$}}  (a10);
\draw[->] (a10) to node[left] {\scalebox{\rs}{$s_2$}}  (a11);
\draw[->] (a11) to node[left] {\scalebox{\rs}{$s_0$}}  (a12);

\draw[->,bend right=55] (a7) to node[left,pos =.6]  {\scalebox{\rs}{$s_2$}}   (a1) ;
\draw[->,bend left=55] (a12) to  node[left,pos =.8] {\scalebox{\rs}{$s_2$}}   (a6);

\draw[->,bend right=65] (a9) to node[right,pos =.25] {\scalebox{\rs}{$s_0$}}(a3);
\draw[->,bend left=65] (a10) to node[left,pos =.75]  {\scalebox{\rs}{$s_0$}} (a4);

\draw[->,bend right=95] (a11) to node[right,pos =.5] {\scalebox{\rs}{$s_1$}} (a1);
\draw[->,bend left=95] (a12) to node[left,pos =.5]  {\scalebox{\rs}{$s_1$}}  (a2);

\end{tikzpicture}
$$
\caption{The graph $\Ga_\rho$ in type $\tilde{G}_2$}
\end{figure}

\end{example}
\newcommand{\wW}{\widetilde{W}}
\subsection{The graphs $\Gamma_\ga$}
\label{gammagamma}

The construction of the previous section can be generalised to any level $0$ weight. The main difference is that such weights can have a non trivial stabiliser in  $W$.  We may assume that $\gamma = \om_{i_1}+\cdots+\om_{i_\ell}\in P$ and we set $J'=\{i_1,\ldots,i_\ell\}$. Indeed, let $\ga$ be a level 0 weight so that~$\ga\in P+\Z\delta$. Since $\delta$ is stable under the action of $W_a$ we may assume that $\ga\in P$. Next, there exists a dominant weight which lies in the $W$-orbit of $\ga$, hence we may assume that $\ga=k_1\om_1+\cdots +k_\ell\om_\ell$ where $(k_1,\ldots,k_\ell)\in \N^\ell$. Now the orbit of $k_1\om_1+\cdots +k_\ell\om_\ell$ under the action of $W$ only depends on whether~$k_i=0$ or not. This is not true for the action of $W_a$ but in the graphs $\Gamma_\ga$ that we will define, we may assume that $\gamma = \om_{i_1}+\cdots+\om_{i_\ell}\in P$ for some $J'=\{i_1,\ldots,i_\ell\}$ just by changing the indeterminates labelling the arrows. 
 
 \smallskip

 So, let $\gamma = \om_{i_1}+\cdots+\om_{i_\ell}\in P$, $J'=\{i_1,\ldots,i_\ell\}$ and $J$ be the complement of $J'$ in $\{1,\ldots,n\}$. Then 
the stabiliser of $\ga$ in $W$ is $W_J$, the group generated by the set $\{s_j,j\in J\}$ and the orbit of $\ga$ under the action of $W$ is in bijection with the set $W^J$ of minimal length representatives of the left cosets $W\slash W_{J}$. 
It is a standard result that any $x\in W$ can be uniquely written as $x^Jx_J$ where $x^J\in W^J$ and $x_J\in W_J$.

\smallskip

Recall that the action of $t_{\beta}$ with $\beta\in Q^\vee$ on $P_{0}$ is given by $t_{\beta}(\gamma)=\gamma-\scald{\gamma}{\beta}\delta$. From there, we see that the
group $\wW_J:= W_{J}\ltimes Q^\vee_{J}$ where 
$$Q^\vee_{J}=\{\beta\in Q^\vee\mid\scald{\gamma}{\beta}=0\} = \bigoplus_{j\in J} \Z\al^\vee_j$$ 
will stabilise $\ga$ (in fact it is strictly contained in the stabiliser of $\ga$ in $W_a$). The group $\wW_J$ is an affine Weyl group.
It is generated by the orthogonal reflections with respect to the hyperplanes $H_{\beta,k}$ with $\beta\in Q^\vee_{J}$ and $k\in \Z$. It acts simply transitively on the connected components of the set
$$\cF_J := V\setminus\{H_{\al,k}\mid\al\in Q_J^\vee,k\in \Z\}.$$
Following Lusztig \cite{Lusz1}, we call these connected components $J$-alcoves. The fundamental $J$-alcove is 
$$\cAJ= \bigcap_{\al\in \Phi_J^+} H_{\al,0}^+\cap H_{\al,1}^-=\{x\in V\mid 0<\scal{x}{\al}<1 \text{ for all $\al\in \Phi^+_J$}\}$$
where $\Phi_J^+\subset \Phi^+$ is the positive root system of $W_J$. The closure of $\cA_J$ is a fundamental domain for the action of~$\widetilde{W}_J$ on $V$, that is,  each $\la\in V$ is conjugate under the action of~$\widetilde{W}_J$ to exactly one element in the closure of $\cAJ$. Given $\beta\in Q^\vee$, we will slightly abuse the notation and write $\beta\in \cAJ$ to mean that~$\beta$ lies in the closure of~$\cAJ$.  Any $x\in \wW_J$ can be uniquely written as $ut_{\al}$ where $u\in W_J$ and $\al\in Q_J^\vee$. Note that when $J=\{1,\ldots,n\}$ we recover the usual alcoves associated to $W_a$ and $\cA_J=A_0$, which explain the terminology of $J$-alcoves. 

\smallskip

The affine Weyl group $\widetilde{W}_J$ and the set of $J$-alcoves have been studied in details in \cite{GLP} in the extended case where the authors studied positively folded alcove paths in $\cAJ$ that bounce on the walls of $\cAJ$.
We also refer to \cite{MST} where the authors study retractions in buildings with respect to chimneys (which are our $J$-alcoves) in order to study orbits in affine flag varieties.

\smallskip

Let $\Z[Q^\vee\slash Q_J^\vee]$ be the group algebra of the quotient $Q^\vee\slash Q_J^\vee$. We write $\ov{\al}$ for the class of $\al\in Q^\vee$ modulo~$Q_J^\vee$ and $\sz^{\ov{\al}}$ for the corresponding element of $\Z[Q^\vee\slash Q_J^\vee]$. We now define our graphs $\Gamma_\ga$ and refer to §\ref{subsec_Examples} for some examples.
\begin{definition}\label{defGgam}
The graph $\Gamma_{\ga}$ is the oriented weighted graph with weights in  $\Z[Q^\vee\slash Q_J^\vee]$ defined as follows
\begin{itemize}
\itemsep=1mm
\item[\tbb] the set of vertices is $W^J$ 
\item[\tbb] there is an arrow $w\underset{a_i}{\overset{i}{\lra}} s_iw$ of type $i$ for $i\in I^{\ast}$ if 
$s_{i}w$ belong to $W^J$ and $\ell(s_iw)=\ell(w)+1$
\item[\tbb] there is an arrow $w\underset{z^{\ov{\al}}}{\overset{0}{\lra}} \sf(s_0w)^J$ of type $0$ if $\ell(\sf(s_0w)^J)<\ell(w)$ and $\wt(s_0w) =\al$. 
\end{itemize}
\end{definition}
When $\ga=\rho$ the graph $\Ga_\rho$ encodes the orbit of $\rho$ and the structure of the semi-direct product $W\ltimes Q^\vee$ together with the periodic orientation of the alcoves in $W_a$. We will show that  the graph~$\Ga_\ga$ encodes the orbit of $\ga$ under the action of $W_a$ and the orientation of the alcoves in $\cAJ$ in a very similar manner.  

\smallskip

It is clear that if we forget about the $0$ arrows then the graph $\Ga_\ga$ encodes the orbit of $\ga$ under the action of $W$ since the stabiliser of $\ga$ is $W_J$. Let $w\in W^J$ and consider the weight $w(\ga)$. We want to show that the action of $s_0$ on $w(\ga)$ is determined by the graph $\Ga_\ga$. First, assume that $\ell(\sf(s_0 w)^J)=\ell(w)$. In this case, since $s_\theta w$ and $w$ are comparable in the Bruhat order, the elements $(s_\theta w)^J = \sf(s_0w)^J$ and $w^J = w$ are also comparable in the Bruhat order, and therefore they must be equal. Further, the equality $(s_\theta w)^J=w$ implies that there exists $v\in W_J$ such that $s_\theta w = wv$ which in turn implies $v = w^{-1}s_\theta w$. This shows that $(\theta^\vee)w\in \Phi_J$. Finally we get 
$$s_0w(\ga) = s_\theta t_{\theta^\vee}w(\ga)= s_\theta w t_{(\theta^\vee)w}(\ga) =  s_\theta w(\ga) =(s_\theta w)^J(\ga) = w(\ga)$$
which explains why there is no arrow of type $0$ leaving from $w$ in this case. For the other cases, we see, doing exactly the same as in the case $\ga=\rho$, that the weights of the $0$ arrow determine the coefficient of $\delta$ in the action of $s_0$ on $w(\ga)$.

\smallskip

We have seen that $W_J\ltimes Q^\vee_J$ stabilises the weight $\ga$ (but is not the full stabilizer of $\ga$) so that to determine the orbit of $\ga$ under $W_a$ it is sufficient to look at the action of  $W_a\slash (W_J\ltimes Q^\vee_J)$. The later  is in bijection with $W^J\times Q^\vee\slash Q^\vee_J$ and we will see that this set is closely related to the fundamental $J$-alcove~$\cAJ$. We also refer to \cite{GLP} for details. 

\begin{lemma}
\label{clU}
Let $\al\in Q^\vee$ and $x\in W_a$. 
\begin{enumerate}
\item There exists a unique element $\beta\in \overline{\alpha}$  such that  $\beta\in \cAJ$. 
\item There exists a unique element $w\in\widetilde{W}_J$ such that $xw\in \cAJ$. We write $p_J(x) = xw$.
\item If $\al\in \cAJ$, there exists a unique element $v\in W_J$ such that $vt_{\al}\in \cAJ$. We write $v = v_{J}(\al)$.
\end{enumerate}
\end{lemma}
\begin{proof}
(1) Let $\al\in Q^\vee$.  There exists a unique element $\beta$ in $\cAJ$ in the orbit of $\al $ under the action of $\widetilde{W}_J$. Since $(\al)\si_{\kappa,k}$ is equal to $\al$ plus a multiple of $\kappa$, we see that $\beta$ is equal to $\al$ plus a linear combination of roots in $\Phi_J$, so that $\beta\equiv \al \mod Q_J^\vee$, that is $\beta\in \ov{\al}$. 

\smallskip

(2) The alcove $x$ lies in a connected component $\cU$ of $V\setminus\{H_{\al,k}\mid\al\in Q_J^\vee,k\in \Z\}$ and there exists a unique element $w\in \widetilde{W}_J$ that sends this component to $\cAJ$.

\smallskip

(3) Assume that $\al\in \cAJ$. 
Let $w \in \widetilde{W}_J$ be such that $t_{\al}w\in \cAJ$. Write $w=vt_{\beta}$ where $v\in W_J$ and $\beta\in Q_J^\vee$.  Then $t_{\al}w = vt_{(\al)v+\beta}\in \cAJ$. Note that since $\al\in \cAJ$, the element $w$ is a product of reflections with respect to hyperplanes going through $\al$, hence $\wt(t_{\al}w) = 
\al$  and this forces the equality $(\al)v+\beta = \al$.
\end{proof}
Let $x\in \cAJ$ and $\al\in Q^\vee$. According to the second point of the lemma, there exists a unique $w\in \widetilde{W}_J$ such that $xt_\al w = p_J(xt_\al)\in \cAJ$. It is easy to see that the element $p_J(xt_\al)$ only depends on $x$ and the class of $\al$ modulo $Q_J^\vee$. Indeed, if $\beta = \al + \nu_J$ with $\nu_J\in Q_J^\vee$, then $t_{-\nu_J}w\in \tilde{W_J}$ and $xt_\beta t_{-\nu_J}w = xt_\al w\in\cAJ$, hence showing that $p_J(xt_\beta) = p_J(xt_\al)$. We write $x\add t_{\ov{\al}} = p_J(xt_\al)$.
\begin{proposition}
The map 
$$\begin{array}{ccccccc}
\cAJ\times Q^\vee\slash Q_J^\vee & \to & \cAJ\\
(u,\ov{\al})&\mapsto&  u\add t_{\ov{\al}}
\end{array}$$
defines a right group action of $Q^\vee\slash Q_J^\vee$ on the fundamental $J$-alcove $\cAJ$.
\end{proposition}
\begin{proof}
We need to show that for all $u\in \cAJ$ we have  $\left(u\add t_{\ov{\beta}}\right)\add t_{\ov{\al}} = u\add t_{\ov{\al+\beta}}$. Fix $u\in \cAJ$. There exists a unique $x\in \widetilde{W}_J$ such that $u\add t_{\ov{\beta}} =  ut_\beta x\in \cAJ$. Write $x=vt_{\ga}$ with $v\in W_J$ and $\ga\in Q^\vee_J$. The root $v(\al)$ is congruent to $\al$ modulo~$Q^\vee_J$. There exists a unique $y\in \widetilde{W}_J$ such that $(u\add t_{\ov{\beta}})t_{{v(\al)}}y\in \cU_{J}$. We have 
$$\left(u\add t_{\ov{\beta}}\right)\add t_{\ov{\al}} = \left(u\add t_{\ov{\beta}}\right)\add t_{\ov{v(\al)}} =ut_\beta xt_{{v(\al)}}y= ut_{\beta+\al} xy\in \cAJ.$$
By definition, there exists a unique $z\in \widetilde{W}_J$ such that $u\add t_{\ov{\al+\beta}} = u t_{\al+\beta}z\in \cAJ$.  
This shows that $yx = z$ and $\left(u\add t_{\ov{\beta}}\right)\add t_{\ov{\al}} = u\add t_{\ov{\al+\beta}}$ as required. \end{proof}

\begin{theorem}
\label{UJ}
The map $W^J\times Q^\vee\slash Q^\vee_J\to \cAJ$ defined by $(u,\ov{\al}) \mapsto u\add t_{\ov{\al}}$ is a bijection. \end{theorem}
\begin{proof}
Let $x\in \cAJ$. Let $x=uvt_{\al}$ where $u\in W^J$, $v\in W_J$ and $\al\in \cAJ$ since $x\in \cAJ$. The alcoves $vt_{\al}$ and $uvt_{\al}$ lie in the same connected component of $\cF_J$, therefore $vt_{\al}\in \cAJ$.  This forces  $v = v_J(\al)$ by the previous lemma.  This shows that
$$\cAJ\subset \{uv_J(\al)t_{\al}\mid \al\in \cAJ,u\in W^J\}.$$
By definition, we have $v_J(\al)t_\al\in \cAJ$ and  this implies that $uv_J(\al)t_\al\in \cAJ$ for all $u\in W^J$, showing the reverse inclusion. Let $\al\in \cAJ$ and $u\in W^J$. We have 
 $uv_{J}(\al)t_{\al} = u t_{\beta}v_{J}(\al)\in\cAJ$ where $\beta\equiv \al$ modulo~$Q_J^\vee$.  It follows that $u\add t_{\ov{\beta}} = u\add t_{\ov{\al}} = uv_{J}(\al)t_{\al}$ hence showing that the map of the theorem is surjective. \\
Let $u,u'\in W^J$ and $\al,\beta\in Q^\vee$ such that 
$u\add t_{\ov{\al}} = u'\add t_{\ov{\beta}}$. Let $\al_0,\beta_0\in \cAJ$ be such that $\al_0\equiv \al\mod Q_J^\vee$ and $\beta_0\equiv \beta\mod Q_J^\vee$. We know that $uv_J(\al_0)t_{\al_0} = u\add t_{\ov{\al}} = u'\add t_{\ov{\beta}} = u'v_J(\beta_0)t_{\beta_0}$ which implies that $\al_0=\beta_0$ since $uv_J(\al_0)$ and $u'v_J(\beta_0)$ lie in $W$. In particular we have $\al\equiv \beta \mod Q_J^\vee$. Now $uv_J(\al_0)=u'v_J(\beta_0)$ which implies that $u=u'$ and $v_J(\al_0) = v_J(\beta_0)$ since $u,u'\in W^J$ and $v_J(\al_0),v_J(\beta_0)\in W_J$. The map of the theorem is injective. 
\end{proof}
In the following theorem, we show how, given $w\in \cAJ$ and $s_i\in S$, one can determine whether the alcove~$s_iw$ lies in the $J$-alcove $\cAJ$ by looking at the graph $\Ga_\ga$. Further we show that the sign of the crossing from $w$ to $s_iw$ (when they are both in $\cAJ$) is also determined by  $\Ga_\ga$. 

\smallskip

As in Proposition \ref{Gamma_rho_orientation}, we allow the weight $\ov{\al}$ to be equal to $0$. This allows us to treat the cases $i\in I^\ast$ and $i=0$ together.

\begin{theorem}
\label{cross_strip}
Let $w\in \cAJ$ and $i\in I$. We have $s_iw\in \cAJ$ if and only if there is an arrow of type $i$ adjacent to $\sf(w)^J$ in $\Ga_\ga$. In this case 
\begin{enumerate}
\item if $\sf(w)^J\underset{a_iz^{\ov{\al}}}{\overset{i}{\lra}} \sf(s_iw)^J$ then $\ov{\wt(s_iw)}=\ov{\wt(w)+\al}$ and the crossing is negative,
\item $\sf(s_iw)^J\underset{a_iz^{\ov{\al}}}{\overset{i}{\lra}} \sf(w)^J$ then $\ov{\wt(s_iw)}=\ov{\wt(w)-\al}$ and the crossing is positive.
\end{enumerate}
\end{theorem}
\begin{proof}
Write $w =uvt_{\beta}$ where $\beta\in \cAJ$, $v=v_J(\beta)$ and $u=\sf(w)^J\in W^J$.  
Let $H_{\al,k}$ be the hyperplane separating the alcoves  $w = u vt_{\beta}$ and $s_iw= s_iu vt_{\beta}$. Then the hyperlane separating  $u$ and $s_iu$ is of direction~$(\al)v^{-1}$. Since $v\in W_J$, we have $\al\in \Phi_J$ if and only if $(\al)v^{-1}\in \Phi_J$. Finally, we get that 
$$s_iw\in \cAJ\iff \al \notin \Phi_J\iff (\al)v^{-1}\notin \Phi_J \iff s_iu\in \cAJ.$$
Assume first that $i\in I^\ast$. Then we have $s_iu\in \cAJ$ if and only if $s_iu\in W^J$. By definition of $\Ga_\ga$, there is an arrow between $u$ and $s_iu$ if and only if $s_iu\in W^J$. In this case, we obtain that $\sf(s_iw)^J=s_iu$ and 
there is an arrow $u\underset{a_i}{\overset{i}{\lra}} s_iu$ when $s_iu>u$ in which case the crossing is negative and  there is an arrow $s_iu\underset{a_i}{\overset{i}{\lra}} u$ when $s_iu<u$ in which case the crossing is positive. 
Since $\wt(s_iw) = \wt(w)$, this proves (1) when~$i\in I^\ast$.

\smallskip

Assume that $i=0$. The hyperplane separating $u$ and $s_0u$ is $H_{(\theta)u,1}$. It follows that
\begin{align*}
s_0u\notin \cAJ &\iff (\theta)u\in \Phi_J 
\iff \exists v\in W_J, u^{-1}s_\theta u = v
\iff \exists v\in W_J, s_\theta u = uv
\iff (s_\theta u)^J = u.
\end{align*}
Assume that $s_0w\in \cAJ$. Then $s_0u\in \cAJ$ and $(s_\theta u)^J \neq  u$. But since the elements $(s_\theta u)^J$ and $u^J=u$ are comparable in the Bruhat order this implies that  $\ell((s_\theta u)^J) \neq  \ell(u)$ and therefore there is an arrow between $u=\sf(w)^J$ and $(s_\theta u)^J$.
Conversely assume that there is an arrow of type $0$ adjacent to $u$ in $\Ga_\ga$. Then $\ell((s_\theta u)^J) \neq  \ell(u)$
which implies that $(s_\theta u)^J \neq  u$ and therefore $s_0u\in \cAJ$ as required. 

\medskip

Assume that there is an arrow $u \underset{z^{\ov{\al}}}{\overset{0}{\lra}} s_\theta u$, that is we have $\ell((s_\theta u)^J)<\ell(u)$. Then  $s_\theta u<u$ and we have seen in the proof of Proposition \ref{Gamma_rho_orientation} that this implies that the crossing from $u$ to $s_0u$ is negative. Finally since $\al = \wt(s_0u)$ we have 
$$\wt(s_0w) = \wt(s_0uvt_\beta) = \wt(s_\theta u t_{\al}vt_\beta) = \beta+v(\al) = \wt(w)+v(\al)$$
and since $v\in W_J$, we have $\ov{v(\al)}=\ov{\al}$ and $\ov{\wt(s_0w)}  = \ov{\wt(w) + \al}$ as required. This prove (1) in the case $i=0$. The proof of (2) is similar. 
\end{proof}
\begin{corollary}
\label{crossUJ}
Let $w,u\in \cAJ$ be two adjacent alcoves and $\al\in Q^\vee$.  We have $w\crossp u$ if and only if $w\add t_{\ov{\al}}\crossp u\add t_{\ov{\al}}$
\end{corollary}
\begin{proof}
According to the previous theorem, the sign of the crossing from $w$ to $u$ is determined by the element $\sf(w)^J$ and $\sf(u)^J$. Similarly, the sign of the crossing from $w\add t_{\ov{\al}}$ to $u\add t_{\ov{\al}}$ is determined by the element $\sf(w\add t_{\ov{\al}})^J$ and $\sf(u\add t_{\ov{\al}})^J$. There exists $x,y\in \widetilde{W}_J$ such that $w\add t_{\ov{\al}}=wt_\al x$ and $u\add t_{\ov{\al}} = ut_\al y$. Writing $x=v_1t_{\beta_1}$ and $y=v_2t_{\beta_2}$ with $v_1,v_2\in W_J$ and $\beta_1,\beta_2\in Q_J$, we see that $w\add t_{\ov{\al}} = wv_1t_{v_1(\al)+\beta_1}$ and $u\add t_{\ov{\al}} = uv_2t_{v_2(\al)+\beta_2}$. Thus we get $\sf(w\add t_{\ov{\al}})^J = (wv_1)^J=w^J$ and $\sf(u\add t_{\ov{\al}})^J = (uv_2)^J=u^J$ and the result.  
\end{proof}
As in the case of $\ga = \rho$, the graph $\Ga_\ga$ can be interpreted as an automaton for reduced expression of the elements lying in $\cAJ\cap W_a^{\La_0}$. Indeed, we have $w\in \cAJ\cap W_a^{\La_0}$ and the expression $s_{i_1}\ldots s_{i_N}$ is reduced if and only if $s_{i_{k}}\ldots s_{i_N}\in \cAJ$ for all $1\leq k\leq N$ and if the crossing from $s_{i_{k}}\ldots s_{i_N}$ to $s_{i_{k+1}}s_{i_k}\ldots s_{i_N}$ is positive for all $1\leq k\leq N-1$. Thus, in terms of the graph $\Ga_\ga$, we have $w\in \cAJ\cap W_a^{\La_0}$ if and only if there exists a path~$\pi_{\vec{w}}$ in $\Ga_\ga$ starting at $e$ and following reversed arrows of type $i_1,\ldots,i_N$. We have given two examples of such automaton in Example \ref{strip_G2}.

\subsection{Some examples of $\Ga_\ga$}
\label{subsec_Examples}
We now give a series of examples of graphs $\Ga_\ga$ with a drawing of the corresponding strips in type $\tilde{G}_2$.

\begin{example}
Let $W$ be of type $\tilde{A}_3$ and $\gamma = \om_1$. The stabiliser of $\ga$ in $W$ is the parabolic subgroup generated by $J = \{s_2,s_3\}$ and the set of minimal length representatives is 
$W^J = \{e,s_1,s_2s_1,s_3s_2s_1\}$.
We have $Q_J^\vee =\Z\al_2^\vee +\Z\al_3^\vee$ and the labels of the graph $\Ga_\ga$ are in the group algebra of~$Q^\vee\slash Q^\vee_J=\Z \ov{\al_1}$. We represent the graph $\Gamma_\ga$ below where we have set $\sz_1 = \sz^{\ov{\al_1^\vee}}$. We omit the weights $a_1,a_2,a_3$ since they are all equal to $1$. 
\begin{figure}[H]
$$
\begin{tikzpicture}[scale =1]
\tikzstyle{vertex}=[inner sep=2pt,minimum size=10pt,ellipse,draw]

\node[vertex] (a1) at (0,0) {\scalebox{.7}{$e$}};
\node[vertex] (a2) at (0,-1) {\scalebox{.7}{$s_1$}};
\node[vertex] (a3) at (0,-2) {\scalebox{.7}{$s_2s_1$}};
\node[vertex] (a4) at (0,-3) {\scalebox{.7}{$s_3s_2s_1$}};

\draw[->] (a1) edge node  {} (a2);
\draw[->] (a2) edge node{} (a3);
\draw[->] (a3) edge node{} (a4);
\draw[->] (a4) edge[bend left = 90,left] node{$\frac{1}{\sz_1}$} (a1);

\end{tikzpicture}$$
\caption{The graph $\Ga_{\om_1}$ in type $\tilde{A}_3$}
\end{figure}

Next consider the weight $\gamma = \om_2$. 
 The stabiliser of $\ga$ in $W$ is the parabolic subgroup generated by $J = \{s_1,s_3\}$ and the set of minimal length representatives is $W^J = \{e,s_2, s_1s_2,s_3s_2,s_1s_3s_2,s_2s_1s_3s_2\}$.
We have $Q_J^\vee =\Z\al_1^\vee +\Z\al_3^\vee$ and the labels of the graph $\Ga_\ga$ are in the group algebra of~$Q^\vee\slash Q^\vee_J=\Z \ov{\al_2}$. We represent the graph $\Gamma_\ga$ below where we have set $\sz_2 = \sz^{\ov{\al_2^\vee}}$. Again we omit the weights $a_i$'s.

\begin{figure}[H]
$$
\begin{tikzpicture}[scale =1]
\tikzstyle{vertex}=[inner sep=2pt,minimum size=10pt,ellipse,draw]

\node[vertex] (a1) at (0,0) {\scalebox{.7}{$e$}};
\node[vertex] (a2) at (0,-1) {\scalebox{.7}{$s_2$}};
\node[vertex] (a3) at (-.7,-2) {\scalebox{.7}{$s_1s_2$}};
\node[vertex] (a4) at (.7,-2) {\scalebox{.7}{$s_3s_2$}};
\node[vertex] (a5) at (0,-3) {\scalebox{.7}{$s_1s_3s_2$}};
\node[vertex] (a6) at (0,-4) {\scalebox{.7}{$s_2s_1s_3s_2$}};

\draw[->] (a1) edge node  {} (a2);
\draw[->] (a2) edge node{} (a3);
\draw[->] (a2) edge node{} (a4);
\draw[->] (a3) edge node{} (a5);
\draw[->] (a4) edge node{} (a5);
\draw[->] (a5) edge node{} (a6);

\draw[->]  (a5) edge[bend left = 90] node[left]{$\frac{1}{\sz_2}$} (a1);
\draw[->]  (a6) edge[bend right = 90] node[right]{$\frac{1}{\sz_2}$} (a2);

\end{tikzpicture}$$
\caption{The graph $\Ga_{\om_2}$ in type $\tilde{A}_3$}
\end{figure}
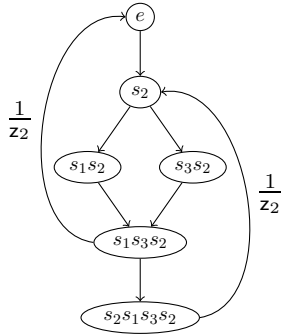
\end{example}

\begin{example}
\label{strip_G2}
Let $W$ be of type $\tilde{G}_2$ as in Example \ref{G2}. We first consider the level $0$-weight  $\ga = \om_2$. The stabiliser of $\ga$ in $W$ is the parabolic subgroup generated by $J = \{s_1\}$ and the set of minimal length representatives is $W^J =\{e,s_2,s_1s_2,s_2s_1s_2,s_1s_2s_1s_2,s_2s_1s_2s_1s_2\}$.
We have $Q_J^\vee = \Z\al_1^\vee$ and the labels of $\Ga_\ga$ are in the group algebra of~$Q^\vee\slash Q^\vee_J = \Z\ov{\al_2}$. We represent the graph $\Gamma_\ga$ below on the left hand-side where we have set $\sz_2 = \sz^{\ov{\al_2^\vee}}$. In the middle, we represent the strip $\cAJ$. The subsets in green are from top to bottom $W^J\add t_{\ov{2\al_2}},W^J$ and $W^J\add t_{-\ov{2\al_2}}$ and the subsets in violet are from top to bottom $W^J\add t_{\ov{\al_2}},W^J\add t_{-\ov{\al_2}}$ and~$W^J\add t_{-\ov{3\al_2}}$. Finally, on the right handside, we represent the automaton for reduced expressions in $\cAJ\cap W_a^{\La_0}$. We keep the same convention as in Example \ref{automate_G2} (we have inverted the direction of the arrows and the identity is the black vertex).  

\begin{figure}[H]
\begin{minipage}{5cm}
$$
\begin{tikzpicture}[scale =.8]
\tikzstyle{vertex}=[inner sep=2pt,minimum size=10pt,ellipse,draw]

\node[vertex] (a1) at (0,0) {\scalebox{.7}{$e$}};
\node[vertex] (a2) at (0,-1) {\scalebox{.7}{$s_2$}};
\node[vertex] (a3) at (0,-2) {\scalebox{.7}{$s_1s_2$}};
\node[vertex] (a4) at (0,-3) {\scalebox{.7}{$s_2s_1s_2$}};
\node[vertex] (a5) at (0,-4) {\scalebox{.7}{$s_1s_2s_1s_2$}};
\node[vertex] (a6) at (0,-5) {\scalebox{.7}{$s_2s_1s_2s_1s_2$}};

\draw[->] (a1) edge node[left]{\scalebox{.7}{2}} (a2);
\draw[->] (a2) edge node[left]{\scalebox{.7}{3}} (a3);
\draw[->] (a3) edge node[left]{\scalebox{.7}{2}} (a4);
\draw[->] (a4) edge node[left]{\scalebox{.7}{3}} (a5);
\draw[->] (a5) edge node[left]{\scalebox{.7}{2}} (a6);

\draw[->]  (a6) edge[bend left = 90] node[left]{$\frac{1}{\sz_2^2}$} (a1);
\draw[->]  (a4) edge[bend right = 90] node[right]{$\frac{1}{\sz_2}$} (a2);
\draw[->]  (a5) edge[bend left = 90] node[left]{$\frac{1}{\sz_2}$} (a3);

\end{tikzpicture}$$
\end{minipage}
\begin{minipage}{7cm}
\begin{center}
\begin{tikzpicture}[scale =.8]
\newcommand{\y}{1.5}
\newcommand{\x}{.866}

\draw[fill = Green!70!] (1/2*\x,1/2*\y) -- (0,1) -- (-\x,.5) -- (-\x,-.5) -- (-1/2*\x,-.75);
\draw[fill = Green!20!] (1/2*\x-\x,1/2*\y - \y) -- (0-\x,1-\y) -- (-\x-\x,.5-\y) -- (-\x-\x,-.5-\y) -- (-1/2*\x-\x,-.75-\y);
\draw[fill = Green!20!] (1/2*\x+\x,1/2*\y + \y) -- (0+\x,1+\y) -- (-\x+\x,.5+\y) -- (-\x+\x,-.5+\y) -- (-1/2*\x+\x,-.75+\y);

\draw[fill = Blue!50!] (-3/2*\x,.75) -- (-\x,.5) -- (-\x,-.5) -- (-2*\x,-1) -- (-5/2*\x,-.75);
\draw[fill = Blue!20!] (-3/2*\x+\x,.75+\y) -- (-\x+\x,.5+\y) -- (-\x+\x,-.5+\y) -- (-2*\x+\x,-1+\y) -- (-5/2*\x+\x,-.75+\y);
\draw[fill = Blue!20!] (-3/2*\x-\x,.75-\y) -- (-\x-\x,.5-\y) -- (-\x-\x,-.5-\y) -- (-2*\x-\x,-1-\y) -- (-5/2*\x-\x,-.75-\y);

\draw[fill = gray] (0,0) -- (0,1) -- (\x/2,.75);

\draw[line width = .4mm] (-2*\x,-2*\y) -- (2*\x,2*\y);
\draw[line width = .4mm] (-4*\x,-2*\y) -- (0*\x,2*\y);
\draw[->,line width = .5mm] (0,0) -- (2*\x,0);
\draw[->,line width = .5mm] (0,0) -- (\x,\y);
\draw[->,line width = .5mm] (0,0) -- (-\x,\y);

\draw[->,line width = .5mm] (0,0) -- (3*\x,\y);
\draw[->,line width = .5mm] (0,0) -- (0,2*\y);
\draw[->,line width = .5mm] (0,0) -- (-3*\x,\y);

\node at (-.5,0) {\scalebox{.8}{$W^\ga$}};
\node at (-.3+\x,.3+\y) {\scalebox{.5}{$\ov{2\al_2}$}};
\node at (-.3-\x,.3-\y) {\scalebox{.5}{$t_{-\ov{2\al_2}}\cdot W^\ga$}};


\draw (-3.464,3) -- (3.464,3);
\draw (-3.464,1.5) -- (3.464,1.5);
\draw (-3.464,0) -- (3.464,0);
\draw (-3.464,-1.5) -- (3.464,-1.5);
\draw (-3.464,-3) -- (3.464,-3);

\foreach \n in {-4,...,4} {
\draw (\n * \x,-3) -- (\n * \x,3);}

\foreach \e in {-1,1} {
\draw (\e*2*\x,3) -- (\e*4*\x,2);
\draw (\e*0*\x,3) -- (\e*4*\x,1);
\draw (\e*-2*\x,3) -- (\e*4*\x,0);
\draw (\e*-4*\x,3) -- (\e*4*\x,-1);

\draw (-\e*4*\x,2) -- (\e*4*\x,-2);
\draw (-\e*4*\x,1) -- (\e*4*\x,-3);

\draw (-\e*4*\x,0) -- (\e*2*\x,-3);
\draw (-\e*4*\x,-1) -- (\e*0*\x,-3);
\draw (-\e*4*\x,-2) -- (-\e*2*\x,-3);}

\foreach \e in {-1,1} {
\draw (\e*2*\x,3) -- (\e*4*\x,0);
\draw (\e*0*\x,3) -- (\e*4*\x,-3);
\draw (-\e*2*\x,3) -- (\e*2*\x,-3);
\draw (-\e*4*\x,3) -- (\e*0*\x,-3);
\draw (-\e*4*\x,0) -- (-\e*2*\x,-3);}
\end{tikzpicture}
\end{center}
\end{minipage}
\begin{minipage}{5cm}
$$
\begin{tikzpicture}[scale =.8]
\tikzstyle{vertex}=[inner sep=2pt,minimum size=5pt,circle,draw]
\def\rs{0.7}

\node[vertex,fill=black] (a1) at (0,0) {};
\node[vertex] (a2) at (0,-1)  {};
\node[vertex] (a3) at (0,-2)  {};
\node[vertex] (a4) at (0,-3)  {};
\node[vertex] (a5) at (0,-4)  {};
\node[vertex] (a6) at (0,-5) {};

\draw[->] (a1) edge node[left]{\scalebox{\rs}{$s_0$}} (a2);
\draw[->] (a2) edge node[left]{\scalebox{\rs}{$s_2$}}  (a3);
\draw[->] (a3) edge node[left]{\scalebox{\rs}{$s_1$}}  (a4);
\draw[->] (a4) edge node[left]{\scalebox{\rs}{$s_2$}} (a5);
\draw[->] (a5) edge node[left]{\scalebox{\rs}{$s_1$}} (a6);

\draw[->]  (a6) edge[bend left = 90] node[left]{\scalebox{\rs}{$s_2$}} (a1);
\draw[->]  (a6) edge[bend right = 90] node[right]{\scalebox{\rs}{$s_0$}} (a4);
\draw[->]  (a5) edge[bend left = 90] node[left]{\scalebox{\rs}{$s_0$}} (a3);

\end{tikzpicture}$$
\end{minipage}
\caption{The graph $\Ga_{\om_2}$, the set $\cA_{\{s_1\}}$ and the associated automaton in type $\tilde{G}_2$}
\end{figure}
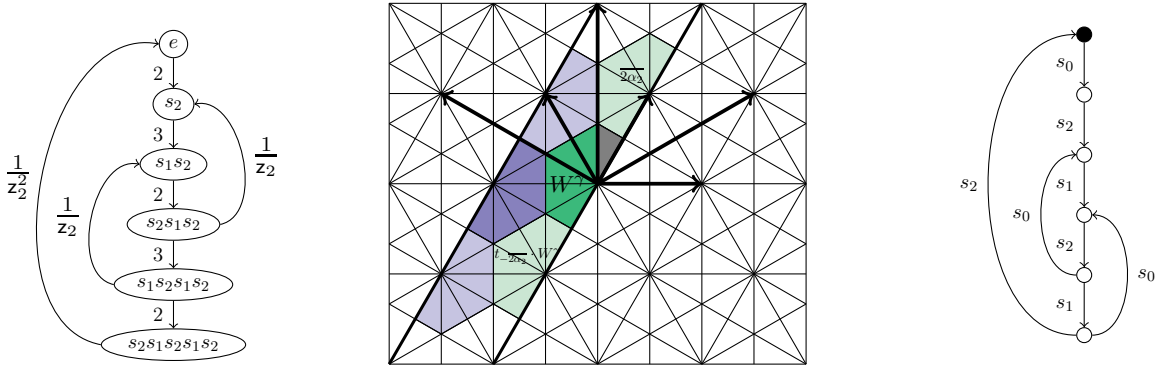

Next consider the weight  $\ga = \om_1$. The stabiliser of $\ga$ in $W$ is the parabolic subgroup generated by $J = \{s_2\}$ and the set of minimal length representatives is:
$W^\ga =\{e,s_1,s_2s_1,s_1s_2s_1,s_2s_1s_2s_1,s_1s_2s_1s_2s_1\}.$
We have $Q_J^\vee = \Z\al_2^\vee$ and the labels of $\Ga_\ga$ are in the group algebra of~$Q^\vee\slash Q^\vee_J=\Z \ov{\al_1}$. We represent the graph $\Gamma_\ga$ below on the left hand-side where we have set $\sz_1 = \sz^{\ov{\al_1^\vee}}$. In the middle, we represent the strip $\cAJ$. The subsets in green are from top to bottom $W^J\add t_{\ov{\al_2}},W^J,W^J\add t_{-\ov{\al_2}}$. Finally, on the right handside, we represent the automaton for reduced expressions in $\cAJ\cap W_a^{\La_0}$.

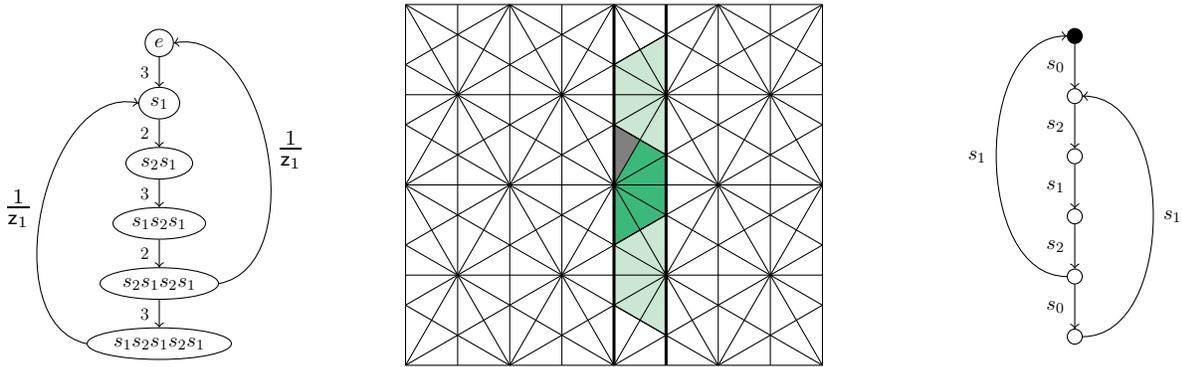
\begin{figure}[H]
\begin{minipage}{5cm}
$$
\begin{tikzpicture}[scale =.8]
\tikzstyle{vertex}=[inner sep=2pt,minimum size=10pt,ellipse,draw]

\node[vertex] (a1) at (0,0) {\scalebox{.7}{$e$}};
\node[vertex] (a2) at (0,-1) {\scalebox{.7}{$s_1$}};
\node[vertex] (a3) at (0,-2) {\scalebox{.7}{$s_2s_1$}};
\node[vertex] (a4) at (0,-3) {\scalebox{.7}{$s_1s_2s_1$}};
\node[vertex] (a5) at (0,-4) {\scalebox{.7}{$s_2s_1s_2s_1$}};
\node[vertex] (a6) at (0,-5) {\scalebox{.7}{$s_1s_2s_1s_2s_1$}};

\draw[->] (a1) edge node[left]{\scalebox{.6}{3}} (a2);
\draw[->] (a2) edge node[left]{\scalebox{.6}{2}} (a3);
\draw[->] (a3) edge node[left]{\scalebox{.6}{3}} (a4);
\draw[->] (a4) edge node[left]{\scalebox{.6}{2}} (a5);
\draw[->] (a5) edge node[left]{\scalebox{.6}{3}}(a6);

\draw[->]  (a6) edge[bend left = 70,out = 90,in=90] node[left]{$\frac{1}{\sz_1}$} (a2);
\draw[->]  (a5) edge[bend right = 70,out =-90,in=-90] node[right]{$\frac{1}{\sz_1}$} (a1);

\end{tikzpicture}$$
\end{minipage}
\begin{minipage}{7cm}
\begin{center}
\begin{tikzpicture}[scale =.8]
\newcommand{\y}{1.5}
\newcommand{\x}{.866}

\draw[fill = Green!70!] (0,1) -- (\x,.5) -- (\x,-.5) -- (0,-1);
\draw[fill = Green!20!] (\x,\y+1) -- (0,\y+.5) -- (0,\y-.5) -- (\x,\y-1);
\draw[fill = Green!20!] (\x,-\y+1) -- (0,-\y+.5) -- (0,-\y-.5) -- (\x,-\y-1);

%
%
\draw[fill = gray] (0,0) -- (0,1) -- (\x/2,.75);

\draw[line width = .4mm] (0,-2*\y) -- (0,2*\y);
\draw[line width = .4mm] (\x,-2*\y) -- (\x,2*\y);


\draw (-3.464,3) -- (3.464,3);
\draw (-3.464,1.5) -- (3.464,1.5);
\draw (-3.464,0) -- (3.464,0);
\draw (-3.464,-1.5) -- (3.464,-1.5);
\draw (-3.464,-3) -- (3.464,-3);

\foreach \n in {-4,...,4} {
\draw (\n * \x,-3) -- (\n * \x,3);}

\foreach \e in {-1,1} {
\draw (\e*2*\x,3) -- (\e*4*\x,2);
\draw (\e*0*\x,3) -- (\e*4*\x,1);
\draw (\e*-2*\x,3) -- (\e*4*\x,0);
\draw (\e*-4*\x,3) -- (\e*4*\x,-1);

\draw (-\e*4*\x,2) -- (\e*4*\x,-2);
\draw (-\e*4*\x,1) -- (\e*4*\x,-3);

\draw (-\e*4*\x,0) -- (\e*2*\x,-3);
\draw (-\e*4*\x,-1) -- (\e*0*\x,-3);
\draw (-\e*4*\x,-2) -- (-\e*2*\x,-3);}

\foreach \e in {-1,1} {
\draw (\e*2*\x,3) -- (\e*4*\x,0);
\draw (\e*0*\x,3) -- (\e*4*\x,-3);
\draw (-\e*2*\x,3) -- (\e*2*\x,-3);
\draw (-\e*4*\x,3) -- (\e*0*\x,-3);
\draw (-\e*4*\x,0) -- (-\e*2*\x,-3);}
\end{tikzpicture}
\end{center}
\end{minipage}
\begin{minipage}{5cm}
$$
\begin{tikzpicture}[scale =.8]
\tikzstyle{vertex}=[inner sep=2pt,minimum size=5pt,circle,draw]
\def\rs{0.7}

\node[vertex,fill=black]  (a1) at (0,0) {};
\node[vertex] (a2) at (0,-1)  {};
\node[vertex] (a3) at (0,-2)  {};
\node[vertex] (a4) at (0,-3)  {};
\node[vertex] (a5) at (0,-4)  {};
\node[vertex] (a6) at (0,-5) {};

\draw[->] (a1) edge node[left]{\scalebox{\rs}{$s_0$}} (a2);
\draw[->] (a2) edge node[left]{\scalebox{\rs}{$s_2$}}  (a3);
\draw[->] (a3) edge node[left]{\scalebox{\rs}{$s_1$}}  (a4);
\draw[->] (a4) edge node[left]{\scalebox{\rs}{$s_2$}} (a5);
\draw[->] (a5) edge node[left]{\scalebox{\rs}{$s_0$}} (a6);

\draw[->]  (a5) edge[bend left = 90] node[left]{\scalebox{\rs}{$s_1$}} (a1);
\draw[->]  (a6) edge[bend right = 90] node[right]{\scalebox{\rs}{$s_1$}} (a2);

\end{tikzpicture}$$
\end{minipage}

\caption{The graph $\Ga_{\om_1}$, the set $\cA_{\{s_2\}}$ and the associated automaton in type $\tilde{G}_2$}
\end{figure}
\end{example}

\subsection{KR graphs in type $A$ and tableau combinatorics}
\label{SubsecKeys}
As mentionned in the introduction, in type~$A_n$, we can associate to any dominant weight $\la = \sum \la_i\om_i\in \mathring{P}$ a finite affine crystal graph $\wB(\la)$. The affine Weyl group acts on $\wB(\la)$ and the image by the weight function of the orbit of the unique vertex of weight $\la$ in $\wB(\la)$  under this action gives the orbit of $\la$ under the action of $W_a$. Hence the graph $\Ga_\la$ also encodes the action of $W_a$ on $\wB(\la)$. 

\medskip

In type $A$, there is a nice and simple combinatorial process that allows one to construct the  graph~$\Ga_\ga$ (without the weights on its arrows).  Let $\ga = \om_{i_1}+\cdots+\om_{i_\ell}\in P_0$ and let $J'=\{i_1,\ldots,i_\ell\}$. We associate to~$\gamma$ the partition $\la_\ga$ whose Young diagram is made with $\ell$ columns of length $i_1,\ldots,i_\ell$. Recall that a semistandard tableau of shape $\la_\ga$ is a filling of $\la_\ga$ by integers between $1$ and $n$ whose columns strickly increase from top to bottom and weakly increase from left to right. We define $T_\ga$ to be the semistandard tableau of shape $\la_\ga$ that contains only $i$'s on the $i$-th row. A key tableau of shape $\la_\ga$ is a semistandard tableau of shape $\la_\ga$ such that the entries in any column $C$ also belong to the column located immediately to the left of $C$, except for the first column (see the figure below).    

\medskip

We can define an action of the affine Weyl group on the set of key tableaux of a given shape as follows.
Let $T$ be a key tableau. If $i\in \{1,\ldots,n\}$ then $s_i\add T$ is the key tableau obtained by changing the entries $i$ into $i+1$ (respectively $i+1$ into $i$) in each column of $T$ which contains $i$ but not $i+1$ (respectively $i+1$ but not $i$). The tableau $s_0\add T$ is the key tableau obtained by changing the entries $n+1$ into $1$  (respectively $1$ but not $n+1$) in the columns of $T$ which contains $n+1$ but not  $1$ (respectively $1$ but not $n+1$), and then by rearranging the columns in increasing order if needed.

\begin{example}
If $\ga=\om_1+\om_3$  in type $\tilde{A}_4$, then the partition $\la_{\ga}$ associated to $\ga$ is made of 2 columns of length $1,3$ and we have
$$\la_{\ga} = \scalebox{.7}{\ydiagram{2,1,1}}\qu{and} T_{\ga} = \scalebox{.7}{\begin{ytableau} 1 &1\\ 2\\3\end{ytableau}}.$$
We have 
$$\scalebox{.7}{\begin{ytableau} 1 &1\\ 2\\3\end{ytableau}} \ \overset{s_1}{\longmapsto}\  
\scalebox{.7}{\begin{ytableau} 1 &2\\ 2\\3\end{ytableau}}
 \ \overset{s_3}{\longmapsto}\  
\scalebox{.7}{\begin{ytableau}  1 &2\\ 2\\4\end{ytableau}}
 \ \overset{s_2}{\longmapsto}\  
\scalebox{.7}{\begin{ytableau} 1 &3\\ 3\\4\end{ytableau}}
 \ \overset{s_1}{\longmapsto}\  
\scalebox{.7}{\begin{ytableau} 2 &3\\ 3\\4\end{ytableau}}
 \ \overset{s_0}{\longmapsto}\  
\scalebox{.7}{\begin{ytableau} 1 &3\\ 2\\3\end{ytableau}}.
 $$
\end{example}

\medskip

The graph $\wB(\ga)$ can be obtain as follows:
\begin{enumerate}
\item[{\tiny $\bullet$}]  the vertices are the key tableaux which form the orbit of $T_{\la_\ga}$ under the action of $W_a$,
\item[{\tiny $\bullet$}] there is an arrow between $T$ and $T'$ of type $i$ if $T'=s_i\add T$ and $T\neq T'$. 
\end{enumerate}
\begin{example}
\label{tableaux}
Let $W$ be of type $\tilde{A}_3$ and $\gamma = \om_1+\om_2$.   The stabiliser of $\ga$ in $W$ is the parabolic subgroup generated by $J = \{s_3\}$ and the set of minimal length representatives is:
$$W^J = \{e,s_1,s_2,s_1s_2,s_2s_1,s_3s_2,s_3s_2s_1,s_2s_1s_2,s_3s_1s_2,s_3s_2s_1s_2,s_2s_3s_1s_2,s_3s_2s_3s_1s_2\}$$
We obtain the following graphs

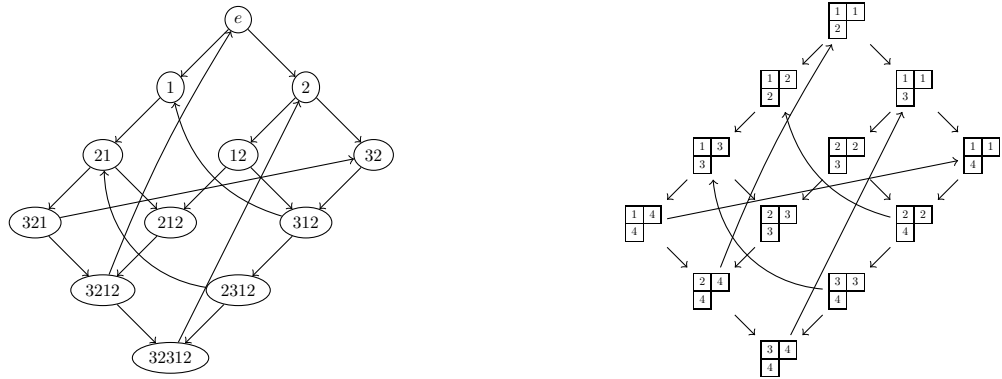
\begin{figure}[H]
\begin{minipage}{8cm}
$$
\begin{tikzpicture}[scale =.9]
\tikzstyle{vertex}=[inner sep=2pt,minimum size=10pt,ellipse,draw]
\node[vertex] (a1) at (0,0) {\scalebox{.6}{$e$}};

\node[vertex] (a2) at (-1,-1) {\scalebox{.6}{$1$}};;
\node[vertex] (a3) at (1,-1) {\scalebox{.6}{$2$}};;

\node[vertex] (a4) at (-2,-2) {\scalebox{.6}{$21$}};;
\node[vertex] (a5) at (0,-2) {\scalebox{.6}{$12$}};;
\node[vertex] (a6) at (2,-2) {\scalebox{.6}{$32$}};;

\node[vertex] (a7) at (-3,-3) {\scalebox{.6}{$321$}};;
\node[vertex] (a8) at (-1,-3) {\scalebox{.6}{$212$}};;
\node[vertex] (a9) at (1,-3) {\scalebox{.6}{$312$}};;

\node[vertex] (a10) at (-2,-4) {\scalebox{.6}{$3212$}};;
\node[vertex] (a11) at (0,-4) {\scalebox{.6}{$2312$}};;

\node[vertex] (a12) at (-1,-5) {\scalebox{.6}{$32312$}};;

\draw[->] (a1) edge node  {} (a2);
\draw[->]  (a1) edge node{} (a3);

\draw[->]  (a2) edge node{} (a4);
\draw[->] (a3) edge node{} (a5);
\draw[->] (a3) edge node{} (a6);

\draw[->]  (a4) edge node{} (a7);
\draw[->]  (a4) edge node{} (a8);

\draw[->]  (a5) edge node{} (a8);
\draw[->]  (a5) edge node{} (a9);

\draw[->]  (a6) edge node{} (a9);

\draw[->]  (a7) edge node{} (a10);

\draw[->]  (a8) edge node{} (a10);

\draw[->]  (a9) edge node{} (a11);

\draw[->]  (a10) edge node{} (a12);
\draw[->]  (a11) edge node{} (a12);

\draw[->] (a10) edge[bend left = 5] node[pos=.3,left]{{\tiny $$}} (a1);
\draw[->] (a9) edge[bend left = 30] node{} (a2);
\draw[->] (a12) edge[bend left = 0] node{} (a3);
\draw[->] (a11) edge[bend right = -40] node{} (a4);
\draw[->] (a7) edge[bend left = 0] node[pos=.8,right]{{\tiny $$}} (a6);
\end{tikzpicture}$$
\end{minipage}
\begin{minipage}{8cm}
$$
\begin{tikzpicture}[scale =.9]
\tikzstyle{vertex}=[inner sep=2pt,minimum size=10pt]
\node[vertex] (a1) at (0,0) {\scalebox{.4}{\begin{ytableau} 1 &1\\ 2\end{ytableau}}};

\node[vertex] (a2) at (-1,-1) {\scalebox{.4}{\begin{ytableau} 1 &2\\ 2\end{ytableau}}};
\node[vertex] (a3) at (1,-1) {\scalebox{.4}{\begin{ytableau} 1 &1\\ 3\end{ytableau}}};

\node[vertex] (a4) at (-2,-2)  {\scalebox{.4}{\begin{ytableau} 1 &3\\ 3\end{ytableau}}};
\node[vertex] (a5) at (0,-2)  {\scalebox{.4}{\begin{ytableau} 2 &2\\ 3\end{ytableau}}};
\node[vertex] (a6) at (2,-2)  {\scalebox{.4}{\begin{ytableau} 1 &1\\ 4\end{ytableau}}};

\node[vertex] (a7) at (-3,-3)   {\scalebox{.4}{\begin{ytableau} 1 &4\\ 4\end{ytableau}}};
\node[vertex] (a8) at (-1,-3) {\scalebox{.4}{\begin{ytableau} 2 &3\\ 3\end{ytableau}}};
\node[vertex] (a9) at (1,-3) {\scalebox{.4}{\begin{ytableau} 2 &2\\ 4\end{ytableau}}};

\node[vertex] (a10) at (-2,-4) {\scalebox{.4}{\begin{ytableau} 2 &4\\ 4\end{ytableau}}};
\node[vertex] (a11) at (0,-4) {\scalebox{.4}{\begin{ytableau} 3 &3\\ 4\end{ytableau}}};

\node[vertex] (a12) at (-1,-5) {\scalebox{.4}{\begin{ytableau} 3 &4\\ 4\end{ytableau}}};

\draw[->] (a1) edge node  {} (a2);
\draw[->]  (a1) edge node{} (a3);

\draw[->]  (a2) edge node{} (a4);
\draw[->] (a3) edge node{} (a5);
\draw[->] (a3) edge node{} (a6);

\draw[->]  (a4) edge node{} (a7);
\draw[->]  (a4) edge node{} (a8);

\draw[->]  (a5) edge node{} (a8);
\draw[->]  (a5) edge node{} (a9);

\draw[->]  (a6) edge node{} (a9);

\draw[->]  (a7) edge node{} (a10);

\draw[->]  (a8) edge node{} (a10);

\draw[->]  (a9) edge node{} (a11);

\draw[->]  (a10) edge node{} (a12);
\draw[->]  (a11) edge node{} (a12);

\draw[->] (a10) edge[bend left = 5] node[pos=.3,left]{{\tiny $$}} (a1);
\draw[->] (a9) edge[bend left = 30] node{} (a2);
\draw[->] (a12) edge[bend left = 0] node{} (a3);
\draw[->] (a11) edge[bend right = -40] node{} (a4);
\draw[->] (a7) edge[bend left = 0] node[pos=.8,right]{{\tiny $$}} (a6);
\end{tikzpicture}$$
\end{minipage}
\caption{The unweighted graph $\Ga_{\om_1+\om_2}$ in $\tilde{A_3}$ and the crystal $\widehat{B}(\om_1+\om_2)$}
\end{figure}
\end{example}

\begin{remark}
	There exist similar simple tableaux constructions for the classical types and for type $G_2$.
\end{remark}

\section{Homology ring of affine Grassmannians}\label{Sec_Homolo_Ring}
In this section, we introduce the homology rings of affine
Grassmannians which will be a key ingredient to show that the graphs that we have constructed in the previous sections are positively multiplicative.

\subsection{Homology rings of Affine Grassmannians}
\label{subsec_affGrass}
Recall the definition of $W_a^{\La_0}$ in Section \ref{aff_grass}. 
The homology algebra $\sLa$ associated to $W_{a}$ is a commutative $\mathbb{Q}$-algebra with a distinguished basis
$\{\xi_{w},w\in W_{a}^{\Lambda_{0}}\}$ corresponding to affine Schubert
classes whose associated structure coefficients are nonnegative integers: 
\[
\xi_{w}\xi_{w'}=\sum_{w''\in W_{a}^{\Lambda_{0}}%
}c_{w,w'}^{w''}\xi_{w^{''}}\qu{with $c_{w,w'}^{w''}\in\mathbb{Z}_{\geq0}$ for all $w,w'\in W_a$.}
\]
This ring is generated by a
finite number of classes $\xi_{\rho_{a}},a\in X=\{1,\ldots,N\}$ where~$\{\rho_{a},a\in X\}$ is a finite subset of $W_{a}^{\Lambda_{0}}$ whose
description can be made explicit but depends on the affine root system
considered (see \cite{LSSC}, \cite{Pon} and also \cite{LLMSSZ} Chapter 4 for a
review on affine Schubert calculus in general affine types). We assume that $\rho_{1}=s_{0}$.\ A general combinatorial description of the
coefficients~$c_{w,w'}^{w^{\prime\prime}}$ is not known in general but
there exist simple formulas for the multiplication $\xi_{\rho_{a}}\xi_{w}$
where $a\in X$ (often referred as Pieri rules in the literature). In particular, as stated in \cite{LSHPieri}, we have
\begin{equation}
\xi_{\rho_{1}}\xi_{w}=\xi_{s_{0}}\xi_{w}=\sum a_{i}\xi_{s_iw} \label{PieriBox}%
\end{equation}
where the sums is over all $i\in I$ such that $s_iw>w$ and $s_iw\in W_{a}^{\Lambda_{0}}$.
Recall here that the integers $a_{i}$ have been defined in Section~\ref{Subsec_affineRS}. More generally, we have for any $a\in X$
\begin{equation}
\xi_{\rho_{a}}\xi_{w}=\sum_{w'}c_{\rho_a,w}^{w'}\xi_{w'}
\label{GenralPieri}%
\end{equation}
where the integers $c_{\rho_a,w}^{w'}$ can be positive only when $w'>w$
for the weak Bruhat order in $W_{a}^{\Lambda_{0}}$, that is if there exists $u\in W_a$ such that $w'=uw$ with $\ell(w')=\ell(u)+\ell(w)$. 
For the affine
classical types, the rings of affine Grassmannians can be made more explicit by using subrings of symmetric functions and the previous Pieri rules then have a simple combinatorial description (see \cite{LLMSSZ,Pon}).

\smallskip

The set $W_{a}^{\Lambda_{0}}$ is stable by translation
by classical dominant weights in $P^+$ in the following sense: given any alcove $A\in W_{a}^{\Lambda_{0}}$ and any dominant weight $\la\in P^+$, the alcove $At_\la$ lies in $W_a^{\Lambda_0}$. In general, the element~$wt_\la$ with $w\in W_a$ and $\la\in P$ is not in the affine Weyl group but in the extended affine Weyl group. There exists however a unique element $u\in W_a$ such that $A_0 wt_\la = A_0u$ with equality as subsets of~$V$ but not pointwise. We will denote this element of $W_a$ by $w\star t_\la$. When $\la\in Q^\vee$, we have $wt_\la\in W_a$ and~$w\star t_\la = wt_\la$. To lighten the notation we will write $\xi_{t_{\om_i}}$ instead of $\xi_{e\star t_{\om_i}}$. 
There is a nice factorization formula in $\sLa$ (see \cite{LLMSSZ} Chapter 4 Corollary
4.14 and Equation (4.30))
\[
\xi_{w\star t_{\om_i}}=\xi_{w}\xi_{t_{\omega_{i}}} =\xi_{t_{\omega_{i}}} \xi_w\qu{for all $w\in W^{\La_0}_a$ and $i\in I^\ast$.}
\]

\begin{definition}
Let $\sB_0$ be the subset of $W_{a}^{\Lambda_{0}}$ of elements
$w$ such that $w\star t_{-\om_{i}}\notin W_{a}^{\Lambda_{0}}$ for all $i\in I^\ast$.
\end{definition}
Given $w\in W_a^{\La_0}$, there exists a unique weight $\kappa = m_1\om_1+\cdots+ m_n\om_n\in P^+$ and a unique $u\in \sB_0$ such that $ w\star t_{-\kappa} = u\in \sB_0$. 
As a consequence of the factorization above, we see that each
Schubert class element $\xi_{w}$ can be written under the form (recall that $\sLa$ is commutative)%
\begin{equation}
\xi_{w}=\xi_{t_{\omega_{1}}}^{m_{1}}\cdots\xi_{t_{\omega_{n}}}^{m_{n}}\xi_{u},u\in\sB_0.
\label{Factorization}%
\end{equation}
The set $\sB_0$ is finite with cardinality $\frac{1}{n_{W}}\mathrm{card}(W)$
where $n_{W}$ is the cardinality of $P/Q^\vee$.

\begin{example}
In affine type $A_{n}$, the ring $\sLa$ can be identified with the
subalgebra of the algebra of symmetric polynomials over
$\mathbb{Q}$ in infinitely many variables generated by the complete symmetric
functions~$\xi_{\rho_{a}}=h_{a}$ with $a=1,\ldots,n$. The Schubert basis
then corresponds to the $n$-Schur polynomials and the elements $\xi
_{t_{\omega_{i}}},i=1,\ldots,n$ to the ordinary Schur functions $s_{R_{i}}$ indexed by
the rectangular partitions $R_{i}=(n-i+1)^{i}$. We then have $n_W=n+1$ and the cardinality of $\sB_0$ is equal to $n!$.
 We refer to 
\cite{LLMSSZ} for a complete exposition and for the related interesting and very
rich combinatorics. Similar results exist for the other classical types and are
stated in \cite{LSSC} and \cite{Pon}. \end{example}

\begin{example}
\label{A2_G2_Box}
We represent the set of alcoves that lies in $\sB_0$ in type $\tilde{A}_2$ and $\tilde{G}_2$ in light gray (except for the fundamental alcove which is in dark gray). The fundamental weights $\om_1,\om_2$ are in blue.  In type $\tilde{A}_2$ we have added 3 translates of $\sB_0$ by $\om_1,\om_2$ and $\om_1+\om_2$ in green, blue and red respectively.

\begin{minipage}{8cm}
\begin{figure}[H]
\begin{center}
\begin{tikzpicture}[scale =.6]
\newcommand{\y}{.866}
\newcommand{\x}{.5}
\draw[fill = lightgray] (0,-2*\y) --  (-\x,-\y) --  (0,0) --  (\x,-\y);
\draw[fill = gray] (0,-2*\y) -- (-\x,-\y) --  (\x,-\y);
\draw[line width = .4mm] (0,-2*\y) -- (-6*\x,4*\y);
\draw[line width = .4mm] (0,-2*\y) -- (6*\x,4*\y);

 \draw[fill = ForestGreen!30!] (0+\x,-2*\y+\y) --  (-\x+\x,-\y+\y) --  (+\x,0+\y) --  (\x+\x,-\y+\y);
 \draw[fill = NavyBlue!50!] (0-\x,-2*\y+\y) --  (-\x-\x,-\y+\y) --  (-\x,0+\y) --  (\x-\x,-\y+\y);
 \draw[fill = Red!50!] (0,-2*\y+2*\y) --  (-\x,-\y+2*\y) --  (0,0+2*\y) --  (\x,-\y+2*\y);

\draw[line width = .6mm,color = blue,->] (0,-2*\y) -- (\x,-\y);
\draw[line width = .6mm,color = blue,->] (0,-2*\y) -- (-\x,-\y);
\node at (.6,-1.4) {\scalebox{.6}{$\om_1$}};
\node at (-.5,-1.4) {\scalebox{.6}{$\om_2$}};



\foreach \n in {-1,...,2}{
\draw (-3,2* \n * \y) -- (3,2 * \n * \y);}

\foreach \n in {-2,...,1}{
\draw (-3.5,2* \n * \y + \y) -- (3.5,2 * \n * \y + \y);}

\foreach \m in {0,...,6}{
\foreach \n in {0,...,3}{
\draw (-3+2*\m*\x,4*\y-\n*2*\y) -- (-3.5+2*\m*\x,3*\y-\n*2*\y);}

\foreach \n in {0,...,2}{
\draw (-3.5+2*\m*\x,3*\y-\n*2*\y) -- (-3+2*\m*\x,2*\y-\n*2*\y);}}

\foreach \m in {0,...,6}{
\foreach \n in {0,...,3}{
\draw (-3.5+2*\m*\x+\x,4*\y-\n*2*\y) -- (-3+2*\m*\x+\x,3*\y-\n*2*\y);}

\foreach \n in {0,...,2}{
\draw (-3+2*\m*\x+\x,3*\y-\n*2*\y) -- (-3.5+2*\m*\x+\x,2*\y-\n*2*\y);}}
\end{tikzpicture}
\caption{ the set $\sB_0$ in type $\tilde{A}_2$}
\label{BBA2}
\end{center}
\end{figure}
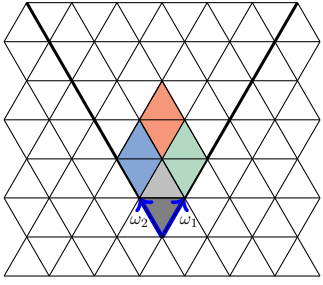
\end{minipage}
\begin{minipage}{8cm}
\begin{figure}[H]
\begin{center}
\begin{tikzpicture}[scale =.6]
\newcommand{\y}{1.5}
\newcommand{\x}{.866}

\draw[line width = .4mm] (-2*\x,-2*\y) -- (-2*\x,2*\y);
\draw[line width = .4mm] (-2*\x,-2*\y) -- (2*\x,2*\y);
\draw[fill = lightgray] (-2*\x,-2*\y) -- (-2*\x,0*\y) -- (-\x,\y) -- (-\x,-\y);
\draw[fill = gray] (-2*\x,-2*\y) -- (-2*\x,-2*\y + 1) -- (-2*\x + 1/2*\x,-2*\y + .75);

\draw[line width = .6mm,color = blue,->] (-2*\x,-2*\y) -- (-\x,-\y);
\draw[line width = .6mm,color = blue,->] (-2*\x,-2*\y) -- (-2*\x,0*\y);

\node at (-\x+.4,-\y-.3) {\scalebox{.6}{$\om_2$}};
\node at (-3*\x+.3,-.2) {\scalebox{.6}{$\om_1$}};


\draw (-3.464,3) -- (3.464,3);
\draw (-3.464,1.5) -- (3.464,1.5);
\draw (-3.464,0) -- (3.464,0);
\draw (-3.464,-1.5) -- (3.464,-1.5);
\draw (-3.464,-3) -- (3.464,-3);

\foreach \n in {-4,...,4} {
\draw (\n * \x,-3) -- (\n * \x,3);}

\foreach \e in {-1,1} {
\draw (\e*2*\x,3) -- (\e*4*\x,2);
\draw (\e*0*\x,3) -- (\e*4*\x,1);
\draw (\e*-2*\x,3) -- (\e*4*\x,0);
\draw (\e*-4*\x,3) -- (\e*4*\x,-1);

\draw (-\e*4*\x,2) -- (\e*4*\x,-2);
\draw (-\e*4*\x,1) -- (\e*4*\x,-3);

\draw (-\e*4*\x,0) -- (\e*2*\x,-3);
\draw (-\e*4*\x,-1) -- (\e*0*\x,-3);
\draw (-\e*4*\x,-2) -- (-\e*2*\x,-3);}

\foreach \e in {-1,1} {
\draw (\e*2*\x,3) -- (\e*4*\x,0);
\draw (\e*0*\x,3) -- (\e*4*\x,-3);
\draw (-\e*2*\x,3) -- (\e*2*\x,-3);
\draw (-\e*4*\x,3) -- (\e*0*\x,-3);
\draw (-\e*4*\x,0) -- (-\e*2*\x,-3);}
\end{tikzpicture}
\end{center}
\caption{ the set $\sB_0$ in type $\tilde{G}_2$}
\label{BBG2}
\end{figure}
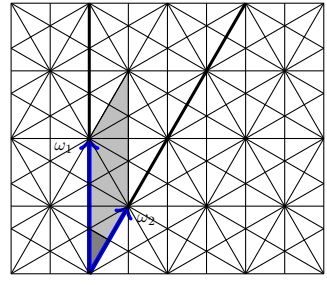
\end{minipage}
\end{example}

Recall the definition of the graph $\Gamma_{\La_0}$ in Definition \ref{Def_AG}. The set of arrows in $\Gamma_{\La_0}$ is the set of pairs $(w,s_iw)\in  W_a^{\La_0}\times  W_a^{\La_0}$ such that $s_iw>w$. If we define the edge weight function $\mathsf{p}$ by setting  $\mathsf{p}(w,s_iw) = a_i$  then  the weighed graph $\Gamma_{\La_0}$ encodes the multiplication by $\xi_{\rho_1}$ in the algebra $\sLa$ and in the basis $\{\xi_w\mid w\in W_a^{\La_0}\}$. Indeed we have for $w\in W_a^{\La_0}$ 
$$\xi_{\rho_1}\xi_{w} = \sum_{(w,w')\in E} \mathsf{p}(w,w') \xi_{w'}.$$

\begin{example}
We go back to the affine type $A_{n}^{(1)}$ case.
The graph $\Gamma_{\Lambda_{0}}$ can alternatively be regarded as the lattice of
$(n+1)$-core partitions where the generators $s_{i},i\in I$ act
on the $(n+1)$-core partitions by either removing all the removable
$i$-nodes or by adding all the addable $i$-nodes (this is well-defined since a core partition cannot contain addable $i$-nodes and removable $i$-nodes). There also exists a bijection between $(n+1)$-core partitions
and $n$-bounded partitions, that is partitions whose Young diagram has at most
$n$ columns. This yields another labelling of $\Gamma_{\Lambda_{0}}$ (which
does not coincide with the Young lattice on $n$-bounded partitions).\ The set
$\sB_0$ contains $n!$ partitions which are exactly the $n$-bounded
partitions that do not contain any rectangle $R_{i}$. Here again we refer to~\cite{LLMSSZ} for a complete exposition.
\end{example}

\subsection{Finite PM-graphs from affine Grassmannians}
\label{posrho}

\begin{lemma}
\label{Lem_RiAlInde}The elements $\xi_{t_{\omega_{1}}},\ldots,\xi_{t_{\omega_{n}}}$ are algebrically
independent in the homology algebra $\sLa$.
\end{lemma}
\begin{proof}
Assume we have a polynomial $P$ with coefficients in $\mathbb{Q}$ such that
$P(\xi_{t_{\omega_{1}}},\ldots,\xi_{t_{\omega_{n}}})=0$. This gives a linear combination of Schubert classes
equal to zero. Thus the polynomial $P$ should be equal to zero since these
Schubert classes gives a basis of $\sLa$.
\end{proof}
As we have seen in the previous section, any element $\xi_{w}$  can be written under the form 
\begin{equation*}
\xi_{w}=\xi_{t_{\omega_{1}}}^{m_{1}}\cdots\xi_{t_{\omega_{n}}}^{m_{n}}\xi_{u} \qu{where} u\in\sB_0.
\end{equation*}
By the previous lemma, we see that the subring of $\sLa$ generated by $\xi_{t_{\omega_{1}}},\ldots,\xi_{t_{\omega_{n}}}$ is isomorphic to the ring of polynomials $\Q[\sz_1,\ldots,\sz_n]$ in $n$ variables via the map $\xi_{t_{\om_i}} \leftrightarrow \sz_i$ which in turn isomorphic to the monoid algebra of dominant weights $\Q[P^+]=\langle \sz^\la\mid \la\in P^+\rangle_{\Q}$ via the maps $\sz_i \leftrightarrow \sz^{\om_i}$.  Identifying those rings, we see that~$\sLa$ can be seen as a finite-dimensional ring over~$\Q[P^+]$ with basis $\{\xi_w\mid w \in \sB_0\}$. We extend the scalars to the group algebra $\Q[P]$ (or equivalently to  the ring of Laurent polynomials $\Q[\sz_1^{\pm1},\ldots,\sz^{\pm1}_n]$) and we denote by $\sLa_{P}$ this algebra. It is finite dimensional over $\Q[P]$ with basis $\{\xi_w\mid w\in \sB_0\}$. Note that we are in the setting of Section \ref{mutl_graph} where the algebras are defined over $\Q$-rings of Laurent polynomials. Further, the ring $\sLa_{P}$ is positively multiplicative. For any $w\in \sB_0$, we write $m_{\xi_w}$ for the morphism of $\sLa_{P}$ defined by the multiplication by $\xi_w$ and $\func{Mat}_{\sB_0}(m_{\xi_w})$ to for the matrix with respect to the basis $\{\xi_w\mid w\in \sB_0\}$.

\begin{definition}
We define $\Ga_{\sB_0}$ to be the graph with set of vertices $\sB_0$ and adjacency matrix $\func{Mat}_{\sB_0}(m_{\xi_{\rho_1}})$. 
\end{definition}
Equivalently, the graph $\Ga_{\sB_0}$ can be defined as follows:
we put a weighted arrow with weight $a_i\sz^{\kappa}$ with $i\in I$ and $\kappa\in P_+$ from $w$ to $w'$ in $\sB_0$ whenever $\ell(s_iw) =\ell(w)+1$ and $s_iw\star t_{-\kappa} = w'$. Note that $\kappa$ is not equal to zero in the coefficient $a_i\sz^\kappa$ only when $s_iw\notin \sB_0$. We can see that $\Ga_{\sB_0}$ is strongly connected. Indeed, consider an element $w$ in $\sB_0$. Then, there
exists $u$ in $W_{a}$ and $\kappa\in Q^\vee_{+}$ such that $uw=t_{\kappa}$ with
$\ell(t_{\kappa})=\ell(u)+\ell(w)$. This yields a directed path in
$\Gamma_{\sB_0}$ from $w$ to $e$. Since it is clear that there also
exists a path from $e$ to any element of $\sB_0$ in the graph
$\Gamma_{\sB_0}$, this shows that $\Gamma_{\sB_0}$ is strongly connected.

\begin{proposition}
\label{mult_B0}
The graph $\Ga_{\sB_0}$ is positively multiplicative at $e$.
\end{proposition}
\begin{proof}
By definition the adjacency matrix of $\Ga_{\sB_0}$ is the matrix of multiplication by $\xi_{\rho_1}$ in the basis $\{\xi_{w}\mid w\in \sB_0\}$ in $\sLa_{P}$ . Since the identity element $\xi_{e}$ of $\sLa_{P}$ lies in $\sB_0$ and since $\sLa_{P}$ is positively multiplicative, we get the result by Proposition \ref{m_x}.
\end{proof}
\begin{example}
We continue example \ref{A2_G2_Box}. The graph $\Ga_{\sB_0}$ in type $\tilde{A}_2$ is particularly simple as there are only two alcoves in $\sB_0$ and only two pairs $(w,s)$ in $\sB\times I$ such that $s_iw\notin \sB$. These pairs are $(s_0,1)$ and~$(s_0,2)$ and we have $A_0s_1s_0=A_0t_{\om_1}$ (as sets) and $A_0s_2s_0=A_0t_{\om_2}$ (as sets) so that $s_1s_0\star t_{-\om_1} = e$ and~$s_2s_0\star t_{-\om_2} = e$. We obtain the following graph and matrix
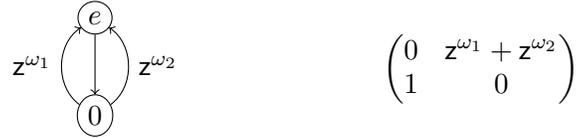
\begin{figure}[H]
\begin{minipage}{5cm}
\begin{center}
\begin{tikzpicture}[scale =1.3]
\tikzstyle{vertex}=[inner sep=2pt,minimum size=10pt,ellipse,draw]

\node[vertex] (a1) at (0,0) {\scalebox{1}{$e$}};
\node[vertex] (a2) at (0,-1) {\scalebox{1}{$0$}};

\draw[->] (a1) -- (a2);
\draw[->,bend left=55] (a2) to node[left] {\scb{1}{$\sz^{\om_1}$}}(a1) ;
\draw[->,bend right=55] (a2) to node[right] {\scb{1}{$\sz^{\om_2}$}} (a1) ;

\end{tikzpicture}
\end{center}
\end{minipage}
\begin{minipage}{5cm}
\begin{center}
$\begin{pmatrix}
0&\sz^{\om_1}+\sz^{\om_2}\\1&0
\end{pmatrix}$
\end{center}
\end{minipage}
\caption{The graph $\Ga_{\sB_0}$ in type $\tilde{A}_2$ and its transition matrix.}
\end{figure}
 In type $\tilde{G}_2$, the situation is more complicated. There are 6 pairs $(w,i)\in \sB\times I$ such that $s_iw\notin \sB_0$. We represent these pairs below on the left hand-side where the gray area is the set $\sB_0$. For each pair $(w,i)$ such that $s_iw\notin \sB_0$ we put an arrow between $(w,s_iw)$. We color the arrows in both figures to show which pairs corresponds to which arrow in the graph. 
\begin{figure}[H]
\begin{center}
\begin{minipage}{2cm}
\begin{center}
\begin{tikzpicture}[scale=.8]

\draw[fill,color = lightgray] (0,0) -- (0,3) -- (.866,4.5) -- (.866,1.5);
\draw[fill,color = gray] (0,0) -- (0,1) -- (.433,.75) ;

\draw[line width = .4mm] (0,0) -- (0,6);
\draw[line width = .4mm] (0,0) -- (1.732,3);
\draw[line width = .4mm] (1.732,6) -- (1.732,3);
\draw[line width = .4mm] (0,3) -- (.866,4.5);
\draw[line width = .4mm] (0,6) -- (.866,7.5);
\draw[line width = .4mm] (.866,4.5) -- (1.732,6); 
\draw[line width = .4mm] (.866,4.5) -- (.866,7.5);
\draw[line width = .4mm] (.866,1.5) -- (.866,4.5);

\draw[line width = .2mm, color = NavyBlue,->] (.75,2.1) -- (1.05,2.1);
\draw[line width = .2mm, color = NavyBlue,->] (.75,3.9) -- (1.05,3.9);

\draw[line width = .2mm, color = Green,->] (.75,2.8) -- (1.05,2.8);
\draw[line width = .2mm, color = Green,->] (.75,3.2) -- (1.05,3.2);

\draw[line width = .2mm, color = Red,->] (.65,3.8) -- (.35,4);
\draw[line width = .2mm, color = Red,->] (.5,3.4) -- (.15,3.6);

\newcommand{\y}{0}
\newcommand{\x}{0}

\draw (0+\x,1+\y) -- (.433+\x,.75+\y);
\draw (0+\x,1+\y) -- (.866+\x,1.5+\y);
\draw (0+\x,1.5+\y) -- (.866+\x,1.5+\y);
\draw (.866+\x,1.5+\y)-- (0+\x,2+\y);
\draw (.866+\x,1.5+\y)-- (0+\x,3+\y);
\draw (0+\x,2+\y) -- (.866+\x,2.5+\y);
\draw (0+\x,3+\y) -- (.866+\x,2.5+\y);
\draw (0+\x,3+\y) -- (.866+\x,3+\y);
\draw (0+\x,3+\y) -- (.866+\x,3.5+\y);
\draw (.433+\x,3.75+\y) -- (.866+\x,3.5+\y);

\renewcommand{\y}{3}
\renewcommand{\x}{0}

\draw (0+\x,1+\y) -- (.433+\x,.75+\y);
\draw (0+\x,1+\y) -- (.866+\x,1.5+\y);
\draw (0+\x,1.5+\y) -- (.866+\x,1.5+\y);
\draw (.866+\x,1.5+\y)-- (0+\x,2+\y);
\draw (.866+\x,1.5+\y)-- (0+\x,3+\y);
\draw (0+\x,2+\y) -- (.866+\x,2.5+\y);
\draw (0+\x,3+\y) -- (.866+\x,2.5+\y);
\draw (0+\x,3+\y) -- (.866+\x,3+\y);
\draw (0+\x,3+\y) -- (.866+\x,3.5+\y);
\draw (.433+\x,3.75+\y) -- (.866+\x,3.5+\y);

\renewcommand{\y}{1.5}
\renewcommand{\x}{.866}

\draw (0+\x,1+\y) -- (.433+\x,.75+\y);
\draw (0+\x,1+\y) -- (.866+\x,1.5+\y);
\draw (0+\x,1.5+\y) -- (.866+\x,1.5+\y);
\draw (.866+\x,1.5+\y)-- (0+\x,2+\y);
\draw (.866+\x,1.5+\y)-- (0+\x,3+\y);
\draw (0+\x,2+\y) -- (.866+\x,2.5+\y);
\draw (0+\x,3+\y) -- (.866+\x,2.5+\y);
\draw (0+\x,3+\y) -- (.866+\x,3+\y);
\draw (0+\x,3+\y) -- (.866+\x,3.5+\y);
\draw (.433+\x,3.75+\y) -- (.866+\x,3.5+\y);
%

\end{tikzpicture}
\end{center}
\end{minipage}
\begin{minipage}{8cm}
\begin{center}
\begin{tikzpicture}[scale=.65]
\tikzstyle{vertex}=[inner sep=.2pt,minimum size=10pt,ellipse,draw]

\node[vertex] (a1) at (0,0) {\scalebox{.45}{$e$}};

\node[vertex] (a2) at (0,-1) {\scalebox{.45}{$0$}};

\node[vertex] (a3) at (0,-2) {\scalebox{.45}{$20$}};

\node[vertex] (a4) at (0,-3) {\scalebox{.45}{$120$}};

\node[vertex] (a5) at (0,-4) {\scalebox{.45}{$2120$}};

\node[vertex] (a6) at (-1,-5) {\scalebox{.45}{$02120$}};

\node[vertex] (a7) at (1,-5) {\scalebox{.45}{$12120$}};

\node[vertex] (a8) at (0,-6) {\scalebox{.45}{$102120$}};

\node[vertex] (a9) at (0,-7) {\scalebox{.45}{$2102120$}};

\node[vertex] (a10) at (0,-8) {\scalebox{.45}{$12102120$}};

\node[vertex] (a11) at (0,-9) {\scalebox{.45}{$212102120$}};

\node[vertex] (a12) at (0,-10) {\scalebox{.45}{$0212102120$}};

\draw (a1) to node[left] {\scalebox{.8}{1}} (a2);
\draw (a2) to node[left] {\scalebox{.8}{2}}  (a3);
\draw (a3) to node[left] {\scalebox{.8}{3}}  (a4);
\draw (a4) to node[left] {\scalebox{.8}{2}}  (a5);

\draw (a5) to node[left] {\scalebox{.8}{1}}  (a6);
\draw (a5) to node[right] {\scalebox{.8}{3}}  (a7);

\draw (a6) to node[left] {\scalebox{.8}{3}}  (a8);
\draw (a7) to node[right] {\scalebox{.8}{1}}  (a8);

\draw (a8) to node[left] {\scalebox{.8}{2}}  (a9);
\draw (a9) to node[left] {\scalebox{.8}{3}}  (a10);
\draw (a10) to node[left] {\scalebox{.8}{2}}  (a11);
\draw (a11) to node[left] {\scalebox{.8}{1}}  (a12);

\draw[->,NavyBlue,bend right=55] (a7) to node[left,pos =.6]  {{\scb{.7}{$2\sz^{\om_2}$}}} (a1) ;
\draw[->,NavyBlue,bend left=55] (a12) to  node[left,pos =.8] {{\scb{.7}{$2\sz^{\om_2}$}}} (a6);

\draw[->,Green,bend right=65] (a9) to node[right,pos =.25]  {{\scb{.7}{$\sz^{\om_2}$}}} (a3);
\draw[->,Green,bend left=65] (a10) to node[left,pos =.75]  {{\scb{.7}{$\sz^{\om_2}$}}} (a4);

\draw[->,Red,bend right=75] (a11) to node[right,pos =.5] {{\scb{.7}{$3\sz^{\om_1}$}}} (a1);
\draw[->,Red,bend left=75] (a12) to node[left,pos =.5]  {{\scb{.7}{$3\sz^{\om_1}$}}} (a2);

\end{tikzpicture}
\end{center}
\end{minipage}
\begin{minipage}{5cm}
\begin{center}
\scalebox{.55}{$\begin{pmatrix}
0&0&0&0&0&2\sz^{\om_2}&0&0&0&0&3\sz^{\om_1}&0\\
1&0&0&0&0&0&0&0&0&0&0&3\sz^{\om_1}\\
0&2&0&0&0&0&0&0&\sz^{\om_2}&0&0&0\\
0&0&3&0&0&0&0&0&0&\sz^{\om_2}&0&0\\
0&0&0&2&0&0&0&0&0&0&0&0\\
0&0&0&0&3&0&0&0&0&0&0&0\\
0&0&0&0&1&0&0&0&0&0&0&2\sz^{\om_2}\\
0&0&0&0&0&1&3&0&0&0&0&0\\
0&0&0&0&0&0&0&2&0&0&0&0\\
0&0&0&0&0&0&0&0&3&0&0&0\\
0&0&0&0&0&0&0&0&0&2&0&0\\
0&0&0&0&0&0&0&0&0&0&1&0\\
\end{pmatrix}$}\end{center}
\end{minipage}
\end{center}
\caption{The graph $\Ga_{\sB_0}$ in type $\tilde{G}_2$ and its adjacency matrix.}
\end{figure}
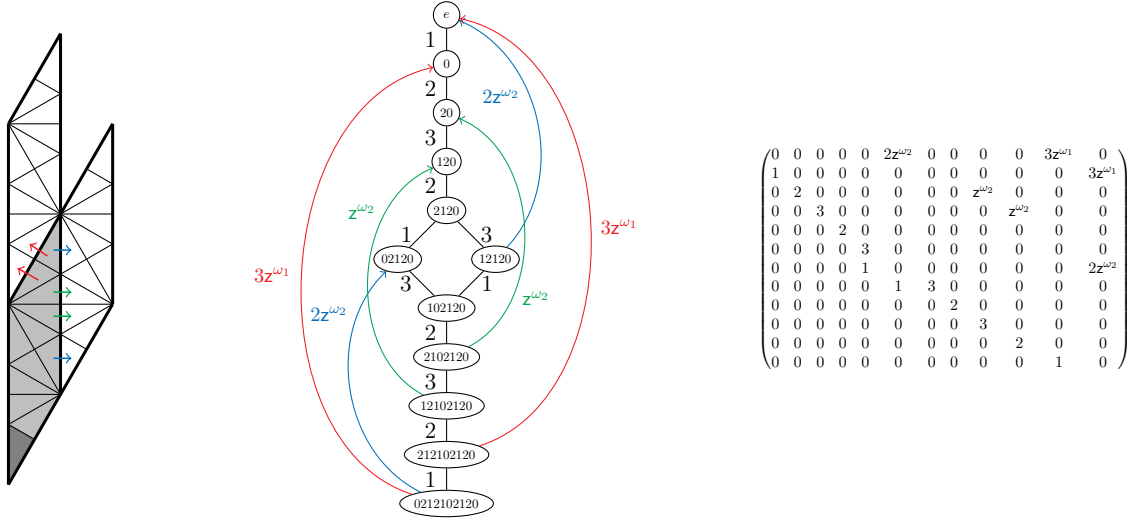
\end{example}

We also have the following interesting proposition whose proof is similar to
Proposition 4.2 in \cite{LT2}. We denote by $\func{Frac}(\sLa)$ and $\Q(P^+)$ the fraction fields of $\sLa$ and $\Q[P^+]$, respectively. Recall the definition of the elements $\xi_{\rho_{a}},a\in X$ at the beginning of §\ref{subsec_affGrass}. 

\begin{proposition}
\ \label{PropLK}

\begin{enumerate}
\item Each $\xi_{\rho_{a}},a\in X$ is algebraic over $\Q(P^+)$.

\item $\func{Frac}(\sLa)$ is an algebraic extension of $\Q(P^+)$ with degree
$\mathrm{card}(\sB_0)$. The set $\mathcal{I}%
=\{\xi_{w}\mid w\in\sB_0\}$ is a basis of $\func{Frac}(\sLa)$ over
$\Q(P^+)$.
\item The element $\xi_{\rho_1}$ is primitive in $\func{Frac}(\sLa)$,
that is we have $\func{Frac}(\sLa)=\Q(P^+)\big[\xi_{\rho_1}\big]$.
\item The expansion of $\Ga_{\sB_0}$ is the graph $\Ga_{\La_0}$ with weights $a_i$.
\end{enumerate}
\end{proposition}
The fact that the expansion of  $\Ga_{\sB_0}$ is the graph $\Ga_{\La_0}$ allows one to use the multiplicative properties of the graph $\Ga_{\sB_0}$ to study random walks on $\Ga_{\La_0}$. This was done in type $A$ by the last two authors in \cite{LT2} in which case we have $a_i=1$ for all $i$. 
\begin{corollary}

The graph $\Gamma_{\sB_0}$ has maximal dimension, that is the
degree of $\mu_{A_{\Ga_{\sB_0}}}$ is equal to $\mathrm{card}(\Gamma
_{\sB_0})$.

\end{corollary}

\begin{proof}
The minimal polynomial of $\xi_{\rho_1}$ has degree $\mathrm{card}(\Gamma
_{\sB_0})$ which is the size of the matrix $A_{\Ga_{\sB_0}}$.
\end{proof}

\begin{remark}
The graph $\Gamma_{\sB_0}$ has maximal dimension and is multiplicative at $e$. It follows by Proposition~\ref{ThGG1} (1) that the matrix $M_e$ is invertible and the
columns of $M_{e}^{-1}$ give the vectors of a positively multiplicative basis
$\sB'=\{\sb_{0}=1,\ldots,\ldots,\sb_{n-1}\}$ expressed
in the basis $\{1,A_{\Ga_{\sB_0}},\ldots,A_{\Ga_{\sB_0}}^{n-1}\}$. The coefficients appearing in these expressions are in $\frac{1}{\det M_{e}}\Q[P]$. By unicity of the multiplicative basis satisfying~$\sb_0=1$, we see that the basis $\sB'$ has to be equal to the basis $\sB = \{\func{Mat}_{\sB_0}(\xi_w)\mid w\in \sB_0\}$. As a consequence, since all the matrices in $\sB$ have coefficients in $\Q[P^+]$, it follows that all the denominators in the expression of the $\sb_i$ in  $\{1,A_{\Ga_\sB},\ldots,A_{\Ga_\sB}^{n-1}\}$ simplify (see Example \ref{ex_A3}). This also provides an explicit algorithm to compute the structure constants with respect to the basis~$\sB$ (these are given by the columns of the matrix in the basis~$\sB$). We have computed explicitely the basis $\sB$ in type $\tilde{G_2}$ \cite{G2}. 
\end{remark}

\newcommand{\sX}{\mathsf{X}}
\newcommand{\up}{\upsilon}

\subsection{The $\Q$-algebra $\sLa_a$}
We have seen in the previous section that $\sLa_{P}$ is a finite-dimensional $\mathbb{Q}[P]$-algebra with basis $\{\xi_w\mid w\in \sB_0\}$. The set $\sB_0$ is a fundamental domain for the action of $P$ (by translations) on~$W_a$, that is, any element of $W_a$ can be uniquely written as $u\star t_{\la}$ where $u\in \sB_0$ and $\la\in P$. We can therefore define two maps $\wt_{\sB_0}:W_a\to P$ and $\sf_{\sB_0}:W_a\to \sB_0$ by the equation $w= \sf_{\sB_0}(w)\star t_{\wt_{\sB_0}(w)}$. Then by setting $\xi_{w} := \sz^{\wt_{\sB_0}(w)}\xi_{\sf_{\sB_0}(w)}$  for all $w\in W_a$ we can turn $\sLa_{P}$ into an infinite-dimensional algebra over $\Q$ with basis~$\{\xi_w\mid w\in W_a\}$. We will denote by $\sLa_a$ this $\Q$-algebra. Recall the definition of $w\overset{+}{\leadsto} s_iw$ in \ref{crossing}. 
\begin{proposition}
\label{new_pieri}
For all $w\in W_a$, we have in $\sLa_a$
\begin{equation}
\label{eq}
\xi_{\rho_1}\xi_{w} =\sum_{i\in I, w\overset{+}{\leadsto} s_iw} a_{i}\xi_{s_iw}
\end{equation}
\end{proposition}
\begin{proof}
We know that for all $w\in W_a^{\La_0}$, we have 
$$\xi_{\rho_1}\xi_{w} = \sum_{i\in I, s_iw\in W_a^{\La_0}, s_iw>w} a_{i}\xi_{s_iw}.$$
If  $w\in W_a^{\La_0}$, the assertion $s_iw\in W_a^{\La_0}$ and $s_iw>w$ is equivalent to $w\overset{+}{\leadsto} s_iw$. Hence (\ref{eq}) holds for~$w\in W_a^{\La_0}$. Now, let $w\in W_a$. There exists $\ga\in Q^\vee$ such that $w = xt_{\ga}\in W_a^{\La_0}$ so that $w = t_{-\ga}x$ and  $\xi_w=\sz^{-\ga}\xi_x$. We have 
$$\xi_{\rho_1}\xi_{w} = \xi_{\rho_1} \sz^{-\ga}\xi_x =  \sz^{-\ga}\sum_{s\in I, x\overset{+}{\leadsto}xs} a_{s}\xi_{xs} = 
\sum_{s\in I, x\overset{+}{\leadsto}xs} a_{s}\xi_{t_{-\ga}sx} = \sum_{s\in I, w\overset{+}{\leadsto}sw} a_{s}\xi_{sw}.$$
The last equality comes from the fact that the orientation is unchanged by translation, that is we have $x\overset{+}{\leadsto}xs$ if and only if $w\overset{+}{\leadsto}ws$.  
\end{proof}
The result above can be generalised to all $\rho_a$ with $a\in X$. Let $w,w'\in W_a$. We write $w\overset{+}\to w'$ if and only if $w$ lies on the positive side of all hyperplane separating $w$ and $w'$. It is clear that if $w$ and $w'$ are adjacent then $w\overset{+}{\leadsto} w'$ is equivalent to $w\overset{+}\to w'$. Then, arguing as above we get for all $w\in W_a$ and $a\in X$:
\begin{equation*}
\xi_{\rho_a}\xi_{w} =\sum_{w',w\overset{+}{\to}w'} a_{w'}\xi_{w'}.
\end{equation*}
Indeed, If  $w\in W_a^{\La_0}$, the assertion $w'\in W_a^{\La_0}$ and $w'\geq w$ in the weak Bruhat order is equivalent to $w\overset{+}{\to}w'$.

\subsection{Fundamental domain and basis of the homology algebra}

In Section \ref{posrho}, we have used the lattice of fundamental weights and the fundamental domain~$\sB_0$ to construct a (finite) positively multiplicative graph $\Ga_{\sB_0}$. In this section we generalise this result to any $\mathbb{Z}$-sublattice $L$ of $P$ of rank $n$ with basis ${\up_1,\ldots,\up_n}$ and to any fundamental domain.

 \begin{definition} 
   A subset $\sB$ of $W_a$ is a fundamental domain for the action of $L$ if and only if any element $w$ of $W_a$ can uniquely be written as a product $u\star t_{\ga}$ where $\ga\in L$ and $u\in \sB$. To any fundamental domain $\sB$ we associate two applications $\wt_\sB:W_a\to L$ and  $\sf_{\sB}:W_a\to \sB$
defined by the equation $w = \sf_{\sB}(w) \star t_{\wt_\sB(w)}$. 
 \end{definition}

Let $\Q[L]=\langle \sz^{\al}\mid \al\in L\rangle$ be the group algebra of $L$ over $\Q$. It is a subalgebra of $\Q[P]$. By the factorisation property and since $L$ is a sublattice of~$P$, we have $\xi_w=\sz^{\wt_\sB(w)}\xi_{\sf_\sB(w)}\in \sLa_{a}$ for all $w\in W_a$ and all fundamental domain $\sB$.
This shows that $\sLa_{a}$ can be turned into a $\Q[L]$-algebra with basis $\{\xi_{w}\mid w\in \sB\}$. We will write $\sLa_{L}$ when we consider~$\sLa_a$ as an algebra over $Q[L]$. 
Note that $\sLa_{L}$  is finite-dimensional since~$L$ is assumed to be of rank $n$. 

\begin{remark}
\label{Q-case}
In the case where $L=Q^\vee$,  $W$ is a fundamental domain for the action of $Q^\vee$ so that $\{\xi_{w}\mid w\in W\}$ is a basis of $\sLa_{Q^\vee}$.
\end{remark}
 
Let $m_{\xi_{\rho_1}}$ be the linear endomorphism of $\sLa_L$ defined by multiplication by $\xi_{\rho_1}$. Given an $L$-fundamental domain $\sB$, we write  $\func{Mat}_{\sB}(m_{\xi_{\rho_1}})$ for the matrix of $m_{\xi_{\rho_1}}$ in the basis $\{\xi_{w}\mid w\in \sB\}$. We then define $\Ga_{\sB}$ to be the graph with set of vertices $\sB$ and adjacency matrix $A_{\Ga_{\sB}} = \func{Mat}_{\sB}(m_{\xi_{\rho_1}})$. Equivalently, the graph $\Ga_{\sB}$ can be defined as follows:
we put a weighted arrow with weight $a_i\sz^{\kappa}$ with $i\in I$ from $w$ to $w'$ in $\sB$ whenever $w\overset{+}{\leadsto} s_iw$, $s_iw\star t_{-\kappa} = w'$ and $\kappa\in L$. Note that $\kappa$ is not equal to zero in the coefficient $a_i\sz^\kappa$ only if $s_iw\notin \sB$. We can see that $\Ga_{\sB}$ is strongly connected. Arguing as in Proposition \ref{mult_B0} we get the following result. 

\begin{proposition}
\label{fdomain}
Let $\sB$ be a fundamental domain for the action of $L$ containing the fundamental alcove~$A_0$. The graph $\Ga_{\sB}$ with adjacency matrix $\func{Mat}_{\sB}(m_{\xi_{\rho_1}})$ is positively multiplicative at $e$.
\end{proposition}
This proposition is the key to show that $\Ga_\rho$ is positively multiplicative. Indeed we will show that the adjacency matrix of $\Ga_\rho$ is essentially the matrix of $m_{\xi_{\rho_1}}$ expressed in the basis associated to the coroot lattice $Q^\vee$ and to the fundamental domain $W$.

\begin{example} Let $L=Q^\vee$ in type $\tilde{A}_2$. We can choose as $L$-fundamental domain $\sB$ the set of alcoves in green and the fundamental alcove $A_0$ which is in dark grey. We represent the graph with adjacency matrix $\func{Mat}_{\sB}(m_{\xi_{\rho_1}})$ on the right hand-side. We set $\sz_1=\sz^{\al_1}$ and $\sz_2=\sz^{\al_2}$. 

\vspace{-.1cm}
\begin{figure}[H]
\begin{center}
\begin{minipage}{8cm}
\psset{linewidth=.05mm}
\begin{center}
\begin{tikzpicture}[scale =.5]
\newcommand{\y}{.866}
\newcommand{\x}{.5}

\draw[fill = ForestGreen!30!] (-\x,\y) -- (-2*\x,2*\y) -- (-\x,3*\y) -- (0,2*\y)  -- (\x,3*\y) -- (2*\x,2*\y) -- (\x,\y);

\draw[fill = gray] (0,0) -- (-\x,\y) -- (\x,\y);

\draw[color = black, ->,line width = .4mm] (0,0) -- (-3*\x,\y);
\draw[color = black, ->,line width = .4mm] (0,0) -- (3*\x,\y);
\draw[color = black, ->,line width = .4mm] (0,0) -- (0,2*\y);

\node at (3*\x,.5*\y) {\scalebox{.6}{$\al^\vee_1$}};
\node at (-3*\x,.5*\y) {\scalebox{.6}{$\al^\vee_2$}};

\draw[line width = .5mm] (0,0) -- (-4*\x,4*\y);
\draw[line width = .5mm] (0,0) -- (4*\x,4*\y);


\foreach \n in {-1,...,2}{
\draw (-3,2* \n * \y) -- (3,2 * \n * \y);}

\foreach \n in {-2,...,1}{
\draw (-3.5,2* \n * \y + \y) -- (3.5,2 * \n * \y + \y);}

\foreach \m in {0,...,6}{
\foreach \n in {0,...,3}{
\draw (-3+2*\m*\x,4*\y-\n*2*\y) -- (-3.5+2*\m*\x,3*\y-\n*2*\y);}

\foreach \n in {0,...,2}{
\draw (-3.5+2*\m*\x,3*\y-\n*2*\y) -- (-3+2*\m*\x,2*\y-\n*2*\y);}}

\foreach \m in {0,...,6}{
\foreach \n in {0,...,3}{
\draw (-3.5+2*\m*\x+\x,4*\y-\n*2*\y) -- (-3+2*\m*\x+\x,3*\y-\n*2*\y);}

\foreach \n in {0,...,2}{
\draw (-3+2*\m*\x+\x,3*\y-\n*2*\y) -- (-3.5+2*\m*\x+\x,2*\y-\n*2*\y);}}
\end{tikzpicture}
\end{center}
\end{minipage}
\begin{minipage}{8cm}
\begin{center}
\begin{tikzpicture}[scale =1]
\tikzstyle{vertex}=[inner sep=2pt,minimum size=10pt,ellipse,draw]

\node[vertex] (a1) at (0,0) {\scalebox{.7}{$e$}};
\node[vertex] (a2) at (0,-1) {\scalebox{.7}{$0$}};
\node[vertex] (a3) at (-1,-2) {\scalebox{.7}{$10$}};
\node[vertex] (a4) at (1,-2) {\scalebox{.7}{$20$}};
\node[vertex] (a5) at (-1,-3) {\scalebox{.7}{$102$}};
\node[vertex] (a6) at (1,-3) {\scalebox{.7}{$201$}};

\draw[->] (a1) edge node  {} (a2);
\draw[->] (a2) edge node{} (a3);
\draw[->] (a2) edge node{} (a4);
\draw[->] (a3) edge node{} (a5);
\draw[->] (a4) edge node{} (a6);
\draw[->]  (a5) edge[bend left = 70] node[left]{\scalebox{.8}{$\sz_1\sz_2$}} (a1);
\draw[->]  (a5) edge[bend left = 0] node[pos=.3,below]{\scalebox{.8}{$\sz_2$}} (a4);

\draw[->]  (a6) edge[bend right = 70] node[right]{\scalebox{.8}{$\sz_1\sz_2$}} (a1);
\draw[->]  (a6) edge[bend right = 0] node[pos=.3,below]{\scalebox{.8}{$\sz_1$}} (a3);
\end{tikzpicture}
\end{center}
\end{minipage}
\end{center}
\vspace{-.2cm}
\caption{A fundamental domain for $Q^\vee$ and the associated graph  in type $\tilde{A}_2$}
\end{figure}
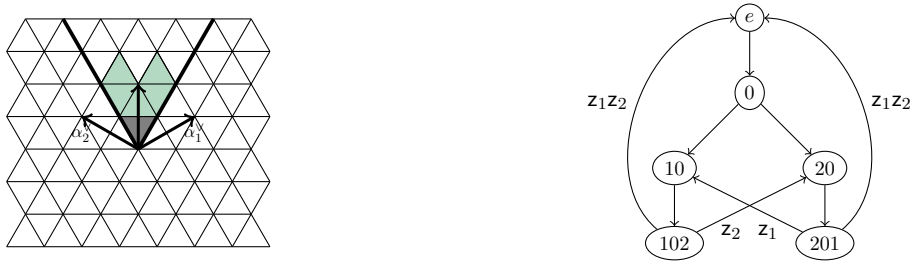
\end{example}

\begin{example}
Let $L=\langle\om_1,2\om_2\rangle_{\Z}$ in type $\tilde{A}_2$. Then $L$ is a sublattice of $P$  of rank 2. The set $\sB$ defined by the alcoves in green and the fundamental alcove $A_0$ in dark grey is a $L$ fundamental domain. We represent the graph with adjacency matrix $\func{Mat}_{\sB}(m_{\xi_{\rho_1}})$ on the right hand-side. We set $\sz_1=\sz^{\om_1}$ and~$\sz_2=\sz^{2\om_2}$. 

\vspace{-.1cm}
\begin{figure}[H]
\begin{center}
\begin{minipage}{5cm}
\psset{linewidth=.05mm}
\begin{center}
\begin{tikzpicture}[scale =.5]
\newcommand{\y}{.866}
\newcommand{\x}{.5}

\draw[fill = ForestGreen!30!] (0,0) -- (-2*\x,2*\y) -- (-\x,3*\y) --  (\x,\y);

\draw[fill = gray] (0,0) -- (-\x,\y) -- (\x,\y);
 
\draw[color = blue, ->,line width = .4mm] (0,0) -- (\x,\y);
\draw[color = blue, ->,line width = .4mm] (0,0) -- (-2*\x,2*\y);

\foreach \n in {-1,...,2}{
\draw (-3,2* \n * \y) -- (3,2 * \n * \y);}

\foreach \n in {-2,...,1}{
\draw (-3.5,2* \n * \y + \y) -- (3.5,2 * \n * \y + \y);}

\foreach \m in {0,...,6}{
\foreach \n in {0,...,3}{
\draw (-3+2*\m*\x,4*\y-\n*2*\y) -- (-3.5+2*\m*\x,3*\y-\n*2*\y);}

\foreach \n in {0,...,2}{
\draw (-3.5+2*\m*\x,3*\y-\n*2*\y) -- (-3+2*\m*\x,2*\y-\n*2*\y);}}

\foreach \m in {0,...,6}{
\foreach \n in {0,...,3}{
\draw (-3.5+2*\m*\x+\x,4*\y-\n*2*\y) -- (-3+2*\m*\x+\x,3*\y-\n*2*\y);}

\foreach \n in {0,...,2}{
\draw (-3+2*\m*\x+\x,3*\y-\n*2*\y) -- (-3.5+2*\m*\x+\x,2*\y-\n*2*\y);}}

\end{tikzpicture}
\end{center}
\end{minipage}
\begin{minipage}{5cm}
\begin{center}
\begin{tikzpicture}[scale =1]
\tikzstyle{vertex}=[inner sep=2pt,minimum size=10pt,ellipse,draw]

\node[vertex] (a1) at (0,0) {\scalebox{.7}{$e$}};
\node[vertex] (a2) at (0,-1) {\scalebox{.7}{$0$}};
\node[vertex] (a3) at (0,-2) {\scalebox{.7}{$10$}};
\node[vertex] (a4) at (0,-3) {\scalebox{.7}{$210$}};

\draw[->] (a1) edge node  {} (a2);
\draw[->] (a2) edge node{} (a3);
\draw[->] (a3) edge node{} (a4);
\draw[->]  (a2) edge[bend left = 70] node[left]{\scalebox{.8}{$\sz_1$}} (a1);
\draw[->]  (a4) edge[bend left = 70] node[left]{\scalebox{.8}{$\sz_1$}} (a3);
\draw[->]  (a4) edge[bend right = 70] node[right]{\scalebox{.8}{$\sz_2$}} (a1);

\end{tikzpicture}
\end{center}
\end{minipage}
\end{center}
\vspace{-.2cm}
\caption{Fundamental domain for $L$ and its associated the graph $\Ga_{\sB}$ in type $\tilde{A}_2$}
\end{figure}
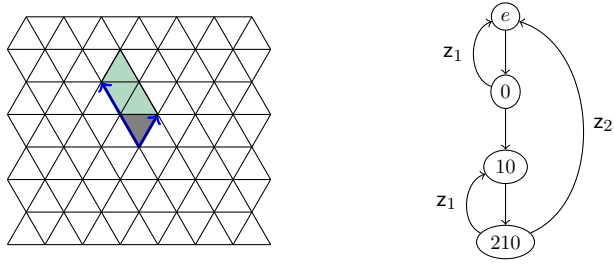
\end{example}

As in the case of $P$, we can prove the following result using similar technics as in \cite{LT2}. Here $\Q(L)$ denotes the fraction field of $\Q[L]$. 
\begin{proposition}
\ \label{PropLM}
Let $L$ be a sublattice of $P$ of rank $n$ and $\sB$ be a $L$-fundamental domain. 
\begin{enumerate}

\item The field $\func{Frac}(\sLa)$ is a finite extension of $\Q(L)$ with degree
$\func{card}(\sB)$. The family $(\xi_{w}\mid w\in\sB)$ is a basis of
$\func{Frac}(\sLa)$ over $\Q(L)$.

\item The element $\xi_{\rho_1}$ is primitive in $\func{Frac}(\sLa)$ as
an extension of $\Q(L)$, i.e we have $\func{Frac}(\sLa)=\Q(L)[\xi_{\rho_1}]$.

\item The graph $\Gamma_{\sB}$ has maximal dimension, that is the
degree of $\mu_{A_{\Ga_{\sB}}}$ is equal to $\mathrm{card}(\Gamma
_{\sB})$.

\end{enumerate}
\end{proposition}

\subsection{Quotient of some subalgebra of the homology algebra}
\label{homlagebra} 
Let $J'\subset \{1,\ldots,n\}$ and let $J$ be the complement of $J'$ in the set $\{1,\ldots,n\}$. Recall the definition of $\cAJ$ in Section~\ref{gammagamma}. In this section, we define a quotient of an ideal of $\sLa_a$ which will have a basis in bijection with $\cAJ$.

\begin{lemma}
Let $\al_j$ be a simple root and let $\cI^+_{\al_j,0}$ be the $\Q$-vector space $\langle\xi_{w}\mid w\subset H^+_{\al_j,0}\rangle_{\Q}$. Then $\cI_{\al_j,0}^+$ is a subalgebra of $\sLa_a$.
\end{lemma}

\begin{proof}
We only need to prove that $\cI^+_{\al_j,0}$ is stable by multiplication. Let $x,y\subset H^+_{\al_j,0}$.  There exists $\la\in \oplus_{k\neq j} \N \om_k$ such that $x\star t_{\la}\in W_a^{\La_0}$. We have $\xi_{x}\xi_y = \xi_{t_{-\la}}\xi_{x\star t_\la}\xi_{y}$. For all $a\in X$, recall that 
$$\xi_{\rho_a}\xi_{w} =\sum_{w',w\overset{+}{\to}w'} a_{w'}\xi_{w'}.$$
Since $x \star t_\la\in W_a^{\La_0}$, we have $\xi_{x \star t_\la}\in \sLa$ and since the $\xi_{\rho_a}$'s generate $\sLa$, we can show that 
$$\xi_{x \star t_\la}\xi_{y} =\sum_{w',y\overset{+}{\to}w'} a_{w'}\xi_{w'}.$$
Now $y\overset{+}{\to}w'$ and $y\subset H^+_{\al_j,0}$  implies that $w'\in H^+_{\al_j,0}$. Since $\la\in \oplus_{k\neq j} \N \om_j$ we get that $w'\star t_{-\la}$ also lies in~$H^+_{\al_j,0}$. Finally $\xi_{x}\xi_y$ lies in $I_{\al_j,0}^+$.
\end{proof}
We set $\cC_{J} = \cap_{j\in J} H_{\al_j,0}^+$ (this is the fundamental Weyl chamber for the group $W_J$). As a consequence of the previous lemma, we see that 
$$\sLa^+_{J,0} = \bigcap_{j\in J} \cI_{\al_j,0}^+ = \langle \xi_{w}\mid w\in \cC_J \rangle_{\Q}$$ is a subalgebra of $\sLa_a$. 

\begin{lemma}
Let $\al\in \Phi_J^+$ and let $\cI^+_{\al,1} = \langle\xi_{w}\mid w\subset H^+_{\al,1}\cap \cC_J\rangle_{\Q}$.
Then  $\cI^+_{\al,1}$ is an ideal of~$\sLa^+_{J,0}$.
\end{lemma}
\begin{proof}
Let $\xi_x\in \sLa^+_{J,0}$ and $\xi_y\in \cI^+_{\al,1}$. There exists $\la\in \oplus_{k\in J'} \N \om_k$ such that $x\star t_{\la}\in W_a^{\La_0}$. We have $\xi_{x}\xi_y = \xi_{t_{-\la}}\xi_{x \star t_\la}\xi_{y}$. Arguying as in the proof of the previous lemma we get
$$\xi_{x \star t_\la}\xi_{y} =\sum_{w',y\overset{+}{\to}w'} a_{w'}\xi_{w'}.$$
Now $y\overset{+}{\to}w'$ implies that $w'\in H_{\al,1}^+\cap \cC_J$. Since $\la\in \oplus_{k\in J'} \N \om_i$, we get that $w'\star t_{-\la}\in H_{\al,1}^+\cap \cC_J$ since the root $\al$ is a positive sum of simple roots in $\{\al_j,j\in J\}$. It follows that $\xi_x\xi_y\in \cI_{\al,1}^+$. 
\end{proof}
\begin{definition}
We define~$\sLa_J$ to be the following quotient:
$$\sLa_{J} := \sLa^+_{J,0}\slash\sum_{\al\in \Phi_J^+}\cI_{\al,1}^+.$$
\end{definition}
It is clear that $\sLa_{J}$ admits $\{\ti{\xi}_{u}\mid u\in \cAJ\}$ as a basis where $\ti{\xi}_{u}$ denotes the image of $\xi_u$ in the quotient $\sLa_J$ and $\cAJ$ is the fondamental $J$-alcove defined in Section \ref{gammagamma}.  
In the case where $J=\emptyset$ and $\ga=\rho$, we get that $\sLa^+_{J,0} = \sLa_a$ and $\sLa_J = \sLa_a$.
The following proposition is a consequence of Proposition \ref{new_pieri}.

\begin{proposition}
\label{mult_GaJ}
For all $w\in \cA_J\cap W_a^{\La_0}$, we have 
$$\tilde{\xi}_{\rho_1}\tilde{\xi}_{w} =\sum_{i\in I, w\overset{+}{\leadsto} s_iw, s_iw\in \cAJ} a_{i}\tilde{\xi}_{s_iw}.$$
\end{proposition}
We have defined an action of $Q^\vee \slash Q_J^\vee$ on $\cAJ$ in Theorem \ref{UJ} and showed that 
$$\cAJ=\{u\add t_{\ov{\al}}\mid u\in W^J, \al\in Q^\vee\}.$$
This action induces a map from $\sLa_J$ to itself defined by 
$\ti{\xi}_{w}\mapsto \ti{\xi}_{w\add t_{\ov{\al}}}$. We also denote this map $t_{\ov{\al}}$ and write $(\ti{\xi}_{w})t_{\ov{\al}} =  \ti{\xi}_{w\add t_{\ov{\al}}}$. Finally, we simply write $\ti{\xi}_{t_{\ov{\al}}}$ instead of $\ti{\xi}_{e\add t_{\ov{\al}}}$ where $e$ is the identity element of $W_a$ (corresponding to the fundamental alcove~$A_0$). 

\smallskip

Our strategy to establish that the graph $\Ga_\rho$ is positively multiplicative is to show that the adjacency matrix of $\Ga_\rho$ is essentially the matrix of $m_{\xi_{\rho_1}}$ in $\sLa_{P}$ in a well chosen basis. We want to do the same thing for $\Ga_\ga$ using the quotient $\sLa_J$, the multiplication by $\ti{\xi}_{\rho_1}$ and the basis $\{\ti{\xi}_w\mid w\in W^J\}$. To do so, we need to prove some factorisation properties in $\sLa_J$ using the element $\ti{\xi}_{t_{\ov{\al}}}$ defined above.

\begin{theorem}
We have $(\tilde{\xi}_{w})t_{\ov{\al}} =\tilde{\xi}_{w}\tilde{\xi}_{t_{\ov{\al}}}=\tilde{\xi}_{t_{\ov{\al}}}\tilde{\xi}_{w}$ for all $w\in W$ and all $\al\in Q^\vee$.
\end{theorem}
\begin{proof}
The map of $m_{\xi_{\rho_1}}$ commutes with $t_{\ov{\al}}$. Indeed, for all $w,u\in \cAJ$  we have $w\crossp u$ if and only if $w\add t_{\ov{\al}}\crossp u\add t_{\ov{\al}}$ by Corollary \ref{crossUJ}. In other words, we have
$$\ti{\xi}_{\rho_1}\cdot (\ti{\xi}_w)t_{\ov{\al}} =  (\ti{\xi}_{\rho_1}\cdot \ti{\xi}_w)t_{\ov{\al}}. $$
 From there, using the generalised Pieri rules, we see that $t_{\ov{\al}}$ commutes with all the $\ti{\xi}_{\rho_a}$ for $a\in X$ and therefore with $\ti{\xi}_x$ for all $x\in \cA_J$. We get for all $w\in \cAJ$
$$\tilde{\xi}_{t_{\ov{\al}}}\cdot \tilde{\xi}_{w} =\tilde{\xi}_{w} \cdot  \tilde{\xi}_{t_{\ov{\al}}}= \tilde{\xi}_{w} \cdot (\tilde{\xi}_e) t_{\ov{\al}}= ( \tilde{\xi}_{w}\tilde{\xi}_e)t_{\ov{\al}} = ( \tilde{\xi}_{w})t_{\ov{\al}} = \xi_{w\add t_{\ov{\al}} }$$
\end{proof}

Let $w=u\add t_{\ov{\al}}$ where $u\in W^J$ and $\al\in Q^\vee$. There exists a unique family $(m_1,\ldots,m_k)$ of integers  such that $\ov{\al} = m_1\ov{\al_{i_1}}+\ldots+m_k\ov{\al_{i_k}}$ where $i_1,\ldots,i_k\in J'$.
It follows from the previous theorem that $\ti{\xi}_w$ with $w\in \cAJ$ can be factorised under the form 
$$\ti{\xi}_w = \ti{\xi}_{t_{\ov{\al_{i_1}}}}^{m_1}\cdot \ti{\xi}_{t_{\ov{\al_{i_2}}}}^{m_2} \cdots \ti{\xi}_{t_{\ov{\al_{i_k}}}}^{m_k}\cdot \xi_{u}.$$
We see that the algebra $\sLa_J$  is finite-dimensional over $\Q[\ti{\xi}_{t_{\ov{\al_{j}}}}^{\pm1}\mid j\in J']$ with basis $\{\ti{\xi}_w\mid w\in W^J\}$. Note that the algebra $\Q[\ti{\xi}_{t_{\ov{\al_{j}}}}^{\pm1}\mid j\in J']$ is isomorphic to the group algebra $\Q[Q^\vee \slash Q^\vee_J]$ via the map 
$$\sz^{\ov{\al}} \mapsto \prod_{j\in J'}\ti{\xi}_{t_{\ov{\al_{j}}}}^{m_j} \qu{where $\ov{\al} = \sum_{j\in J'} m_j\ov{\al_{j}}$.}$$
We will identify these two algebras. 

\smallskip

We define~$\Ga_{W^J}$ to be the graph with set of vertices $W^J$ and adjacency matrix $\func{Mat}_{W^J}(m_{\ti{\xi}_{\rho_1}})$ with coefficients in  $\Z[Q^\vee \slash Q^\vee_J]$. We immediatly get the following result.

\begin{proposition}
\label{pos_ga_ga}
The graph $\Ga_{W^J}$ is positively multiplicative at $e$. 
\end{proposition}

\begin{remark}
From there it is possible to study central random walk in the strip $\cAJ\cap W_a^{\La_0}$. Indeed, let~$\sB_0$ be a subset of alcoves in $\cAJ$ defined by $w\in \sB_0$ if and only if $w\add t_{-\ov{\al}_{j'}}\notin \cAJ\cap W_a^{\La_0}$ for all $j'\in J'$. Then the set $\{\ti{\xi}_{w}\mid w\in \sB_0\}$ is a basis of $\sLa_J$ and  the graph $\Ga_{\sB_0}$ with adjacency matrix $\func{Mat}_{\sB_0}(m_{\ti{\xi}_{\rho_1}})$ is positively multiplicative. It is not hard to see that the expansion of $\Ga_{\sB_0}$ will be the graph of the weak Bruhat order on $\cAJ\cap W_a^{\La_0}$ and therefore the methods of \cite{LT2} should be generalisable to study central randon walks on $\cAJ$.
\end{remark}

\subsection{Automatons for affine Weyl group and random walks}

We conclude this section with some considerations about finite automatons
recognising the reduced expressions of any element in an affine Weyl group
$W$. The existence of such automatons, their minimality and their explicit
description is a difficult problem which has been considered in recent
publications (see in particular \cite{HNW}, \cite{PY} and the references
therein). 

As explained in section \ref{gammagamma}, the results in this paper yield
similar automatons for reduced expressions of affine Grassmanians or affine Grassmanians contain in a strip $\cAJ$.
Moreover in both cases, we
are able to prove that the associated graphs are positively multiplicative. By
the results of \cite{GLT1}, we thus have a complete description of the
positive harmonic functions defined on their expanded version and a good
understanding of the associated central Markov chains. This is in particular
the case for $W_a^{\La_0}$ where these Markov chains can be regarded as
random walks on the alcoves located in the fundamental Weyl chamber which
never cross two times the same hyperplane.

Now, this is a natural question to extend the previous results on random walks
on the alcoves of the fondamental Weyl chamber to similar random walks on the full set of
 alcoves.\ Therefore, since we have finite automatons recognising the
reduced expressions of all the elements in $W_a$, one could ask whether they are
positively multiplicative. In fact this is easy to show this cannot be the
case in general. Indeed, consider the infinite diedral group $W_a=\langle
s_{0},s_{1}\rangle$, that is the affine Weyl group of type $\tilde{A}_{1}$.
The reduced expressions in $W$ have the form $w=(s_{1}s_{0})^{l},w=s_{0}%
(s_{1}s_{0})^{l}$ with $l\geq0$ or $w=(s_{0}s_{1})^{k},w=s_{1}(s_{0}s_{1}%
)^{k}$ with $k\geq0$. It is not difficult to see that the finite automaton
with $3$ vertices $v_{1},v_{2},v_{3},$ start state $v_{1}$, accepting states
$v_{1},v_{2},v_{3},$ and transition matrix%
\[
A=\left(
\begin{array}
[c]{ccc}%
0 & 0 & 0\\
1 & 0 & 1\\
1 & 1 & 0
\end{array}
\right)
\]
recognizes exactly all the previous reduced expressions.\ But this automaton
does not admit any weighted version%
\[
A^{\prime}=\left(
\begin{array}
[c]{ccc}%
0 & 0 & 0\\
a & 0 & d\\
b & c & 0
\end{array}
\right)
\]
(where $a,b,c,d$ positive reals) which could be positively multiplicative To see this, we can apply
Theorem 3.1.2 in \cite{GLT1} showing that the graph with adjacency matrix $A'$
is positively multiplicative if and only if the entries of the matrix%
\[
M=\frac{1}{a^{2}c-b^{2}d}\left(
\begin{array}
[c]{ccc}%
1 & 0 & 0\\
0 & ac & -bd\\
0 & -b & a
\end{array}
\right)
\]
are all nonnegative which is clearly impossible. Observe also that the
previous graph is not simply connected: once the vertex $v_{2}$ or $v_{3}$ is
attained, it becomes impossible to come back at $v_{1}$. 

This indicates that the results of the present paper cannot be extended in a
direct way to the complete affine Weyl group. Therefore the determination of
the positive harmonic functions on the alcove walks supported by all the
alcoves needs a different approach. In the present case, they are particularly
easy to determine due to the very simple structure of the Cayley graph for $W_a$
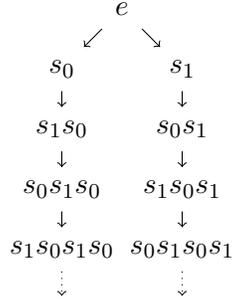
\begin{center}
\begin{tikzpicture}[scale=.8]
\tikzstyle{vertex}=[inner sep=5pt]

\node[vertex] (a1) at (0,0) {\scalebox{1}{$e$}};
\node[vertex] (a21) at (-1,-1) {\scalebox{1}{$s_0$}};
\node[vertex] (a22) at (1,-1) {\scalebox{1}{$s_1$}};

\node[vertex] (a31) at (-1,-2) {\scalebox{1}{$s_1s_0$}};
\node[vertex] (a32) at (1,-2) {\scalebox{1}{$s_0s_1$}};

\node[vertex] (a41) at (-1,-3) {\scalebox{1}{$s_0s_1s_0$}};
\node[vertex] (a42) at (1,-3) {\scalebox{1}{$s_1s_0s_1$}};

\node[vertex] (a51) at (-1,-4) {\scalebox{1}{$s_1s_0s_1s_0$}};
\node[vertex] (a52) at (1,-4) {\scalebox{1}{$s_0s_1s_0s_1$}};

\node[vertex] (a61) at (-1,-5) {};
\node[vertex] (a62) at (1,-5) {};

\draw[->] (a1) edge node  {} (a21);
\draw[->] (a1) edge node  {} (a22);

\draw[->] (a21) edge node{} (a31);
\draw[->] (a31) edge node{} (a41);
\draw[->] (a41) edge node{} (a51);
\draw[->,dotted] (a51) edge node{} (a61);

\draw[->] (a22) edge node{} (a32);
\draw[->] (a32) edge node{} (a42);
\draw[->] (a42) edge node{} (a52);
\draw[->,dotted] (a52) edge node{} (a62);

\end{tikzpicture}
\end{center}

Indeed such an harmonic function $f$ normalized so that $f(e)=1$ satisfies
$f(e)=f(s_{0})+f(s_{1})=1$ and is constant on both subsets of
reduced expressions starting by $s_{0}$ and $s_{1}$. Thus, it is completely
determined by the choice of $f(s_{0})\in]0,1[$ and we have an immediate parametrisation
of the positive harmonic functions by the interval $]0,1[$.\ It would be very
interesting to get a similar parametrisation in the general case of affine
Weyl groups.

\section{Proof of the main Theorems}
In this section we prove that $\Ga_\rho$ and $\Ga_\ga$ are positively multiplicative. To avoid confusion, let us recall how we have defined the different algebras that we have used so far. We started with $\sLa$, the homology ring of affine Grassmannians, which is a commutative $\Q$-algebra with distinguished basis $\{\xi_w\mid w\in W_a^{\La_0}\}$. From there, using the factorisation properties of $\sLa$, we saw that $\sLa$ is a finite-dimensional algebra over~$\Q[P^+]$. This result allowed us to extend the scalars to the ring  $\Q[P]$ and to define a commutative $\Q$-algebra $\sLa_a$ with distinguished basis $\{\xi_w\mid w\in W_a\}$ where $\xi_{w} = \sz^{\wt_{\sB_0}(w)}\xi_{\sf_{\sB_0}(w)}$. Then we showed the following results:
\begin{enumerate}
\item The algebra $\sLa_a$ is a finite-dimensional algebra over $\Q[Q^\vee]$ with basis $\{\xi_w\mid w\in \sB\}$ where~$\sB$ is a  fundamental domain for the action of $Q^\vee$. We denote this algebra by $\sLa_{Q^\vee}$
\item Given $J'=\{i_1,\ldots,i_k\}$ and its complement $J$ in $\{1,\ldots,n\}$, we defined a quotient $\sLa_J$ of some ideal of $\sLa_a$ which is finite-dimensional over $\Q[Q^\vee\slash Q_J^\vee]$ with basis $\{\ti{\xi}_w\mid w\in W^J\}$. 
\end{enumerate}
We use the first assertion to show that $\Ga_\rho$ is positively multiplicative and the second to show that $\Ga_\ga$ where $\ga=\sum_{j\in J'} \om_j$ is positively multiplicative.

\subsection{The graph $\Ga_\rho$}
Let $\bar{\ }$ be the involution of $\Q[Q^{\vee}]$ defined by $\ov{\sz^\al} = \sz^{-\al}$ for all $\al\in Q^\vee$.  Recall that 
\begin{equation*}
\xi_{\rho_1}\xi_{w} =\sum_{i\in I, w\overset{+}{\leadsto} s_iw} a_{i}\xi_{s_iw}\in \sLa_a.
\end{equation*}
By definition the set $W$ is a fundamental domain for the action of $Q^\vee$ so that $\{\xi_w\mid w\in W\}$ is a basis of the algebra $\sLa_{Q^\vee}$. Recall that $\Ga_W$ is the graph with set of vertices $W$ and adjacency matrix $\func{Mat}_{W}(m_{\xi_{\rho_1}})$ where~$m_{\xi_{\rho_1}}$ is the multiplication by $\xi_{\rho_1}$ in $\sLa_{Q^\vee}$ expressed in the basis $\{\xi_w\mid w\in W\}$.

\begin{theorem}
$\func{Mat}_{W}(m_{\xi_{\rho_1}}) = {^t}\ov{A_{\Ga_\rho}}$ and $\Ga_\rho$ is positively multiplicative. 
\end{theorem}
\begin{proof}
Let $x,y\in W$. We have ${^t}\ov{A_{\Ga_\rho}}[x,y] = p\neq 0$ if and only if there is an arrow in $\Ga_\rho$ from $x$ to $y$ of weight $\ov{p}$. There are two cases to consider:
\begin{itemize}
\item[{\tiny $\bullet$}] there exists $i\in I^\ast$ such that $y=s_ix>x$. In this case $\ov{p} = p=a_i \in \N$, $y\overset{+}{\leadsto} x$ and, according to the previous proposition, $a_i\xi_{x}$ will appear when developping $\xi_{\rho_1}\xi_{y}$, showing that  $\func{Mat}_{W}(m_{\xi_{\rho_1}})[x,y] = p$.
\item[{\tiny $\bullet$}] we have $y=s_\theta x<x$. Then $s_0x = \sf(s_0x)t_{\wt(s_0x)} =  yt_{\wt(s_0x)}$ and $\ov{p} = \sz^{\wt(s_0w)}$. By Proposition \ref{Gamma_rho_orientation}, we know that $s_0x\overset{+}{\leadsto}x$ hence showing that the term $\xi_{x}$ will appear when developping $\xi_{\rho_1}\xi_{s_0x}$. We have in $\sLa_a$
$$\xi_{\rho_1}\xi_{s_0x} = \sz^{\wt(s_0x)}\xi_{\rho_1}\xi_y = \xi_x+\cdots$$
hence showing when factorising with respect to the fundamental domain $W$ that $\func{Mat}_{W}(m_{\xi_{\rho_1}})[x,y] = \sz^{-\wt(s_0w)}=p$ as desired. 
\end{itemize}
The fact that $\Ga_\rho$ is positively multiplicative follows easily from the fact that $\Ga_{W}$ is positively multiplicative (see Proposition \ref{fdomain}). Indeed, it suffices to transpose and apply the bar involution to the positive basis associated to $\func{Mat}_{W}(m_{\xi_{\rho_1}})$.
\end{proof}

 \subsection{The graph $\Gamma_\gamma$}
 
Let $\gamma = \om_{i_1}+\cdots+\om_{i_k}$ be a level  $0$ weight. Let $J'=\{i_1,\ldots,i_k\}$ and $J$ be the complement of $J'$ in $\{1,\ldots,n\}$. Let $\bar{\ }$ be the involution of $\Q[Q^{\vee}\slash Q_J^\vee]$ defined by $\ov{\sz^{\ov{\al}}} = \sz^{-\ov{\al}}$ for all $\al\in Q^\vee$.
The set $W^J$ is a fundamental domain for the action of $Q^\vee\slash Q_J^\vee$ on $\cA_J$  so that $\{\ti{\xi}_w\mid w\in W^J\}$ is a basis of the algebra $\sLa_{J}$. Recall that $\Ga_{W^J}$ is the graph with set of vertices $W^J$ and adjacency matrix $\func{Mat}_{W^J}(m_{\ti{\xi}_{\rho_1}})$ where~$m_{\ti{\xi}_{\rho_1}}$ is the multiplication by $\ti{\xi}_{\rho_1}$ in $\sLa_{J}$ expressed in the basis $\{\ti{\xi}_w\mid w\in W^J\}$.
\begin{theorem}
$\func{Mat}_{W^J}(m_{\ti{\xi}_{\rho_1}}) = {^t}\ov{A_{\Ga_\ga}}$ and $\Ga_\ga$ is positively multiplicative. 
\end{theorem}
\begin{proof}
Let $x,y\in W^J$. We have ${^t}\ov{A_{\Ga_\ga}}[x,y] = p\neq 0$ if and only if there is an arrow in $\Ga_\ga$ from $x$ to $y$ of weight $\ov{p}$. There are two cases to consider:
\begin{itemize}
\item[{\tiny $\bullet$}] there exists $i\in I^\ast$ such that $y=s_iw>w$ and $y\in W^J$. In this case $\ov{p} = p=a_i \in \N$, $y\overset{+}{\leadsto} x$ and, according to Proposition \ref{mult_GaJ}, $a_i\ti{\xi}_{x}$ will appear when developping $\ti{\xi}_{\rho_1}\ti{\xi}_{y}$, showing that  $\func{Mat}_{W}(m_{\ti{\xi}_{\rho_1}})[x,y] = p$.
\item[{\tiny $\bullet$}] we have $y=(s_\theta x)^J<x$ and $s_0x\in \cAJ$. 
We have  $s_0x = \sf(s_0x)t_{\wt(s_0x)} =  yvt_{\wt(s_0x)} = y\add t_{\ov{\w(s_0x)}}$ with $v\in W_J$ and $\ov{p} = \sz^{\ov{\wt(s_0w)}}$. By Theorem \ref{cross_strip}, we know that $s_0x\overset{+}{\leadsto}x$ hence showing that the term $\ti{\xi}_{x}$ will appear when developping $\ti{\xi}_{\rho_1}\ti{\xi}_{s_0w}$. We have in $\sLa_J$
$$\ti{\xi}_{\rho_1}\ti{\xi}_{s_0x} = \sz^{\ov{\wt(s_0w)}}\ti{\xi}_{\rho_1}\ti{\xi}_y = \ti{\xi}_x+\cdots.$$
hence showing that $\func{Mat}_{W}(m_{\xi_{\rho_1}})[x,y] = \sz^{-\ov{\wt(s_0w)}}=p$ as desired. 
\end{itemize}
The fact that $\Ga_\ga$ is positively multiplicative follows easily from the fact that $\Ga_{W^J}$ is positively multiplicative (see Proposition \ref{pos_ga_ga}). 
\end{proof}

\section{Interpretation in terms of colored particle models}

In this section, we give an interpretation of the parabolic Cayley graphs of classical types introduced in Section \ref{SecGraphGg} in terms of particle moving on either a circle or a segment. The main tool to provide such interpretations is the combinatorial description of the affine Weyl groups of classical types as permutations of $\mathbb{Z}$, see \cite[Section 8]{BjBr}. 

\newcommand{\e}{\varepsilon}
\subsection{Type $\tilde{A}$}
Let $V$ be an euclidean space of dimension $n+1$ with basis $\e_1,\ldots,\e_{n+1}$. 
The set of vectors 
$$\Phi^+=\{\e_i-\e_j\mid 1\leq i<j\leq n+1\}$$ 
forms a positive root system of type $A$ with associated simple system $\Pi=\{\al_1,\ldots\al_n\}$ where $\al_i=\e_i-\e_{i+1}$. We have $\theta=\theta^\vee = \e_1-\e_{n+1}$. From there, we can contruct the affine Weyl group of $W_{a}$ of type $\tilde{A}_{n}$ as in Section~\ref{SubSecAlcoves}. 

\medskip

The group $W_a$ can be realised as the group of permutations~$\si$ of $\mathbb{Z}$ (acting on the right) that satisfy the following properties:
\begin{itemize}
\item[{\tiny $\bullet$}] $(i+n+1)\sigma=(i)\sigma+n+1$ for all $i\in \mathbb{Z}$,
\item[{\tiny $\bullet$}]  $\sum_{i=1}^{n+1}(i)\sigma=\frac{(n+1)(n+2)}{2}$.
\end{itemize}
Any element $\sigma \in W_a$ is uniquely determined by the images of $1,2,\ldots,n+1$. We will sometime denote an element $\si\in W_a$ by its window notation $[(1)\sigma,\ldots ,(n+1)\sigma]$. The distinguished generators $s_0,s_1,\ldots,s_n$ of $W_a$ are given by $s_i=[1,\ldots,i+1,i,\ldots,n+1]$ and $s_{0}=[0,2,\ldots,n,n+2]$.

\medskip

The symmetric group $W$ is the subgroup of $W_a$ corresponding to the sequences $[j_1,\ldots ,j_{n+1}]$ with $1\leq j_i\leq n+1$ for $1\leq i\leq n+1$. We have the following formula for translations and affine symmetries
\begin{align*}
t_{\alpha_i}&=[1,\ldots, (i+n+1),(i-n),\ldots, n+1],\\
s_{\alpha_i,k}&=[1,\ldots ,(i+1)+k(n+1),i-k(n+1) ,\dots ,n+1].
\end{align*} 
Given an element $w=[j_1,\ldots,j_{n+1}]$, left multiplication by a distinguished generator is given by: 
\begin{align*}
s_i[j_1,\ldots,j_{n+1}] &= [j_1,\ldots,j_{i+1},j_i,\ldots,j_{n+1}]\qu{and}\\
s_0[j_1,\ldots,j_{n+1}] &= [j_{n+1}-(n+1),\ldots,j_{i+1},j_i,
\ldots,j_{1}+(n+1)].
\end{align*}
Right mutliplication by $\si$ is straightforward by definition 
$$[j_1,\ldots,j_{n+1}]\si = [(j_1)\si,\ldots,(j_{n+1})\sigma].$$
Let $w\in W$. The equality $s_0w = s_\theta t_{\theta} w = s_\theta w t_{w(\theta)}$ in $W_a$  becomes when $j_{n+1}<j_1$
$$s_0 [j_1,j_2,\ldots,j_{n+1}]= [j_{n+1},j_{2}\ldots,j_n,j_{1}]t_{-\sum_{r=j_{n+1}}^{j_1-1}\alpha_r}.$$
Further for $\si\in W$, we have $\ell(s_i\si)>\ell(\si)$ if and only if $j_i<j_{i+1}$ and $\ell(s_\theta \si)<\ell(\si)$ if and only if $j_1<j_{n+1}$. Hence, the edge structure on $\Gamma_{\rho}$ is given by arrows 
\begin{itemize}
\item[{\tiny $\bullet$}] $[j_1,\ldots,j_i,j_{i+1},\ldots ,j_{n+1}] \overset{i}{\underset{1}{\longrightarrow}} [j_1,\ldots,j_{i+1},j_{i},\ldots ,j_{n+1}]$ if $j_{i+1}>j_{i}$ and $1\leq i\leq n$,
\item[{\tiny $\bullet$}] $[j_1,\ldots,j_{n+1}] \overset{0}{\underset{z^{-\alpha}}{\longrightarrow}}[j_{n+1},\ldots,j_1]$ with $\alpha=\sum_{s=j_{n+1}}^{j_1-1}\alpha_s$ if $j_{n+1}<j_{1}$.
\end{itemize}
\subsubsection*{The graphs $\Gamma_{\gamma}$}  
 Let $J=\{i_1,\ldots,i_\ell\}\subset \{1,\ldots n\}$, $J'$ be the complement of $J$ in $\{1,\ldots n\} $  and  $\gamma=\sum_{j\in J'}\omega_{j}$ be the corresponding dominant weight. The group $W_J$ acts on the $\{1,\ldots,n+1\}$ and the set of orbits can be described as $\{B_{m}\mid 1\leq m\leq r\}$ where all the elements in $B_m$ are smaller than those in $B_{m+1}$.
  The set of minimal length representatives $W^J$ is the subset of $W$ corresponding to sequences $[j_1,\ldots,j_{n+1}]$ where $i$ appears before $j$ whenever $i<j$ and $i,j\in B_m$ for some $1\leq m\leq r$. The definition of $\Gamma_{\gamma}$ given in Definition \ref{defGgam} (see also Subsection \ref{SubsecKeys} for another labelling by Key tableaux) translates then into the following rules:
\begin{itemize}
\item[{\tiny $\bullet$}] $[j_1,\ldots,j_i,j_{i+1},\ldots ,j_{n+1}] \underset{1}{\overset{i}{\longrightarrow}} [j_1,\ldots,j_{i+1},j_{i},\ldots ,j_{n+1}]$ if $j_{i+1}>j_{i}$ and $j_{i},j_{i+1}$ belong to different orbits $B_{m}, B_{m'}$,
\item[{\tiny $\bullet$}]  and $[j_1,\ldots,j_{n+1}] \overset{0}{\underset{z^{-\ov{\alpha}}}{\longrightarrow}} \left([j_{n+1},\ldots,j_1]\right)^J$ with $\alpha=\sum_{s=j_{n+1}}^{j_1-1}\alpha_s$ if $j_{n+1}<j_{1}$ and $j_{1},j_{n+1}$ belong to different orbits $B_{m}, B_{m'}$.
\end{itemize}
Recall here that the weights of $\Ga_\ga$ can be regarded as elements of $Q^\vee\slash Q^\vee_J$ and that for $\al\in Q^\vee$,  $\ov{\al}$ denotes the class of $\al$ modulo $Q^\vee_J$. In the last case, the notation $([k_1,\ldots,k_{n+1}])^J$ means that we take the mininal coset representative with respect to the set $J$.

\medskip

\renewcommand{\sc}{\mathsf{c}}
This graph is isomorphic to a graph encoding the clockwise movement of $n+1$ colored particles on a discrete circle with non-intersecting conditions. Let $\sc:\llbracket 1,n\rrbracket\rightarrow \llbracket 1, r\rrbracket$ be the color map defined by  $i\in B_{\sc(i)}$ for all $i\in \{1,\ldots,n\}$, and let $\mathcal{W}_{J}$ be the set of words in the alphabet $\{1,\ldots,r\}$ with $|B_c|$ occurrences of each letter $c$. Then, $\Gamma_{\gamma}$ is isomorphic to the graph $\widetilde{\Gamma}_{\gamma}$ with vertices $\mathcal{W}_J$ and edge rules: 
\begin{itemize}
\item[{\tiny $\bullet$}]  $c_1\ldots c_ic_{i+1}\cdots c_{n+1}\overset{i}{\underset{1}{\longrightarrow}}c_1\cdots c_{i+1}c_i\cdots c_{n+1}$ if $c_i<c_{i+1}$,
\item[{\tiny $\bullet$}]  $c_1\ldots c_{n+1}\overset{0}{\underset{\prod_{c=c_{n+1}}^{c_1-1}\sz^{-\ov{\alpha_ {k_c}}}}{\longrightarrow}} c_{n+1}\cdots c_1$ if $c_{n+1}<c_1$.
\end{itemize}
\begin{example}
Let $W_a$ be the affine Weyl group $\tilde{A}_n$ and the weight $\gamma=\omega_{i}$ for some $1\leq i\leq n-1$. We have $J=\{1,\ldots,n\}\setminus \{i\}$, $J'=\{i\}$, $B_1=\{ 1,\ldots,i\}$ and $B_2=\{i+1,\ldots,n+1\}$. The color map is defined by $\sc(i) = 1$ if $i\in \llbracket 1,i\rrbracket$ and $\sc(i)=2$ otherwise.  The graph $\widetilde{\Gamma}_{\gamma}$ is the graph with vertices being words with $i$ letters $1$ and $n+1-i$ letters $2$ and edge rules 
\begin{itemize}
\item[{\tiny $\bullet$}] $\cdots 12\cdots \overset{i}{\underset{1}{\longrightarrow}}\ldots 21\cdots $ ,
\item[{\tiny $\bullet$}] and $2c_2\cdots c_{n-1} 1\overset{0}{\underset{\sz^{-\ov{\al_i}}}{\longrightarrow}} 1c_2\cdots c_{n-1}2$.
\end{itemize}
Considering letters $2$ as holes and letters $1$ as particles, $\widetilde{\Gamma}_{\gamma}$ encodes thus the  clockwise movement of $i$ non-intersecting particles on a discrete circle of size $n+1$, with an edge weighted by $\sz^{-\ov{\al_i}}$ when a particle returns to the origin. Such graphs have been studied from a probabilistic point of view in \cite{GLT3}.
\end{example}

\begin{example}
Let $W_a$ be the affine Weyl group $\tilde{A}_3$ and let $\ga=\om_1+\om_2$. We have $J=\{3\}$, $J'=\{1,2\}$,  $B_1=\{1\}$, $B_2=\{2\}$ and $B_3=\{3,4\}$. The color map is given by $\sc(1)=1,\sc(2)=2$ and $\sc(3)=\sc(4)=3$. The graph $\widetilde{\Gamma}_{\gamma}$ is the graph with vertices being words of length 4 with one letter $1$, one letter $2$ and two letters $3$. In the graph below, the letter 1 is encoded by a blue particle, the letter 2 by a green particle and the color 3 by a hole (in gray) and we read the word starting from the north position and moving counterclockwise. 
For instance, we have the following correspondance 
$$\begin{tikzpicture}[baseline={([yshift=-.5ex]current bounding box.center)}]
\def\d{.25}
\def\r{.15}
\draw (0,0) circle (\d cm) ;
\node[fill = gray,circle,inner sep=0pt,minimum size=\r cm] at (0,\d) {};
\node[fill = Blue,circle,inner sep=0pt,minimum size=\r cm] at (-\d,0) {};
\node[fill = gray,circle,inner sep=0pt,minimum size=\r cm] at (0,-\d) {};
\node[fill = Green,circle,inner sep=0pt,minimum size=\r cm] at (\d,0) {};
\end{tikzpicture}
\quad \leftrightarrow \quad 3132.$$
We obtain the following graph where the arrows going up are all of type $0$ and are the only arrows with a weight different from 1. Compare with Example \ref{tableaux}.
%
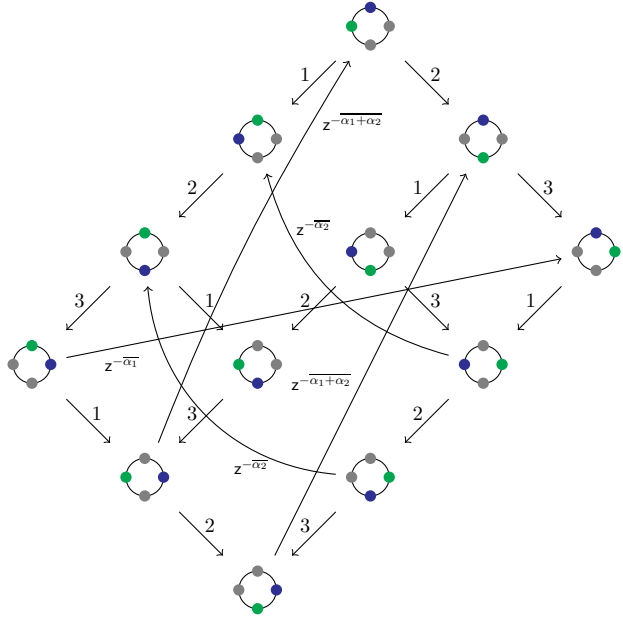
\begin{figure}[H]
$$
\begin{tikzpicture}[scale =1.5]
\tikzstyle{vertex}=[]
\def\d{.25}
\def\r{.15}
\node[vertex] (a1) at (0,0) {\begin{tikzpicture}
\draw (0,0) circle (\d cm) ;
\node[fill = Blue,circle,inner sep=0pt,minimum size=\r cm] at (0,\d) {};
\node[fill = Green,circle,inner sep=0pt,minimum size=\r cm] at (-\d,0) {};
\node[fill = gray,circle,inner sep=0pt,minimum size=\r cm] at (0,-\d) {};
\node[fill = gray,circle,inner sep=0pt,minimum size=\r cm] at (\d,0) {};
\end{tikzpicture}};

\node[vertex] (a2) at (-1,-1)  {\begin{tikzpicture}
\draw (0,0) circle (\d cm) ;
\node[fill = Green,circle,inner sep=0pt,minimum size=\r cm] at (0,\d) {};
\node[fill = Blue,circle,inner sep=0pt,minimum size=\r cm] at (-\d,0) {};
\node[fill = gray,circle,inner sep=0pt,minimum size=\r cm] at (0,-\d) {};
\node[fill = gray,circle,inner sep=0pt,minimum size=\r cm] at (\d,0) {};
\end{tikzpicture}};

\node[vertex] (a3) at (1,-1) {\begin{tikzpicture}
\draw (0,0) circle (\d cm) ;
\node[fill = Blue,circle,inner sep=0pt,minimum size=\r cm] at (0,\d) {};
\node[fill = gray,circle,inner sep=0pt,minimum size=\r cm] at (-\d,0) {};
\node[fill = Green,circle,inner sep=0pt,minimum size=\r cm] at (0,-\d) {};
\node[fill = gray,circle,inner sep=0pt,minimum size=\r cm] at (\d,0) {};
\end{tikzpicture}};

\node[vertex] (a4) at (-2,-2) {\begin{tikzpicture}
\draw (0,0) circle (\d cm) ;
\node[fill = Green,circle,inner sep=0pt,minimum size=\r cm] at (0,\d) {};
\node[fill = gray,circle,inner sep=0pt,minimum size=\r cm] at (-\d,0) {};
\node[fill = Blue,circle,inner sep=0pt,minimum size=\r cm] at (0,-\d) {};
\node[fill = gray,circle,inner sep=0pt,minimum size=\r cm] at (\d,0) {};
\end{tikzpicture}};

\node[vertex] (a5) at (0,-2) {\begin{tikzpicture}
\draw (0,0) circle (\d cm) ;
\node[fill = gray,circle,inner sep=0pt,minimum size=\r cm] at (0,\d) {};
\node[fill = Blue,circle,inner sep=0pt,minimum size=\r cm] at (-\d,0) {};
\node[fill = Green,circle,inner sep=0pt,minimum size=\r cm] at (0,-\d) {};
\node[fill = gray,circle,inner sep=0pt,minimum size=\r cm] at (\d,0) {};
\end{tikzpicture}};

\node[vertex] (a6) at (2,-2){\begin{tikzpicture}
\draw (0,0) circle (\d cm) ;
\node[fill = Blue,circle,inner sep=0pt,minimum size=\r cm] at (0,\d) {};
\node[fill = gray,circle,inner sep=0pt,minimum size=\r cm] at (-\d,0) {};
\node[fill = gray,circle,inner sep=0pt,minimum size=\r cm] at (0,-\d) {};
\node[fill = Green,circle,inner sep=0pt,minimum size=\r cm] at (\d,0) {};
\end{tikzpicture}};

\node[vertex] (a7) at (-3,-3) {\begin{tikzpicture}
\draw (0,0) circle (\d cm) ;
\node[fill = Green,circle,inner sep=0pt,minimum size=\r cm] at (0,\d) {};
\node[fill = gray,circle,inner sep=0pt,minimum size=\r cm] at (-\d,0) {};
\node[fill = gray,circle,inner sep=0pt,minimum size=\r cm] at (0,-\d) {};
\node[fill = Blue,circle,inner sep=0pt,minimum size=\r cm] at (\d,0) {};
\end{tikzpicture}};

\node[vertex] (a8) at (-1,-3) {\begin{tikzpicture}
\draw (0,0) circle (\d cm) ;
\node[fill = gray,circle,inner sep=0pt,minimum size=\r cm] at (0,\d) {};
\node[fill = Green,circle,inner sep=0pt,minimum size=\r cm] at (-\d,0) {};
\node[fill = Blue,circle,inner sep=0pt,minimum size=\r cm] at (0,-\d) {};
\node[fill = gray,circle,inner sep=0pt,minimum size=\r cm] at (\d,0) {};
\end{tikzpicture}};

\node[vertex] (a9) at (1,-3){\begin{tikzpicture}
\draw (0,0) circle (\d cm) ;
\node[fill = gray,circle,inner sep=0pt,minimum size=\r cm] at (0,\d) {};
\node[fill = Blue,circle,inner sep=0pt,minimum size=\r cm] at (-\d,0) {};
\node[fill = gray,circle,inner sep=0pt,minimum size=\r cm] at (0,-\d) {};
\node[fill = Green,circle,inner sep=0pt,minimum size=\r cm] at (\d,0) {};
\end{tikzpicture}};

\node[vertex] (a10) at (-2,-4){\begin{tikzpicture}
\draw (0,0) circle (\d cm) ;
\node[fill = gray,circle,inner sep=0pt,minimum size=\r cm] at (0,\d) {};
\node[fill = Green,circle,inner sep=0pt,minimum size=\r cm] at (-\d,0) {};
\node[fill = gray,circle,inner sep=0pt,minimum size=\r cm] at (0,-\d) {};
\node[fill = Blue,circle,inner sep=0pt,minimum size=\r cm] at (\d,0) {};
\end{tikzpicture}};

\node[vertex] (a11) at (0,-4) {\begin{tikzpicture}
\draw (0,0) circle (\d cm) ;
\node[fill = gray,circle,inner sep=0pt,minimum size=\r cm] at (0,\d) {};
\node[fill = gray,circle,inner sep=0pt,minimum size=\r cm] at (-\d,0) {};
\node[fill = Blue,circle,inner sep=0pt,minimum size=\r cm] at (0,-\d) {};
\node[fill = Green,circle,inner sep=0pt,minimum size=\r cm] at (\d,0) {};
\end{tikzpicture}};

\node[vertex] (a12) at (-1,-5) {\begin{tikzpicture}
\draw (0,0) circle (\d cm) ;
\node[fill = gray,circle,inner sep=0pt,minimum size=\r cm] at (0,\d) {};
\node[fill = gray,circle,inner sep=0pt,minimum size=\r cm] at (-\d,0) {};
\node[fill = Green,circle,inner sep=0pt,minimum size=\r cm] at (0,-\d) {};
\node[fill = Blue,circle,inner sep=0pt,minimum size=\r cm] at (\d,0) {};
\end{tikzpicture}};

\draw[->] (a1) edge node[left,above,pos=.7]  {\scalebox{.7}{$1$}} (a2);
\draw[->]  (a1) edgenode[left,above,pos=.7]  {\scalebox{.7}{$2$}} (a3);

\draw[->]  (a2) edge node[left,above,pos=.7]  {\scalebox{.7}{$2$}} (a4);
\draw[->] (a3) edge node[left,above,pos=.7]  {\scalebox{.7}{$1$}} (a5);
\draw[->] (a3) edge node[left,above,pos=.7]  {\scalebox{.7}{$3$}} (a6);

\draw[->]  (a4) edge node[left,above,pos=.7]  {\scalebox{.7}{$3$}} (a7);
\draw[->]  (a4) edge node[left,above,pos=.7]  {\scalebox{.7}{$1$}} (a8);

\draw[->]  (a5) edge node[left,above,pos=.7]  {\scalebox{.7}{$2$}} (a8);
\draw[->]  (a5) edge node[left,above,pos=.7]  {\scalebox{.7}{$3$}} (a9);

\draw[->]  (a6) edge node[left,above,pos=.7]  {\scalebox{.7}{$1$}} (a9);

\draw[->]  (a7) edge node[left,above,pos=.7]  {\scalebox{.7}{$1$}} (a10);

\draw[->]  (a8) edge node[left,above,pos=.7]  {\scalebox{.7}{$3$}}(a10);

\draw[->]  (a9) edge node[left,above,pos=.7]  {\scalebox{.7}{$2$}}(a11);

\draw[->]  (a10) edge node[left,above,pos=.7]  {\scalebox{.7}{$2$}} (a12);
\draw[->]  (a11) edge node[left,above,pos=.7]  {\scalebox{.7}{$3$}} (a12);

\draw[->] (a10) edge[bend left = 5] node[pos=.85,right]{\scalebox{.6}{$\sz^{-\ov{\al_1+\al_2}}$}} (a1);
\draw[->] (a9) edge[bend left = 30] node[pos=.8,right]{\scalebox{.6}{$\sz^{-\ov{\al_2}}$}} (a2);
\draw[->] (a12) edge[bend left = 0] node[pos=.45,left]{\scalebox{.6}{$\sz^{-\ov{\al_1+\al_2}}$}}(a3);
\draw[->] (a11) edge[bend right = -40] node[pos=.3,below]{\scalebox{.6}{$\sz^{-\ov{\al_2}}$}} (a4);
\draw[->] (a7) edge[bend left = 0] node[pos=.1,below]{\scalebox{.6}{$\sz^{-\ov{\al_1}}$}} (a6);
\end{tikzpicture}$$
\caption{The graph $\widetilde{\Gamma}_{\om_1+\om_2}$ in type $\tilde{A}_3$}
\end{figure}
\end{example}

\subsection{Type $\tilde{C}$}
Let $V$ be an euclidean space of dimension $n$ with basis $\e_1,\ldots,\e_{n}$. 
The set of vectors 
$$\Phi^+=\{\e_i\pm \e_j\mid 1\leq i<j\leq n-1\}\cup \{2\e_i\mid 1\leq i\leq n\}$$ 
forms a positive root system of type $C$ with associated simple system $\Pi=\{\al_1,\ldots\al_{n-1},\al_n\}$ where $\al_i=\e_i-\e_{i+1}$ for $1\leq i\leq n-1$ and $\al_n=2\e_n$. We have $\theta=2\e_1$ and $\theta^\vee = \e_1$. From there we can contruct the affine Weyl group of $W_{a}$ of type $\tilde{C}_{n}$ as in Section \ref{SubSecAlcoves}.

\medskip

The group $W_a$ can be realised as the group of permutations~$\si$ of $\mathbb{Z}$ (acting on the right) that satisfies the following properties
\begin{itemize}
\item[{\tiny $\bullet$}] $(-i)\sigma=-(i)\sigma$ for $i\in\mathbb{Z}$,
\item[{\tiny $\bullet$}] $(i+N)\sigma=(i)\sigma+N$
\end{itemize}
where we have set $N=2n+1$. We see that any element $\si\in W_a$ is completly determined by the image of $1,\ldots,n$. Indeed we have $\si(kN) = kN$ for all $k\in \Z$ and if $n+1\leq i\leq 2n$ we have
$$(i)\si = (N - (N - i))\si = N-(N-i)\si  \qu{and} 1\leq N-i\leq n.$$
For later use, we define the following map 
$$\begin{array}{cccccccc}
\mathsf{s_C}&:& \{1,\ldots,2n\} &\to& \{1,\ldots,2n\}\\
&& i &\mapsto & N-i
\end{array}.$$
We have $(i)\si = -(\syc(i))\si + N$ for all $i\in \{1,\ldots,2n\}$.  
 We will sometime denote an element $\si\in W_a$ by its window notation: $[(1)\sigma,\ldots ,(n)\sigma]$.  The window notation of the distinguished generators $s_0,s_1,\ldots,s_n$ of~$W_a$ is 
\begin{align*}
s_0 &= [-1,2,\ldots,n],\\
s_i &= [1,\ldots,i+1,i,\ldots,n] \qu{for all $1\leq i\leq n-1$,}\\
s_n &= [1,2,\ldots,n-1,n+1].
\end{align*}
The finite Weyl group $W$ is the subgroup of $W_a$ corresponding to the sequences $[j_1\ldots j_n]$ with $1\leq j_i \leq 2n$ for $1\leq i\leq n$. We have the following formula for translations 
\begin{align*}
t_{\al_0} &= [1+N,2,\ldots,n],\\
t_{\al_i}&=[1,\ldots ,(i+1)+N,i-N ,\dots ,n] \qu{for $1\leq i \leq n-1$,}\\
t_{\al_n} &= [1,2,\ldots,n+N].
\end{align*} 
Given an element $w=[j_1,\ldots,j_{n}]$, we have the following rules for right multiplication: 
\begin{align*}
s_0[j_1,\ldots,j_n] &=[-j_1,\ldots,j_n],\\
s_i[j_1,\ldots,j_n] &= [j_1,\ldots,j_{i+1},j_i,\ldots,j_n]\qu{for $1\leq i \leq n+1$,}\\
s_n[j_1,\ldots,j_n] &=[j_1,\ldots,N-j_{n}].
\end{align*} 
Let $w\in W$. We have $s_\theta w<w$ if and only if $j_1>n$ and the equality $s_0w = s_\theta t_{\theta^\vee} w = s_\theta w t_{(\theta^\vee)w}$ in $W_a$ becomes in this case: 
$$s_0 [j_1,j_2,\ldots,j_{n}]= [N-j_1,\ldots,j_{n}]t_{\alpha^{\vee}} \qu{where} \alpha^{\vee}=-\al_n^\vee-\sum_{i=N-j_1}^{n-1}\alpha_i^{\vee}.$$
Hence, the edge structure on $\Gamma_{\rho}$ is given by 
\begin{itemize}
\item[{\tiny $\bullet$}] $[j_1,\ldots,j_i,j_{i+1},\ldots ,j_n] \overset{i}{\underset{2}{\longrightarrow}} [j_1,\ldots,j_{i+1},j_{i},\ldots ,j_n]$ if $j_{i+1}>j_{i}$ and $1\leq i \leq n-1$
\item[{\tiny $\bullet$}]  $[j_1,\ldots,j_n] \overset{n}{\underset{1}{\longrightarrow}} [j_1,\ldots,N-j_n]$ if $j_{n}\leq n$.
\item[{\tiny $\bullet$}] $[j_1,\ldots,j_n] \underset{z^{-\alpha^{\vee}}}{\longrightarrow} [N-j_1,\ldots,j_n]$   if $j_{1}>n$, with $\alpha^{\vee}=\al_n+\sum_{i=N-j_1}^{n-1}\alpha_i^{\vee}$.
\end{itemize}

\subsection*{The graph $\Gamma_{\gamma}$}
Let $J=\{i_1,\ldots,i_\ell\}\subset \{1,\ldots n\}$, $J'$ be the complement of $J$ in $\{1,\ldots n\} $  and  $\gamma=\sum_{j\in J'}\omega_{j}$ be the corresponding weight. The group $W_J$ acts on  $\{1,\ldots,2n\}$. In the case where $n\notin J_r$, the orbits of this action can be described as follows:
\begin{align*}
B_m&=\{k_{m},\ldots,k_{m+1}-1\}\quad\text{for $m=1,\ldots,r$},\\ 
B_{2r-m}&=\{\syc(k_{m+1}-1),\ldots,\syc(k_{m})\}\quad\text{for $m=1,\ldots,r$}
\end{align*}
where $k_{n+1}=n+1$ and $k_m<k_{m+1}$. Note that we must have $n\in B_r$ and $n+1\in B_{r+1}$ in this case. In the case where $n\in J_r$ we have
\begin{align*}
B_m&=\{k_{m},\ldots,k_{m+1}-1\}\quad\text{for $m=1,\ldots,r-1$}\\ 
B_{2r-m}&=\{\syc(k_{m+1}-1),\ldots,\syc(k_{m})\}\quad\text{for $m=1,\ldots,r-1$}\\
B_{r} &=\{k_{r},\ldots,n,\syc(n),\ldots,\syc(k_{r})\}.
\end{align*}
where $k_0=0$ and $k_m<k_{m+1}$. Note that we must have $\{n,n+1\}\subset B_r$ in this case.

\smallskip

The set of minimal length representatives $W^J$ is the subset of $W$ corresponding to sequences $[j_1,\ldots,j_{n}]$ such that 
(1) $i$ appears before $j$ whenever $i<j$ and $i$ and $j$ belong to the same orbit $B_k$ and~(2) if $n\in J$ then $B_r\cap \{j_1,\ldots,j_n\} \subset \{1,\ldots,n\}$. The definition of $\Gamma_{\gamma}$ given in \ref{defGgam} translates then into the following rules:
\begin{itemize}
\item[{\tiny $\bullet$}]  $[j_1,\ldots,j_i,j_{i+1},\ldots ,j_n] \overset{i}{\underset{2}{\longrightarrow}} [j_1,\ldots,j_{i+1},j_{i},\ldots ,j_n]$ if $j_{i+1}>j_{i}$ and $j_{i},j_{i+1}$ belong to two different orbits~$B_{c}, B_{c'}$
\item[{\tiny $\bullet$}]  $[j_1,\ldots,j_n] \overset{n}{\underset{1}{\longrightarrow}} [j_1,\ldots,N-j_n]$ with $j_n\leq n$ and $j_n\notin B_r$ if $n\in J$
\item[{\tiny $\bullet$}]  $[j_1,\ldots,j_n] \overset{0}{\underset{z^{-\ov{\alpha^{\vee}}}}{\longrightarrow}}\bigg([N-j_1,\ldots,j_n]\bigg)^{J}$ if $j_{1}>n$, with $\alpha^{\vee}=\al^\vee_n+\sum_{i=N-j_1}^{n-1}\alpha_i^{\vee}$
\end{itemize}
In the last case, the notation $([k_1,\ldots,k_n])^J$ means that we take the mininal coset representative with respect to the set $J$. 

\smallskip

This graph is isomorphic to a graph encoding the movement on a segment of colored particles with a spin $\{+,-\}$ which exchange their position according to their spin and color. Let $\sc:\{1,\ldots,n\}\to \{1,\ldots,r\}$ be the color map defined by~$i\in B_{\sc(i)}\cup B_{2r-\sc(i)}$ and $\mathsf{sp}:\{1,\ldots,2n\}\to \{ -,+\}$ be the spin map defined by  $\mathsf{sp}(i) = -$ if $1\leq i\leq n $ and $\mathsf{sp}(i)=+$ otherwise.

When $n\notin J$, let $\mathcal{W}_{J}$ be the set of words in $\{1^\pm,\ldots,r^\pm\}$ with $|B_c|$ occurrences of the letter~$c$ (i.e. either $c^+$ or $c^-$). When $n\in J$, let $\mathcal{W}_{J}$ be the set of words in $\{1^\pm,\ldots,r-1^\pm,r\}$ with $|B_c|$ occurrences of the letter $c$ for $c\leq r-1$ and $|B_r|/2$ occurences of the letter $r$. Set $q_c=\sz^{-\ov{\alpha^\vee}}$ with $\alpha^{\vee}=\alpha_n^{\vee}+\sum_{\substack{1\leq r\leq c}}\alpha_r^{\vee}$. Then, $\Gamma_{\gamma}$ is isomorphic to the graph $\widetilde{\Gamma}_{\gamma}$ with vertices $\mathcal{W}_J$ and edge rules
\begin{itemize}
\item[{\tiny $\bullet$}]  $\cdots c_i^{\mathsf{sp}_i}c_{i+1}^{\mathsf{sp}_{i+1}}\cdots x_n\overset{i}{\underset{2}{\longrightarrow}}x_1\cdots c_{i+1}^{\mathsf{sp}_{i+1}}c_i^{\mathsf{sp}_i}\cdots x_n$ if $\mathsf{sp}_{i}(N/2-c_i)<\mathsf{sp}_{i+1}(N/2-c_{i+1})$,
\item[{\tiny $\bullet$}]  $x_1\cdots x_{n-1}c_n^- \overset{n}{\underset{1}{\longrightarrow}} x_1\cdots x_{n-1}c_n^+$ if $c_n\not= r$ when $n \in J$,
\item[{\tiny $\bullet$}] $c_1^+x_2\cdots x_{n} \overset{0}{\underset{q_{c_n}}{\longrightarrow}} c_1^-x_2\cdots x_{n}$ if $c_1\neq r$ when $n \in J$.
\end{itemize}

The first condition in the above rules can be fulfilled in two different ways: either $\mathsf{sp}_{i}=-$ and $\mathsf{sp}_{i+1}=+$, or $\mathsf{sp}_{i}=\mathsf{sp}_{i+1}$ and $\mathsf{sp}_{i}c_i> \mathsf{sp}_{i}c_{i+1}$. In the simple case where $r=2$ and $n\in J$, this graph encodes the movement of $\vert B_1\vert$ particles with a spin $\{+,-\}$ on a discrete segment of size $n$, where particles with spin $-$ go to the right and particles with spin $+$ go to the left.

\begin{example}
Let $W_a$ be the affine Weyl group $\tilde{C}_3$ and let $\ga=\om_2$. We have $J=\{1,3\}$ and $J'=\{2\}$. The group $W_J$ acts on $\{1,2,3,4,5,6\}$  and the orbits are $B_1=\{1,2\}$, $B_2=\{3,4\}$ and $B_3=\{5,6\}$. The elements of $W^J$ are 
\begin{align*}
&e = [1,2,3],\
s_2 = [1,3,2],\
s_1s_2 = [3,1,2],\
s_3s_2 = [1,3,5],\
s_2s_3s_2 = [1,5,3],\\
&s_1s_3s_2=[3,1,5],\
s_1s_2s_3s_2=[5,1,3],\
s_2s_1s_3s_2=[3,5,1],\
s_1s_2s_1s_3s_2=[5,3,1],\\
&s_3s_2s_1s_3s_2=[3,5,6],\
s_1s_3s_2s_1s_3s_2=[5,3,6],\
s_2s_1s_3s_2s_1s_3s_2=[5,6,3].
 \end{align*}
 We have $s_\theta = s_1s_2s_3s_2s_1$. In the list above, the elements that satisfy $s_\theta w<w$ are those with window notation starting by $5$. For all such $w$, we have $(\theta^\vee)w=-\eps_2 = -(\al^\vee_3+\al^\vee_2)$ and  
$s_0w = s_\theta t_{\theta^\vee}w=s_\theta w t_{-\e_2}$
 so that all the $0$-arrows will have weight $\sz^{-\ov{\e_2}}$. We thus omit this information when drawing the $\Ga_\ga$  in the figure below. The color map is defined by $\sc(1) = \sc(2) =1$ and $\sc(3)=2$. Since $3\in J$, the particle $2$ doesn't have spin.   In the graph $\tilde{\Ga}_{\om_2}$ below, the particle $1^-$ is represented in blue, $1^+$ in green and $2$ in gray. We do not indicate the weights $a_i\in \N$ but we do indicate the $i$ when $1\leq i\leq n$.  
 
\vspace{-1cm}

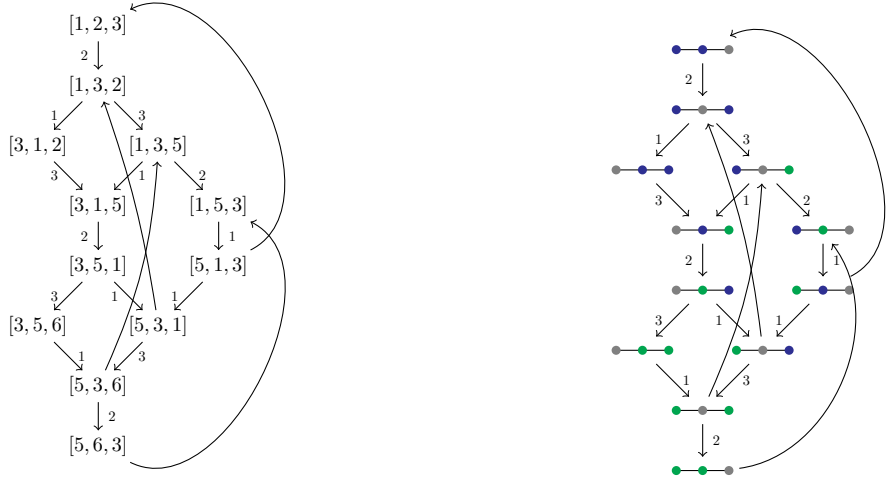
\begin{figure}[H]
$$
\begin{minipage}{8cm}
$$\begin{tikzpicture}[scale =.8]
\tikzstyle{vertex}=[inner sep=2pt,minimum size=10pt]
\node[vertex] (a1) at (0,0) {\scalebox{.7}{$[1,2,3]$}};
\node[vertex] (a2) at (0,-1) {\scalebox{.7}{$[1,3,2]$}};
\node[vertex] (a3) at (-1,-2) {\scalebox{.7}{$[3,1,2]$}};
\node[vertex] (a4) at (1,-2){\scalebox{.7}{$[1,3,5]$}};
\node[vertex] (a5) at (0,-3){\scalebox{.7}{$[3,1,5]$}};
\node[vertex] (a6) at (2,-3){\scalebox{.7}{$[1,5,3]$}};
\node[vertex] (a7) at (0,-4){\scalebox{.7}{$[3,5,1]$}};
\node[vertex] (a8) at (2,-4){\scalebox{.7}{$[5,1,3]$}};
\node[vertex] (a9) at (-1,-5){\scalebox{.7}{$[3,5,6]$}};
\node[vertex] (a10) at (1,-5){\scalebox{.7}{$[5,3,1]$}};
\node[vertex] (a11) at (0,-6){\scalebox{.7}{$[5,3,6]$}};
\node[vertex] (a12) at (0,-7){\scalebox{.7}{$[5,6,3]$}};

\draw[->] (a1) edge node[left]{\scalebox{.5}{$2$}} (a2);
\draw[->] (a2) edge node[left]{\scalebox{.5}{$1$}} (a3);
\draw[->] (a2) edge node[right]{\scalebox{.5}{$3$}}(a4);
\draw[->] (a3) edge node[left]{\scalebox{.5}{$3$}} (a5);
\draw[->] (a5) edge node[left]{\scalebox{.5}{$2$}} (a7);
\draw[->] (a4) edge node[right]{\scalebox{.5}{$2$}} (a6);
\draw[->] (a6) edge node[right]{\scalebox{.5}{$1$}} (a8);

\draw[->] (a7) edge node[left]{\scalebox{.5}{$1$}} (a10);
\draw[->] (a8) edge node[left]{\scalebox{.5}{$1$}} (a10);
\draw[->] (a7) edge node[left]{\scalebox{.5}{$3$}} (a9);

\draw[->] (a9) edge node[right]{\scalebox{.5}{$1$}} (a11);
\draw[->] (a10) edge node[right]{\scalebox{.5}{$3$}} (a11);

\draw[->] (a4) edge node[right]{\scalebox{.5}{$1$}} (a5);

\draw[->] (a11) edge node[right]{\scalebox{.5}{$2$}} (a12);

\draw[->]  (a12) edge[bend right = 90] node[right,pos = .5]{} (a6);
\draw[->]  (a11) edge[bend right = 10] node[pos=.5,right,in = 90]{} (a4);
\draw[->]  (a10) edge[bend right = 5] node[pos=.5,right]{} (a2);
\draw[->]  (a8) edge[bend right = 90] node[pos=.5,right]{} (a1);
\end{tikzpicture}$$
\end{minipage}
\begin{minipage}{8cm}
$$
\begin{tikzpicture}[scale =.8]
\tikzstyle{vertex}=[inner sep=2pt,minimum size=10pt]
\node[vertex] (a1) at (0,0) {
\begin{tikzpicture}[scale =.7]
\draw (0,0) -- (1,0);
\node[fill = Blue,circle,inner sep=0pt,minimum size=.12cm] at (0,0) {};
\node[fill = Blue,circle,inner sep=0pt,minimum size=.12cm] at (.5,0) {};
\node[fill = gray,circle,inner sep=0pt,minimum size=.12cm] at (1,0) {};
\end{tikzpicture}};

\node[vertex] (a2) at (0,-1) {\begin{tikzpicture}[scale =.7]
\draw (0,0) -- (1,0);
\node[fill = Blue,circle,inner sep=0pt,minimum size=.12cm] at (0,0) {};
\node[fill = gray,circle,inner sep=0pt,minimum size=.12cm] at (.5,0) {};
\node[fill = Blue,circle,inner sep=0pt,minimum size=.12cm] at (1,0) {};
\end{tikzpicture}};

\node[vertex] (a3) at (-1,-2) {
\begin{tikzpicture}[scale =.7]
\draw (0,0) -- (1,0);
\node[fill = gray,circle,inner sep=0pt,minimum size=.12cm] at (0,0) {};
\node[fill = Blue,circle,inner sep=0pt,minimum size=.12cm] at (.5,0) {};
\node[fill = Blue,circle,inner sep=0pt,minimum size=.12cm] at (1,0) {};
\end{tikzpicture}};

\node[vertex] (a4) at (1,-2){
\begin{tikzpicture}[scale =.7]
\draw (0,0) -- (1,0);
\node[fill = Blue,circle,inner sep=0pt,minimum size=.12cm] at (0,0) {};
\node[fill = gray,circle,inner sep=0pt,minimum size=.12cm] at (.5,0) {};
\node[fill = Green,circle,inner sep=0pt,minimum size=.12cm] at (1,0) {};
\end{tikzpicture}};

\node[vertex] (a5) at (0,-3){
\begin{tikzpicture}[scale =.7]
\draw (0,0) -- (1,0);
\node[fill = gray,circle,inner sep=0pt,minimum size=.12cm] at (0,0) {};
\node[fill = Blue,circle,inner sep=0pt,minimum size=.12cm] at (.5,0) {};
\node[fill = Green,circle,inner sep=0pt,minimum size=.12cm] at (1,0) {};
\end{tikzpicture}};

\node[vertex] (a6) at (2,-3){
\begin{tikzpicture}[scale =.7]
\draw (0,0) -- (1,0);
\node[fill = Blue,circle,inner sep=0pt,minimum size=.12cm] at (0,0) {};
\node[fill = Green,circle,inner sep=0pt,minimum size=.12cm] at (.5,0) {};
\node[fill = gray,circle,inner sep=0pt,minimum size=.12cm] at (1,0) {};
\end{tikzpicture}};

\node[vertex] (a7) at (0,-4){
\begin{tikzpicture}[scale =.7]
\draw (0,0) -- (1,0);
\node[fill = gray,circle,inner sep=0pt,minimum size=.12cm] at (0,0) {};
\node[fill = Green,circle,inner sep=0pt,minimum size=.12cm] at (.5,0) {};
\node[fill = Blue,circle,inner sep=0pt,minimum size=.12cm] at (1,0) {};
\end{tikzpicture}};

\node[vertex] (a8) at (2,-4){
\begin{tikzpicture}[scale =.7]
\draw (0,0) -- (1,0);
\node[fill = Green,circle,inner sep=0pt,minimum size=.12cm] at (0,0) {};
\node[fill = Blue,circle,inner sep=0pt,minimum size=.12cm] at (.5,0) {};
\node[fill = gray,circle,inner sep=0pt,minimum size=.12cm] at (1,0) {};
\end{tikzpicture}};

\node[vertex] (a9) at (-1,-5){
\begin{tikzpicture}[scale =.7]
\draw (0,0) -- (1,0);
\node[fill = gray,circle,inner sep=0pt,minimum size=.12cm] at (0,0) {};
\node[fill = Green,circle,inner sep=0pt,minimum size=.12cm] at (.5,0) {};
\node[fill = Green,circle,inner sep=0pt,minimum size=.12cm] at (1,0) {};
\end{tikzpicture}};

\node[vertex] (a10) at (1,-5){\begin{tikzpicture}[scale =.7]
\draw (0,0) -- (1,0);
\node[fill = Green,circle,inner sep=0pt,minimum size=.12cm] at (0,0) {};
\node[fill = gray,circle,inner sep=0pt,minimum size=.12cm] at (.5,0) {};
\node[fill = Blue,circle,inner sep=0pt,minimum size=.12cm] at (1,0) {};
\end{tikzpicture}};

\node[vertex] (a11) at (0,-6){
\begin{tikzpicture}[scale =.7]
\draw (0,0) -- (1,0);
\node[fill = Green,circle,inner sep=0pt,minimum size=.12cm] at (0,0) {};
\node[fill = gray,circle,inner sep=0pt,minimum size=.12cm] at (.5,0) {};
\node[fill = Green,circle,inner sep=0pt,minimum size=.12cm] at (1,0) {};
\end{tikzpicture}};

\node[vertex] (a12) at (0,-7){
\begin{tikzpicture}[scale =.7]
\draw (0,0) -- (1,0);
\node[fill = Green,circle,inner sep=0pt,minimum size=.12cm] at (0,0) {};
\node[fill = Green,circle,inner sep=0pt,minimum size=.12cm] at (.5,0) {};
\node[fill = gray,circle,inner sep=0pt,minimum size=.12cm] at (1,0) {};
\end{tikzpicture}};

\draw[->] (a1) edge node[left]{\scalebox{.5}{$2$}} (a2);
\draw[->] (a2) edge node[left]{\scalebox{.5}{$1$}} (a3);
\draw[->] (a2) edge node[right]{\scalebox{.5}{$3$}}(a4);
\draw[->] (a3) edge node[left]{\scalebox{.5}{$3$}} (a5);
\draw[->] (a5) edge node[left]{\scalebox{.5}{$2$}} (a7);
\draw[->] (a4) edge node[right]{\scalebox{.5}{$2$}} (a6);
\draw[->] (a6) edge node[right]{\scalebox{.5}{$1$}} (a8);

\draw[->] (a7) edge node[left]{\scalebox{.5}{$1$}} (a10);
\draw[->] (a8) edge node[left]{\scalebox{.5}{$1$}} (a10);
\draw[->] (a7) edge node[left]{\scalebox{.5}{$3$}} (a9);

\draw[->] (a9) edge node[right]{\scalebox{.5}{$1$}} (a11);
\draw[->] (a10) edge node[right]{\scalebox{.5}{$3$}} (a11);

\draw[->] (a4) edge node[right]{\scalebox{.5}{$1$}} (a5);

\draw[->] (a11) edge node[right]{\scalebox{.5}{$2$}} (a12);

\draw[->]  (a12) edge[bend right = 60] node[right,pos = .5]{\scalebox{.6}{}} (a6);
\draw[->]  (a11) edge[bend right = 10] node[pos=.5,right,in = 90]{\scalebox{.6}{}} (a4);
\draw[->]  (a10) edge[bend right = 7] node[pos=.5,right]{\scalebox{.6}{}} (a2);
\draw[->]  (a8) edge[bend right = 90] node[pos=.5,right]{\scalebox{.6}{}} (a1);
\end{tikzpicture}$$
\end{minipage}
$$
\vspace{-.5cm}
\caption{The graph $\Ga_{\om_2}$ and $\tilde{\Ga}_{\om_2}$ in type $\tilde{C}_3$}
\end{figure}
\end{example}

\begin{example}
Let $W_a$ be the affine Weyl group $\tilde{C}_3$ and let $\ga=\om_3$. We have $J=\{1,2\}$ and $J'=\{3\}$. The group $W_J$ acts on $\{1,2,3,4,5,6\}$  and the orbits are $B_1=\{1,2,3\}$ and $B_2=\{4,5,6\}$. The elements of $W^J$ are 
 \begin{align*}
&e = [1,2,3],\ 
s_3 = [1,2,4],\ 
s_2s_3 = [1,4,2],\ 
s_1s_2s_3 = [4,1,2],\ 
s_3s_2s_3 = [1,4,5],\\
&s_1s_3s_2s_3=[4,1,5],\ 
s_2s_1s_3s_2s_3=[4,5,1],\ 
s_3s_2s_1s_3s_2s_3=[4,5,6].
 \end{align*}
 We have $s_\theta = s_1s_2s_3s_2s_1$ and in the list above, the set of elements that satisfy $s_\theta w<w$ are exactly those whose window notation starts by $4$. As before, the weight of each $0$ arrow will be $\sz^{\ov{-\e_3}}$ and will not be indicated on the graphs below. The associated color map is constant defined by $\sc(1) = \sc(2) = \sc(3)=1$. Since $3\notin J$, the particle $1$ does have spin. In the graph $\tilde{\Ga}_{\om_2}$ below,  $1^-$ is represented in blue and $1^+$ in green. Again, we do not indicate the weights $a_i\in \N$ but we do indicate the $i$ when $1\leq i\leq n$.  

\vspace{-.3cm}
 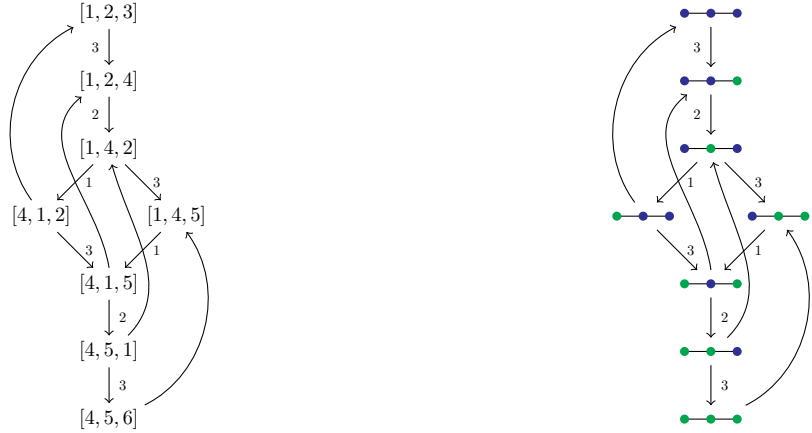
\begin{figure}[H]
 $$
\begin{minipage}{8cm}
$$\begin{tikzpicture}[scale =.9]
\tikzstyle{vertex}=[inner sep=2pt,minimum size=10pt]

\node[vertex] (a1) at (0,0) {\scalebox{.7}{$[1,2,3]$}};
\node[vertex] (a2) at (0,-1) {\scalebox{.7}{$[1,2,4]$}};
\node[vertex] (a3) at (0,-2) {\scalebox{.7}{$[1,4,2]$}};
\node[vertex] (a4) at (-1,-3){\scalebox{.7}{$[4,1,2]$}};
\node[vertex] (a5) at (1,-3){\scalebox{.7}{$[1,4,5]$}};
\node[vertex] (a6) at (0,-4){\scalebox{.7}{$[4,1,5]$}};
\node[vertex] (a7) at (0,-5){\scalebox{.7}{$[4,5,1]$}};
\node[vertex] (a8) at (0,-6){\scalebox{.7}{$[4,5,6]$}};

\draw[->] (a1) edge node[left]{\scalebox{.5}{$3$}} (a2);
\draw[->] (a2) edge node[left]{\scalebox{.5}{$2$}} (a3);

\draw[->] (a3) edge node[right]{\scalebox{.5}{$1$}}(a4);
\draw[->] (a3) edge node[right]{\scalebox{.5}{$3$}}(a5);

\draw[->] (a4) edge node[right]{\scalebox{.5}{$3$}}(a6);
\draw[->] (a5) edge node[right]{\scalebox{.5}{$1$}}(a6);

\draw[->] (a6) edge node[right]{\scalebox{.5}{$2$}}(a7);
\draw[->] (a7) edge node[right]{\scalebox{.5}{$3$}}(a8);

\draw[->] (a8) edge[bend right = 50] node[right,pos=.5]{} (a5);
\draw[->] (a7) edge[bend right = 50,in=190] node[right,pos=.2]{} (a3);

\draw[->] (a6) edge[bend left = 50,in=120,out = 0] node[left,pos=.7]{} (a2);
\draw[->] (a4) edge[bend left = 50] node[left,pos=.5]{} (a1);
\end{tikzpicture}$$
\end{minipage}
\begin{minipage}{8cm}
$$
\begin{tikzpicture}[scale =.9]
\tikzstyle{vertex}=[inner sep=2pt,minimum size=10pt]

\node[vertex] (a1) at (0,0) {
\begin{tikzpicture}[scale =.7]
\draw (0,0) -- (1,0);
\node[fill = Blue,circle,inner sep=0pt,minimum size=.12cm] at (0,0) {};
\node[fill = Blue,circle,inner sep=0pt,minimum size=.12cm] at (.5,0) {};
\node[fill = Blue,circle,inner sep=0pt,minimum size=.12cm] at (1,0) {};
\end{tikzpicture}};

\node[vertex] (a2) at (0,-1) {
\begin{tikzpicture}[scale =.7]
\draw (0,0) -- (1,0);
\node[fill = Blue,circle,inner sep=0pt,minimum size=.12cm] at (0,0) {};
\node[fill = Blue,circle,inner sep=0pt,minimum size=.12cm] at (.5,0) {};
\node[fill = Green,circle,inner sep=0pt,minimum size=.12cm] at (1,0) {};
\end{tikzpicture}};

\node[vertex] (a3) at (0,-2) {
\begin{tikzpicture}[scale =.7]
\draw (0,0) -- (1,0);
\node[fill = Blue,circle,inner sep=0pt,minimum size=.12cm] at (0,0) {};
\node[fill = Green,circle,inner sep=0pt,minimum size=.12cm] at (.5,0) {};
\node[fill = Blue,circle,inner sep=0pt,minimum size=.12cm] at (1,0) {};
\end{tikzpicture}};

\node[vertex] (a4) at (-1,-3){
\begin{tikzpicture}[scale =.7]
\draw (0,0) -- (1,0);
\node[fill = Green,circle,inner sep=0pt,minimum size=.12cm] at (0,0) {};
\node[fill = Blue,circle,inner sep=0pt,minimum size=.12cm] at (.5,0) {};
\node[fill = Blue,circle,inner sep=0pt,minimum size=.12cm] at (1,0) {};
\end{tikzpicture}};

\node[vertex] (a5) at (1,-3){
\begin{tikzpicture}[scale =.7]
\draw (0,0) -- (1,0);
\node[fill = Blue,circle,inner sep=0pt,minimum size=.12cm] at (0,0) {};
\node[fill = Green,circle,inner sep=0pt,minimum size=.12cm] at (.5,0) {};
\node[fill = Green,circle,inner sep=0pt,minimum size=.12cm] at (1,0) {};
\end{tikzpicture}};

\node[vertex] (a6) at (0,-4){
\begin{tikzpicture}[scale =.7]
\draw (0,0) -- (1,0);
\node[fill = Green,circle,inner sep=0pt,minimum size=.12cm] at (0,0) {};
\node[fill = Blue,circle,inner sep=0pt,minimum size=.12cm] at (.5,0) {};
\node[fill = Green,circle,inner sep=0pt,minimum size=.12cm] at (1,0) {};
\end{tikzpicture}};

\node[vertex] (a7) at (0,-5){\begin{tikzpicture}[scale =.7]
\draw (0,0) -- (1,0);
\node[fill = Green,circle,inner sep=0pt,minimum size=.12cm] at (0,0) {};
\node[fill = Green,circle,inner sep=0pt,minimum size=.12cm] at (.5,0) {};
\node[fill = Blue,circle,inner sep=0pt,minimum size=.12cm] at (1,0) {};
\end{tikzpicture}};

\node[vertex] (a8) at (0,-6){
\begin{tikzpicture}[scale =.7]
\draw (0,0) -- (1,0);
\node[fill = Green,circle,inner sep=0pt,minimum size=.12cm] at (0,0) {};
\node[fill = Green,circle,inner sep=0pt,minimum size=.12cm] at (.5,0) {};
\node[fill = Green,circle,inner sep=0pt,minimum size=.12cm] at (1,0) {};
\end{tikzpicture}};

\draw[->] (a1) edge node[left]{\scalebox{.5}{$3$}} (a2);
\draw[->] (a2) edge node[left]{\scalebox{.5}{$2$}} (a3);

\draw[->] (a3) edge node[right]{\scalebox{.5}{$1$}}(a4);
\draw[->] (a3) edge node[right]{\scalebox{.5}{$3$}}(a5);

\draw[->] (a4) edge node[right]{\scalebox{.5}{$3$}}(a6);
\draw[->] (a5) edge node[right]{\scalebox{.5}{$1$}}(a6);

\draw[->] (a6) edge node[right]{\scalebox{.5}{$2$}}(a7);
\draw[->] (a7) edge node[right]{\scalebox{.5}{$3$}}(a8);

\draw[->] (a8) edge[bend right = 50] node[right]{} (a5);
\draw[->] (a7) edge[bend right = 50,in=190] node[right]{} (a3);

\draw[->] (a6) edge[bend left = 50,in=120,out = 0] node[right]{} (a2);
\draw[->] (a4) edge[bend left = 50] node[right]{} (a1);

\end{tikzpicture}$$
\end{minipage}
$$
\caption{The graph $\Ga_{\om_3}$ and $\tilde{\Ga}_{\om_3}$ in type $\tilde{C}_3$}
\end{figure}
\end{example}
There are similar interpretations of the graphs $\Ga_{\ga}$ in type $\tilde{B}_n$ and $\tilde{D}_n$. 


\end{document}